\documentclass[12pt]{amsart}

\usepackage{graphicx}
\usepackage{amssymb, comment, color}
\usepackage[toc,page]{appendix}
\usepackage{epsf,overpic,amssymb,amscd,times,euscript}
\usepackage[utf8]{inputenc}
\usepackage[english]{babel}
\usepackage{amsmath}
\usepackage{amsthm}
\usepackage{amsfonts}
\usepackage{amssymb}
\usepackage{amscd}
\usepackage{bezier}
\usepackage{enumerate}
\usepackage{mathtools}

\usepackage{graphicx}
\usepackage{hyperref}
\usepackage{upgreek}

\usepackage{parskip}

\usepackage{bbm}

\newcommand{\id}{id}

\newcommand{\N}{\mathbbm{N}}                     
\newcommand{\Z}{\mathbbm{Z}}                     
\newcommand{\Q}{\mathbbm{Q}}                     
\newcommand{\R}{\mathbbm{R}}                     

\newcommand{\D}{\mathbb{D}}

\newcommand{\Hy}{\mathbbm{H}}
\newcommand{\Cyl}{\mathrm{Cyl}}

\newcommand{\supp}{\mathrm{supp\,}}             
\newcommand{\Ham}{\mathrm{Ham}}    

\newcommand{\topo}{\mathrm{top}}

\newcommand{\hofer}{\mathrm{Hofer}}

\newtheorem{mainthm}{\sc Theorem}
\newtheorem{mainthmainner}{\sc Theorem}
\newenvironment{mainthma}[1]{%
  \renewcommand\themainthmainner{#1}%
  \mainthmainner
}{\endmainthmainner}
\newtheorem{thm}{Theorem}[section]               
\newtheorem*{thm*}{Theorem}               
\newtheorem{cor}[thm]{Corollary}        
\newtheorem*{cor*}{Corollary}        
\newtheorem{lem}[thm]{Lemma}  
\newtheorem*{lem*}{Lemma}
\newtheorem{prop}[thm]{Proposition}     
\theoremstyle{definition}
\newtheorem{defn}[thm]{Definition}      
\newtheorem{fact}{Fact}[section]
\newtheorem{rem}[thm]{Remark}           
\newtheorem{question}{Question}
 \newtheorem*{acknowledgement*}{\protect\acknowledgementname}
\newcounter{claim}
\newenvironment{claim}[1][]{\refstepcounter{claim}\par\noindent\underline{Claim~\theclaim:}\space#1}{}

 \providecommand{\acknowledgementname}{Acknowledgement}

\author{Marcelo R.R. Alves}
\thanks{M.R.R. Alves was supported by the ERC consolidator grant 646649  ``SymplecticEinstein'' and by the Senior Postdoctoral fellowship of the Research Foundation - Flanders (FWO) in fundamental research 1286921N}
\address{Marcelo R.R. Alves, Faculty of Science,\\
University of Antwerp,
 Campus Middelheim,
 Middelheimlaan 1,
BE-2020 Antwerpen,
Belgium.}
\email{\texttt{marcelorralves@gmail.com}}

\author{Matthias Meiwes}
\thanks{M. Meiwes was supported by the Chair for Geometry and Analysis, RWTH Aachen}
\address{Matthias Meiwes,
Chair for Geometry and Analysis, RWTH Aachen University, Jakobstrasse 2,
DE-52064 Aachen, Germany.}
\email{\texttt{meiwes@mathga.rwth-aachen.de}}

\title[Braid stability and the Hofer metric]{Braid stability and the Hofer metric}
\begin{document}

\begin{abstract}
In this article we show that the braid type of a set of $1$-periodic orbits of a non-degenerate Hamiltonian diffeomorphism on a surface is stable under perturbations which are sufficiently small with respect to the Hofer metric $d_{\rm Hofer}$.
We call this new phenomenon braid stability for the Hofer metric. 

We apply braid stability to study the stability of the topological entropy $h_{\rm top}$ of Hamiltonian diffeomorphisms on surfaces with respect to small perturbations with respect to $d_{\rm Hofer}$. We show that $h_{\rm top}$ is lower semicontinuous on the space of Hamiltonian diffeomorphisms of a closed surface endowed with the Hofer metric, and on the space of compactly supported diffeormophisms of the two-dimensional disk $\D$ endowed with the Hofer metric. This answers the two-dimensional case of a question of Polterovich.

En route to proving the lower semicontinuity of $h_{\rm top}$ with respect to $d_{\rm Hofer}$, we prove that the topological entropy of a diffeomorphism $\phi$ on a compact surface can be recovered from the topological entropy of the braid types realised by the periodic orbits of $\phi$.
\end{abstract}

\maketitle
\tableofcontents

\section{Introduction}

The objective of this article is to study a new type of dynamical stability with respect to the Hofer metric on the space of Hamiltonian diffeomorphisms of a surface. In this section we present the necessary background and context for our results, and give a mostly intuitive and geometric discussion of the results. The precise statements of our main theorems are presented in section \ref{sec:mainresults}.

In order to explain our results we must first recall some notions.

\subsection{Preliminary notions}
Let $\Sigma$ be a closed surface and $\omega$ a symplectic form on $\Sigma$. If $\Sigma = S^2$ then we assume, for reasons that will become clear later, that $\int_{S^2} \omega = 8$.  The symplectomorphisms of $(\Sigma,\omega)$ are the diffeomorphisms of $\Sigma$ which preserve the symplectic form $\omega$.  A time-dependent Hamiltonian $H:S^1 \times \Sigma \to \R$ gives rise to a time-dependent vector field $X_H$ on $\Sigma$, called the Hamiltonian vector-field of $H$, given by the formula 
\begin{equation}
  \iota_{X_{H(t,\cdot)}} \omega = d_\Sigma H(t,\cdot),
\end{equation}
 where $d_\Sigma H(t,\cdot)$ is the differential of $H(t, \cdot) : \Sigma \to \R$ in $\Sigma$. The flow $\phi^t_H$ of $X_H$ is called the Hamiltonian flow of $H$.

A Hamiltonian diffeomorphism $\phi$ is a diffeomorphism of $\Sigma$ which is the time $1$-map of the Hamiltonian flow of some Hamiltonian in $H$. 
We denote the group of Hamiltonian diffeomorphisms of $(\Sigma, \omega)$ by $\mathrm{Ham}(\Sigma,\omega)$. Hamiltonian diffeomorphisms are symplectomorphisms, and we refer the reader to \cite{leonidsbook} for a proof that $\mathrm{Ham}(\Sigma,\omega)$ is indeed a group with respect to the composition of diffeomorphisms. 

For simplicity we introduce some terminology. If $\phi$ is a Hamiltonian diffeomorphism and $H$ is a Hamiltonian such that $\phi$ is the time $1$-map of $\phi_H$, we say that $H$ generates $\phi$.

Recall that a Hamiltonian $H:S^1 \times \Sigma \to \R$ is called normalized if $\int_\Sigma H_t\omega=0$ for each $t\in S^1$, where $H_t(\cdot):= H(t,\cdot)$. Given $\phi\in \mathrm{Ham}(\Sigma,\omega)$ it is always possible to find a normalized Hamiltonian $H$ which generates $\phi$.


From now on all time-dependent Hamiltonian functions on closed surfaces considered in this paper are assumed to be normalized.
 
Recall that the Hamiltonian action $\mathcal{A}_{H}(y, w_y)$ of a pair $(y,w_y)$ of a contractible loop $y:S^1 \to \Sigma$ and some disk capping $w_y: \overline{\D} \to \Sigma$ of $y$ is defined as
\begin{equation} \label{eq:defaction}
    \mathcal{A}_{H}(y,w_y) := -\int_{\overline \D} (w_y)^*\omega + \int_0^1 H(t,y(t)) dt. 
\end{equation}
 If $\Sigma$ is distinct from $S^2$ the action $\mathcal{A}_{H}(y) = \mathcal{A}_{H}(y,w_y)$ does not depend on the choice of $w_y$ but only on $y$. The case of $(S^2,\omega)$ will be considered in section \ref{sec:Hamonsphere}. 

If $\Sigma \neq S^2$, we fix for each free homotopy class of loops $\alpha\in [S^1, M]$ a representative $\eta_{\alpha}:S^1 \to \Sigma$ of $\alpha$. 
If $y$ is a smooth non-contractible loop and $\alpha$ its free homotopy class, we say that a smooth mapping $w_y:[0,1] \times S^1 \to \Sigma$ with $w_y(0,t) = \eta_{\alpha}(t)$ and $w_y(1,t) = y(t)$ is a cylindrical capping of $y$. 
The Hamiltonian action $\mathcal{A}_H(y, w_y)$ of a pair $(y,w_y)$ of a non-contractible loop $y: S^1 \to \Sigma$ and a cylindrical capping $w_y$ of $y$ is defined as 
\begin{equation} \label{eq:defaction2}
    \mathcal{A}_{H}(y,w_y) := -\int_{\overline \D} (w_y)^*\omega + \int_0^1 H(t,y(t)) dt. \end{equation}
     If $\Sigma \neq T^2$, then the action does not depend on the choice of $w_y$, whereas if $\Sigma = T^2$, the choice of cylinder has to be considered as well, see section \ref{sec:Hamontorus}.

We may define for any two freely homotopic loops $y$ and $y'$ in $\Sigma$ the quantity 
\begin{equation}\label{def_Delta}
\Delta_H(y,y') := \inf_{w_y, w_{y'}}\left|\mathcal{A}_H(y,w_y) - \mathcal{A}_H(y',w_{y'})\right|, 
\end{equation}
where the infimum is taken over all cappings $w_y$ and $w_{y'}$ of $y$ and $y'$; disk cappings if $y,y'$ are contractible, and cylindrical capping if $y,y'$ are non-contractible. 

The Hofer metric $d_{\rm Hofer}$ is a Finsler metric on $\mathrm{Ham}(\Sigma,\omega)$ which is bi-invariant with respect to the group structure of $\mathrm{Ham}(\Sigma,\omega)$. If $\phi_1$ and $\phi_2$ are elements of $\mathrm{Ham}(\Sigma,\omega)$ we define
 \begin{equation}
     d_{\rm Hofer}(\phi_1,\phi_2) =  \inf_{H \in  \mathcal{I}(\phi_1, \phi_2)} \int_{0}^1 || H_t ||dt,
 \end{equation}
 where $H:S^1 \times \Sigma \to \R$ is in $  \mathcal{I}(\phi_1, \phi_2)$ if it is normalized and generates $\phi_1^{-1}\circ \phi_2$, and where  $|| H_t ||:= \max_{p \in \Sigma} H_t  - \min_{p \in \Sigma} H_t$. It is a highly non-trivial fact that $d_{\rm Hofer}$ is a non-degenerate metric, see \cite{Hofer1990}. 
 
 We also consider in this paper Hamiltonian diffeomorphisms of the disk $\D$ endowed with the symplectic structure $\omega_0 = dx\wedge dy$. For this situation we consider for $c\in \R \setminus 2\pi \Q$ the set $\mathrm{Ham}_c(\D)$ of Hamiltonian diffeomorphisms which coincide with the irrational rotation by angle $c$ in a neighbourhood of $\partial \D$, or the set $\mathrm{Ham}_0(\D)$ of compactly supported Hamiltonian diffeomorphisms on $(\D,\omega_0)$. The set $\mathrm{Ham}_0(\D)$ is a group, and  although $\mathrm{Ham}_c(\D)$ is not a group, we can still define the Hofer metric on it and study its dynamical significance. We refer the reader to \ref{sec:Hamsondisk} for the precise definitions we adopt in this setting.

Lastly, if $(\Sigma,\omega)$ is a closed surface and $\phi$ is a Hamiltonian diffeomorphism, we will say that $\phi$ is non-degenerate if all $1$-periodic orbits of $\phi$ are non-degenerate. We say that $\phi$ is strongly non-degenerate if for every $n>0$ all $n$-periodic orbits of $\phi$ are non-degenerate. The same definition works for Hamiltonian diffeomorphisms in $\mathrm{Ham}_c(\D)$ where $c\in \R \setminus 2\pi \Q$.
For $\mathrm{Ham}_0(\D)$ we must adapt the definition, since every Hamiltonian diffeomorphism in $\mathrm{Ham}_0(\D)$ has degenerate $1$-periodic orbits; see \ref{sec:Hamsondisk}.
 
 
 \subsection{Hofer metric and dynamics}
 
 The properties of the Hofer metric are the subject of intense research in the field of symplectic topology.
One direction of investigation to better understand $d_{\rm Hofer}$ is to investigate how dynamical properties of Hamiltonian diffeomorphisms vary under perturbations with respect to $d_{\rm Hofer}$. This has been pursued by several authors, see \cite{chor2020,kislev2018bounds,leonidpersistence,PolterovichShelukhin2016,Usher}. Floer homology together with its action filtration behaves in a very stable way under perturbations with respect to $d_{\rm Hofer}$.
This culminated in the dynamical stability result of Polterovich and Shelukhin in \cite{PolterovichShelukhin2016}, which says that in a path of Hamiltonian diffeomorphisms that is continuous with respect to $d_{\rm Hofer}$ the spectrum of the diffeomorphisms changes continuously in a certain sense; see also \cite{leonidpersistence,Usher}. 
With the introduction of the theory of persistence modules to Floer theory in \cite{PolterovichShelukhin2016}, the authors obtained moreover good quantitative estimates on the size of perturbations with respect to $d_{\rm Hofer}$ under which spectral stability holds. 
More recently, the second author and Chor \cite{chor-meiwes} have studied stability properties of the topological entropy $h_{\rm top}$ under perturbations with respect to $d_{\rm Hofer}$. In this paper, we also investigate this question but with techniques and in a setting which are different from that of \cite{chor-meiwes}.

For the discussion that follows we assume that $\Sigma$ is either the disk or a closed surface different from $S^2$, so that we can define the action of contractible $1$-periodic orbits of Hamiltonians on $(\Sigma,\omega)$.
Let $p$ be a periodic point of a Hamiltonian diffeomorphism $\phi$ of (not necessarily minimal) period $n>0$. Assume that $H$ is a normalized Hamiltonian that generates $\phi$ and that the $n$-periodic orbit of the Hamiltonian flow of $H$ which starts at $p$ is contractible. 
One topological quantity that we can associate to the pair $(p,n)$ is its action with respect to the Hamiltonian $H$. To define it, we denote by $\gamma$ the $n$-periodic orbit of the flow $\phi_H$ which satisfies $\gamma(0)=p$. The action of $(p,n)$ is defined to be $\mathcal{A}_{nH}(\gamma)$.
It was proved by Schwarz in \cite{schwarz-action} that $\mathcal{A}_H(\gamma)$ does not depend on the choice of the normalized Hamiltonian $H$ generating $\phi$, so that it is indeed a topological property of the pair $(p,n)$.

As observed by Polterovich-Shelukhin and Usher (see \cite{PolterovichShelukhin2016,Usher}), if $p$ is a periodic point of (not necessarily minimal) period $n$ of a non-degenerate Hamiltonian diffeomorphism $\phi \in \mathrm{Ham}(\Sigma,\omega)$, then given $\varepsilon>0$ there exists an open neighbourhood $\mathcal{U}_\varepsilon$ of $\phi$ in $\mathrm{Ham}(\Sigma,\omega)$  endowed with the Hofer geometry, such that every element $\phi'$ in $\mathcal{U}_\varepsilon$ has a periodic point $p'$ of period $n$ whose action is $\varepsilon$-close to the action of $(p,n)$. One can think of this as saying that the periodic point $p$ of $\phi$ gives rise to periodic points of the same (not necessarily minimal) period and similar action for Hamiltonian diffeomorphisms that are sufficiently close to $\phi$. A natural question is the following: do these periodic orbits, which arise from $p$, inherit other topological properties of $p$?
The main results  of the present paper show that under certain conditions the answer to this question is yes.

To explain why this question is non-trivial we notice that small perturbations in the sense of the Hofer metric are not necessarily small in the $C^0$-sense. So, no matter how close a Hamiltonian diffeomorphism $\phi'$ is to $\phi$ in the sense of $d_{\rm Hofer}$, one cannot guarantee that the periodic points of $\phi'$ which arise from the periodic point $p$ are close to periodic point $p$.

\subsection{Braid type of a collection of $1$-periodic orbits}

To present our braid stability results we recall how one associates to a collection $\mathcal{P}=\{p_1,...,p_k\}$ of fixed points of a Hamiltonian diffeomorphism $\phi$ on $(\Sigma,\omega)$ its braid type. 

For this we first consider a Hamiltonian $H:S^1 \times \Sigma \to \R$
and let $\mathcal{Y}$ be a finite collection of distinct $1$-periodic orbits of $X_H$.  We associate to  $\mathcal{Y}$ a braid $\mathcal{B}(\mathcal{Y})$ in $S^1 \times \Sigma$. 
\begin{defn} \label{def:braid}
The braid $\mathcal{B}(\mathcal{Y})$ in $S^1 \times \Sigma$ associated to a finite set $\mathcal{Y}:= \{\gamma_1,\gamma_2,...,\gamma_k\}$ of distinct $1$-periodic orbits of $X_H$ is defined as
\begin{equation} \label{eq:braid}
    \mathcal{B}(\mathcal{Y}) :=  \bigcup_{i=1}^k \xi_i,
\end{equation}
where $\xi_i$ are given by
\begin{equation} \label{eq:braidcomponents}
  \xi_i := \{(t,\gamma_i(t)) \ | \ t \in S^1\}.
\end{equation}
We call $\mathcal{B}(\mathcal{Y})$ the braid associated to the set $\mathcal{Y}$ of $1$-periodic orbits of $H$.
\end{defn}

\begin{rem}
Notice that $\mathcal{B}(\mathcal{Y})$ is a link and not braid, but we nevertheless abuse notation and call it braid. There are two reasons for doing this. The first is that it is easier to define the braid type of a set of $1$-periodic orbits using this link rather than the braid in $[0,1] \times \Sigma$ which is obtained from $\mathcal{B}(\mathcal{Y})$ by cutting $S^1 \times \Sigma$ along the surface $\{0\} \times \Sigma$. The second is that in our arguments it is always the link $\mathcal{B}(\mathcal{Y})$ which will appear.

From now on, a braid on $\Sigma$ will always mean a smooth link in $S^1 \times \Sigma$ which is transverse to all the surfaces $\{t\} \times \Sigma$.
\end{rem}

\begin{rem}
Although in the present paper we will only be considering the situation where $\mathcal{P}$ is a collection of fixed points of a Hamiltonian diffeomorphism $\phi$, it is straightforward to see that one can also associate a braid in $S^1 \times \Sigma$ to a collection $\mathcal{Q}$ of periodic points of $\phi$ such that $\phi(\mathcal{Q})= \mathcal{Q}$.
\end{rem}



{  
For a Hamiltonian diffeomorphism $\phi$  and a choice of Hamiltonian $H$ generating $\phi$, one can use the procedure given in Definition \ref{def:braid} to associate to a collection $\mathcal{P}=\{p_1,...,p_k\}$ of fixed points, a braid $\mathcal{B}({\mathcal{Y}}_{\mathcal{P}})$.

Notice that there is a certain ambiguity in the construction of $\mathcal{B}({\mathcal{Y}}_{\mathcal{P}})$ since it depends on the choice of $H$. If we perform this construction with another Hamiltonian $H'$ which generates $\phi$ we obtain a different braid $\mathcal{B}({\mathcal{Y'}}_{\mathcal{P}})$: in case $\Sigma:= \D$ or $\Sigma$ has genus $\geq 2$ the braids $\mathcal{B}({\mathcal{Y}}_{\mathcal{P}})$ and $\mathcal{B}({\mathcal{Y'}}_{\mathcal{P}})$ are conjugated to $\overline{\mathcal{Y}}_{\mathcal{P}}$, but in general the relationship between the two braids is more complicated. Since there is no preferred choice of Hamiltonian generating $\phi$ the natural object to be associated to $\mathcal{P}$ is not a braid but an equivalence class of braids, which was introduced by Boyland and is called the braid type of $\mathcal{P}$.}

Before presenting the definition of braid types, we explain a geometric condition that implies that two braids have the same braid type. For this we need the following terminology. If $\xi$ and $\xi'$ are smooth links in $ S^1 \times \Sigma$ which are transverse to the surfaces  $\{t\} \times \Sigma \subset S^1 \times \Sigma$, and are isotopic among links transverse these surfaces, we say that \textit{ $\xi$ and $\xi'$ are freely isotopic as braids}. 

\begin{fact} \label{fact:braidtype1}
Two braids in $S^1 \times \Sigma$ which are freely isotopic as braids have the same braid type. 
\end{fact}

Using Fact \ref{fact:braidtype1} we prove below that if  ${\mathcal{P}}$ is a collection of fixed points of a non-degenerate Hamiltonian diffeomorphism $\phi$, then there is a collection of fixed points of the Hofer-close Hamiltonian diffeomorphisms $\phi'$  that arise from ${\mathcal{P}}$ and have the same braid type as ${\mathcal{P}}$. 
The reason why such information is useful is that, many dynamical invariants of braids such as the topological entropy of a braid, are actually invariants of the braid type. To explain why this is the case we need the dynamical definition of braid type which is explained in the next section.


\subsection{Braids and surface dynamics} \label{sec:braids}

In the following we shortly recall relevant standard notions of surface dynamics, see for example \cite{Birman, Boyland, FranksMisiurewicz2002} for more details. 
Let $\Sigma$ be a compact oriented surface (with or without boundary), and let, for each $m\in \N$,  $X_m \subset \Sigma\setminus \partial \Sigma$ be a set of $m$ points in $\Sigma$. We denote the mapping class group on the $X_m$-punctured surface  $\Sigma$ by $\mathcal{M}(\Sigma, X_m)$. It is defined as the group of isotopy classes of orientation preserving homeomorphisms on $\Sigma$ that preserve $X_m$ setwise, and such that, if $\partial \Sigma \neq 0$, the allowed isotopies fix each boundary component setwise. We will denote by $[\phi]$ the element in $\mathcal{M}(\Sigma, X_m)$ that $\phi:\Sigma \to \Sigma$ represents. We are ready to give the definition of the braid type of a braid. 

\begin{defn} \label{def:braidtype2}
Let $f:\Sigma \to \Sigma$  be a homeomorphism that is isotopic to the identity, and let $\mathcal{Q} = \{q_0, \ldots , q_{m-1}\}$ be a set of periodic points of $f$ such that $f(\mathcal{Q})= \mathcal{Q}$.
The \textit{braid type $[f,\mathcal{Q}]$} of the pair $(f,\mathcal{Q})$ is  the conjugacy class in $\mathcal{M}(\Sigma,X_m)$ of elements $[h \circ f \circ h^{-1}]\in \mathcal{M}(\Sigma, X_m)$, where $h:\Sigma \to \Sigma$ is a homeomorphism (preserving the boundary components setwise) such that $h({\mathcal{Q}}) = X_m$. 
If $\mathcal{Q} = (q_0, \ldots, q_m)$ is a periodic orbit and $\overline{\mathcal{Q}} = \{q_0, \ldots, q_m\}$ the set associated to it, we also will write $[f,\mathcal{Q}]$ instead of  $[f,\overline{\mathcal{Q}}]$. 
\end{defn}

{
We now sketch the proof of Fact \ref{fact:braidtype1}. We first explain how a braid $\mathcal{B}$ of $m$-strands in $S^1 \times \Sigma$ defines a braid type $[f,\mathcal{Q}]$. Informally this is seen as follows: One first cuts $S^1 \times \Sigma$ along $\{0\} \times \Sigma$ to obtain from $\mathcal{B}$ a proper braid in $[0,1] \times \Sigma$.
One creates an isotopy of homeomorphisms starting at the identity by sliding along the braid, such that $\mathcal{B}= \mathcal{B}(\mathcal{Q})$ is the braid associated to a collection of periodic orbits $\mathcal{Q}$ of the homeomorphism $f$ that is created. We  say that $\mathcal{B}$ represents $[f,\mathcal{Q}]$. This construction is part of the proof of the Birman exact sequence where it is also shown that different representatives of an element of the braid group induce the same element of the mapping class group; see chapter 4 of \cite{Birman} or section 2 of \cite{Matsuoka}. 
Finally, one shows that braids $\mathcal{B},\mathcal{B'} \subset S^1 \times \Sigma$ are freely isotopic braids if, and only if, they give rise to conjugated pairs $(f,\mathcal{Q})$ and $(f',\mathcal{Q}')$ which therefore represent the same braid type in the sense of Definition \ref{def:braidtype2}. For details, we refer the reader to section 2 of \cite{Matsuoka} and section 4 of \cite{Boyland}.}

A braid type invariant of $[f,\mathcal{Q}]$ measuring its complexity is the growth rate of the induced action of $f$ on the fundamental group of $\Sigma \setminus {\mathcal{Q}}$. It is defined as
\begin{align}\label{fundgrowth}
\Gamma_{\pi_1}([f,\mathcal{Q}]) := \sup_{g \in \pi_1(\Sigma\setminus {{\mathcal{Q}}},x_0)} \limsup_{n\to \infty} \frac{\log\left(l_S(f^n_*(g))\right)}{n},
\end{align}
where $f_*$ is the automorphism on $\pi_1(\Sigma\setminus {{\mathcal{Q}}},x_0)$ with respect to some basepoint $x_0$ and a path $\sigma$
from $x_0$ to $f(x_0)$, $S$ is a set of generators, and $l_S(h)$ is the minimal length of a word in $S$ and $S^{-1}$ that is needed to represent $h$. The right hand side of \eqref{fundgrowth} is, for fixed $f$ and $\mathcal{Q}$, independent of all the choices made, and moreover invariant under conjugation and isotopy, see e.g. \cite{Bowen}. Hence $\Gamma_{\pi_1}([f,\mathcal{Q}])$ is  well-defined. 
Furthermore it follows from elementary properties of $f_*$, that for all $k\in \N$
\begin{align} 
\Gamma_{\pi_1}([f^k,\mathcal{Q}]) = k\Gamma_{\pi_1}([f,\mathcal{Q}]). 
\end{align}

By an inequality of Manning (see \cite{Bowen} for a proof in the present setting of a punctured surface and $f$ differentiable),
\begin{align}\label{fund_ent}
\Gamma_{\pi_1}([f,\mathcal{Q}]) \leq h_{\topo}(f).
\end{align} 
In fact,  $h_{\topo}([f,\mathcal{Q}]):= \inf_g h_{\topo}(g)$, the infimum of $h_{\topo}(g)$ over all $g$  with $[g,\mathcal{Q}'] = [f,\mathcal{Q}]$ for some $\mathcal{Q}'$,  is realized by the maximal topological entropy of a pseudo-Anosov component of a map in the Thurston-Nielsen canonical form, and moreover   $\Gamma_{\pi_1}([f,\mathcal{Q}]) = h_{\topo}([f,\mathcal{Q}])$. 
We will not use this fact, while \eqref{fund_ent} is important for this article. 

If $\mathcal{B}(\mathcal{Q})$ is the braid associated to the set $\mathcal{Q}$ of periodic orbits of $f$, we think of it as a representative of the braid type $[f,\mathcal{Q}]$ and define
\begin{equation}
    h_{\topo}(\mathcal{B}(\mathcal{Q})):=\Gamma_{\pi_1}([f,\mathcal{Q}]).
\end{equation}
From the discussion above, it is clear that any diffeomorphism $\phi$ of $\Sigma$  which has a set of periodic orbits which realizes the braid $\mathcal{B}(\mathcal{Q})$ for some choice of isotopy between $id$ and $\phi$ satisfies 
\begin{equation}
    h_{\topo}(\phi) \geq h_{\topo}(\mathcal{B}(\mathcal{Q})).
\end{equation}

\subsubsection{Topological entropy, braids and lower semicontinuity of $h_{\rm top}$ with respect to $d_{\rm Hofer}$}

\




From the discussion above, one can ask if the topological entropy of a surface diffeomorphism $\phi$ can be recovered from the topological entropy of the braid types which are realized by sets of periodic orbits of $\phi$. More precisely, given $\phi$ a surface diffeomorphism and $k>0$ a positive integer, we let $\mathrm{Braid}(k,\phi)$ be the set of braid types of sets of periodic orbits of $\phi$ of period $k$. We then define  $\mathrm{Braid}(\phi):= \cup_{k=1}^{+\infty} \mathrm{Braid}(k,\phi) $. Given a surface diffeomorphism is it true that \begin{equation}
    h_{\rm top}(\phi) = \sup_{\mathcal{Y} \in \mathrm{Braid}(\phi)} h_{\rm top}(\mathcal{Y}) \mbox{ ? }
\end{equation}
Hall proved in \cite{Hall_horseshoe} that this is the case for the horseshoe map, as a consequence of his study of horseshoe braid types. Using a different approach that combines the methods of Franks-Handel \cite{FranksHandel1988} and Katok-Mendoza \cite{Hasselblat-Katok} we 
obtain the following result, which states that this equality is valid for any surface diffeomorphism: this is essentially a restatement of Theorem \ref{thm:approximation}.
\begin{thm*}
Let $\Sigma$ be a compact surface, and let $\varphi: \Sigma\to \Sigma$ be a diffeomorphism such that $h_{\topo}(\varphi)>0$. Then $$h_{\rm top}(\phi) = \sup_{\mathcal{Y} \in \mathrm{Braid}(\phi)} h_{\rm top}(\mathcal{Y}).$$
\end{thm*}
\proof 
Let $\varphi:\Sigma \to \Sigma$ be diffeomorphism. If $h_{\topo}(\varphi) = 0$ there is nothing to prove. So assume $h_{\topo}(\varphi) >0$. Let $\varepsilon>0$. By Theorem \ref{thm:approximation} there is a hyperbolic $k$-periodic orbit $\mathcal{P} = (p_0, \ldots, p_{k-1})$ of $\varphi$, for some $k\in \N$, such that $$\Gamma_{\pi_1}([\varphi, \overline{\mathcal{P}}])>h_{\topo}(\varphi)-\epsilon,$$ where  $\overline{\mathcal{P}} = \{p_0,\ldots, p_{k-1}\}$ and $\Gamma_{\pi_1}([\varphi, \overline{\mathcal{P}}])$ is the growth rate of the action induced by  $\varphi$ on the fundamental group of $\Sigma \setminus \overline{\mathcal{P}}$, see Section \ref{sec:braids} for the definition of $\Gamma_{\pi_1}$ and some of its properties. But as explained in Section \ref{sec:braids}, if $\mathcal{B}(\mathcal{P})$ denotes the braid type of $\mathcal{P}$, then $$ h_{\topo}(\mathcal{B}(\mathcal{P}))= \Gamma_{\pi_1}([\varphi, \overline{\mathcal{P}}]). $$ The theorem then follows directly. \qed

The combination of this theorem with our braid stability results  allows us to obtain lower semicontinuity of $h_{\rm top}$ with respect to the Hofer metric.  This gives an answer to the two-dimensional version of the following question of Leonid Polterovich.

\begin{question} [Polterovich]
What continuity or stability properties does the topological entropy $h_{\rm top}$ have with respect to $d_{\rm Hofer}$?
\end{question}

\color{black}

\begin{rem}
Our methods and results in the present paper differ from those in \cite{chor-meiwes}, where the problem of stability of $h_{\topo}$ with respect to $d_{\rm Hofer}$ was first investigated. 
While in \cite{chor-meiwes} the stable lower bounds on $h_{\topo}$ stem from the properties of the braids projected back to the surface, such as geometric intersection numbers, the methods here allow us to deal with lower bounds that hold for more general braid types. For example, the results in \cite{chor-meiwes} only deal with surfaces $\Sigma$ of genus $\geq 2$, while our results also deal with $S^2$, $T^2$, and $D^2$. 
In \cite{chor-meiwes} it was shown that there are balls of any radius in Hofer's metric on which $h_{\topo}$ is positive. In the current paper we will only deal with stability properties under small perturbations. 
\end{rem} 

\begin{rem}
The continuity properties of $h_{\rm top}$ with respect to different topologies on spaces of dynamical systems has been much studied. The topological entropy is a measure of the complexity of a dynamical system, and it is interesting to investigate if dynamical complexity is stable under perturbation of a system. For example, the combined results of Yomdin \cite{Yomdin1987} and Newhouse \cite{Newhouse} imply that $h_{\rm top}$ is continuous with respect to the $C^\infty$-topology on the space of $C^\infty$-smooth diffeomorphisms of a closed surface. The analogous result does not hold for higher dimensional manifolds, as showed in \cite{Misiurewicz}. Using the fundamental work of Katok \cite{Katok}, Nitecki showed in \cite{nitecki} that $h_{\topo}$ is lower semicontinuous with respect to the $C^0$-topology on the space of $C^{1+\epsilon}$ diffeomorphisms of a closed surface, for any $\epsilon>0$. 
\end{rem}

The relationship between contact topology and topological entropy of families of contactomorphisms has been studied extensively and fruitfully in recent years by various methods. A large class of contactomorphisms are those that arise via Reeb flows and there
is an abundance of contact manifolds for which the topological entropy or the exponential orbit growth rate is positive for all Reeb flows. Examples and dynamical properties of those manifolds are investigated in \cite{AASS,Alves-Cylindrical,Alves-Anosov,A2,AlvesColinHonda2017,AlvesMeiwes2018,FrauenfelderSchlenk2006,MacariniSchlenk2011}. Some of these results generalise to positive contactomorphisms \cite{Dahinden2,Dahinden}, and results on the dependence of some lower bounds on topological entropy with respect to their positive contact Hamiltonians have been obtained in \cite{DahindenC0}. A related discussion and results on questions of $C^0$-stability of the topological entropy of geodesic flows can be found in \cite{ADMM}.
Aspects of the relationship between the topological entropy of Hamiltonian diffeomorphisms and Floer homology are also studied in \cite{chor-meiwes,CGG,FrauenfelderSchlenk2006}.

In a joint work in progress \cite{ADMP} of the authors, Abror Pirnapasov and Lucas Dahinden, we study robustness and stability properties of $h_{\topo}$ of $3$-dimensional Reeb flows using methods inspired by the ones of the present paper, and using the forcing theory for $h_{\topo}$ of Reeb flows developed in \cite{AlvesPirnapasov}.

\textbf{Acknowledgements:} This work benefited from discussions with Fr\'ed\'eric Bourgeois, Felix Schlenk and Umberto Hryniewicz. Our special thanks to Leonid Polterovich who first proposed to us the study of continuity properties of $h_{\topo}$ with respect to the Hofer metric.

\section{Main results} \label{sec:mainresults}

In this section we state precisely our main results. For this, we need the following definition. 
Let $H_{\oplus}$ be a path of Hamiltonians whose time $1$-map is $\phi_{\oplus}$. 

\begin{defn} \label{def:isolated}
Let $\mathcal{Y}_\oplus=\{\gamma_1,...,\gamma_k\}$ be a finite collection of distinct  $1$-periodic orbits of $X_{H_\oplus}$ that represent all the same free homotopy class of loops $\alpha$ and $\epsilon>0$ be a positive real number. We say that the collection $\mathcal{Y}_\oplus$ is \textit{$\epsilon$-isolated for the action $\mathcal{A}_{H_\oplus}$} if 
    \begin{itemize}
\item for all $1$-periodic orbits $\gamma,\gamma'$ of $X_{H_{\oplus}}$ representing $\alpha$ we have that $\Delta_{H_{\oplus}}(\gamma,\gamma')$ is either $0$ or $\geq\epsilon$.
\item for all $\gamma \in \mathcal{Y}_{\oplus}$ and all $1$-periodic orbits $\gamma'$ of $X_{H_{\oplus}}$ representing $\alpha$ we have that $\Delta_{H_{\oplus}}(\gamma, \gamma') = 0$ implies $\gamma'\in \mathcal{Y}_{\oplus}$. 
\end{itemize}
\begin{rem}
Note that $\Delta_{H_{\oplus}}(\gamma,\gamma')$ is defined in \eqref{def_Delta}, and that if the action does not depend on the capping, then  $\Delta_{H_{\oplus}}(\gamma,\gamma')$ is just the positive action difference, e.g. if $\gamma, \gamma'$ are contractible and $\Sigma \neq S^2$, then $\Delta_{H_{\oplus}}(y,y')=|\mathcal{A}_{H_{\oplus}}(\gamma) -\mathcal{A}_{H_{\oplus}}(\gamma')|$.
\end{rem}

\end{defn}

\begin{rem}
Notice that in this definition we do not require that for two different $1$-periodic orbits $\gamma_i \neq \gamma_j$ in $\mathcal{Y}_\oplus$ we have $\Delta_{H_{\oplus}}(\gamma_i,\gamma_j) \neq 0$. Indeed, in some of our most important dynamical results all the elements  $\mathcal{Y}_\oplus$ have the same action.
\end{rem}

Theorems \ref{thm:main}, \ref{thm:disk} and \ref{thm:disk0} given conditions under which we can guarantee that a braid type of periodic orbits of a Hamiltonian diffemorphism on a surface will persist under small perturbations with respect to $d_{\rm Hofer}$. Theorem \ref{thm:main} deals with the case of closed surfaces.

\begin{mainthm} \label{thm:main}
Let $\Sigma$ be a closed surface and $\omega$ be a symplectic form on $\Sigma$. If $\Sigma=S^2$ we assume that $\int_{S^2}\omega =8$. 
Let $\phi_\oplus$ be a non-degenerate Hamiltonian diffeomorphism of $(\Sigma,\omega)$ and $H_\oplus$ be a path of normalized Hamiltonians whose time $1$-map is $\phi_\oplus$. Assume that there exists a finite collection $\mathcal{Y}_\oplus=\{\gamma_1,...,\gamma_k\}$  of distinct, pairwise freely homotopic $1$-periodic orbits of $H_\oplus$ and a number $\epsilon>0$ such that $\mathcal{Y}_\oplus$ is $100\epsilon$-isolated for $\mathcal{A}_{H_\oplus}$, and let $\mathcal{B}(\mathcal{Y}_\oplus)$ be the braid in $S^1 \times \Sigma$ associated to $Y_\oplus$ as in definition \ref{def:braid}. 
If $\Sigma = S^2 $ we assume moreover that $\epsilon < \frac{1}{400}$.

Then, for any non-degenerate Hamiltonian diffeomorphism $\phi_\ominus$ whose Hofer distance to $\phi_\oplus$ is $< \epsilon$, there exist a path $H_\ominus$ of normalized Hamiltonians whose time $1$-map is $\phi_\ominus$ and a finite set $\mathcal{Y}_\ominus$ of $1$-periodic orbits of $H_\ominus$ such that $$\mathcal{B}(\mathcal{Y}_\ominus) \mbox{ is freely isotopic as a braid to } \mathcal{B}(\mathcal{Y}_\oplus),$$
where $\mathcal{B}(\mathcal{Y}_\ominus)$ is the braid associated to $\mathcal{Y}_\ominus$.
\end{mainthm}

We now state our main results for Hamiltonian diffeomorphisms of the disk.

\begin{mainthm} \label{thm:disk}
Let $\phi_\oplus \in \mathrm{Ham}_c(\D)$ be a non-degenerate Hamiltonian diffeomorphism of $(\D,dx \wedge dy)$ for some $c\in \R \setminus \Q$ and $H_\oplus$ be a path of normalized Hamiltonians whose time $1$-map is $\phi_\oplus$. Assume that there exists a finite collection $\mathcal{Y}_\oplus=\{\gamma_1,...,\gamma_k\}$   of distinct $1$-periodic orbits of $H_\oplus$ and a number $\epsilon>0$ such that $\mathcal{Y}_\oplus$ is $100\epsilon$-isolated for $\mathcal{A}_{H_\oplus}$, and let $\mathcal{B}(\mathcal{Y}_\oplus)$ be the braid in $S^1 \times \Sigma$ associated to $Y_\oplus$ as in definition \ref{def:braid}.

Then, for any non-degenerate Hamiltonian diffeomorphism $\phi_\ominus\in \mathrm{Ham}_c(\D)$ whose Hofer distance to $\phi_\oplus$ is $< \epsilon$, there exist a path $H_\ominus$ of normalized Hamiltonians whose time $1$-map is $\phi_\ominus$ and a finite set $\mathcal{Y}_\ominus$ of $1$-periodic orbits of $H_\ominus$ such that $$\mathcal{B}(\mathcal{Y}_\ominus) \mbox{ is freely isotopic as a braid to } \mathcal{B}(\mathcal{Y}_\oplus),$$
where $\mathcal{B}(\mathcal{Y}_\ominus)$ is the braid associated to $\mathcal{Y}_\ominus$.
\end{mainthm}

For compactly supported Hamiltonian diffeomorphisms of the disk we need to impose further conditions on the periodic orbits that form the braid.

\begin{mainthm} \label{thm:disk0}
Let $\phi_\oplus \in \mathrm{Ham}_0(\D)$ be a compactly supported non-degenerate Hamiltonian diffeomorphism of $(\D,dx \wedge dy)$ and $H_\oplus$ be a path of normalized Hamiltonians whose time $1$-map is $\phi_\oplus$. Assume that there exists a finite collection $\mathcal{Y}_\oplus=\{\gamma_1,...,\gamma_k\}$   of distinct $1$-periodic orbits of $H_\oplus$ and a number $\epsilon>0$ such that $\mathcal{Y}_\oplus$ is $100\epsilon$-isolated for $\mathcal{A}_{H_\oplus}$ and such that $\mathcal{A}_{H_\oplus}(\gamma_i) \neq 0$ for all $i \in \{1,...,k\}$. Let $\mathcal{B}(\mathcal{Y}_\oplus)$ be the braid in $S^1 \times \Sigma$ associated to $Y_\oplus$ as in definition \ref{def:braid}.

Then, for any compactly supported non-degenerate Hamiltonian diffeomorphism $\phi_\ominus\in \Ham_0(\D)$ whose Hofer distance to $\phi_\oplus$ is $< \epsilon$, there exist a path $H_\ominus$ of normalized Hamiltonians whose time $1$-map is $\phi_\ominus$ and a finite set $\mathcal{Y}_\ominus$ of $1$-periodic orbits of $H_\ominus$ such that $$\mathcal{B}(\mathcal{Y}_\ominus) \mbox{ is freely isotopic as a braid to } \mathcal{B}(\mathcal{Y}_\oplus),$$
where $\mathcal{B}(\mathcal{Y}_\ominus)$ is the braid associated to $\mathcal{Y}_\ominus$.
\end{mainthm}

\begin{rem}\label{main_weaker_assumption}
In the above theorems the isolation condition on the orbits can be replaced by the more technical condition that we will call quasi-isolation,  where a set of orbits is $\epsilon$-quasi-isolated if the set of their action values are $\epsilon$-isolated and if there are no non-constant Floer cylinder with energy $< \epsilon$ which are asymptotic to one of those orbits. See section \ref{sec:quasi-iso} for the precise definition and some consequences, and section \ref{sec:replacing_assumption} for the argument why Theorems \ref{thm:main} and \ref{thm:disk} remain true with this assumption. 
\end{rem}

\subsection{Applications}
\subsubsection{Braid stability for a set of non-degenerate orbits}
We first formulate a theorem that will follow from the results in section  \ref{sec:mainresults}. It is a stability statement for the braid that  that is associated to a set of non-degenerate, pairwise freely homotopic $1$-periodic orbits of a (not necessarily non-degenerate) Hamiltonian diffeomorphism on $\Sigma$.  

In the following we say that a compactly supported Hamiltonian diffeomorphism on $\D$ is non-degenerate in its support if all its $1$-periodic orbits in the interior of its support are non-degenerate. 

\begin{thm}\label{thm:nondeg_stable}
Let $\Sigma$ be a closed surface or the two-disc $\D$ equipped with a symplectic form $\omega$ (where $\omega = \omega_0$ if $\Sigma = \D$).  
Let $\phi_{\oplus}$ be a compactly supported Hamiltonian diffeomorphism  and $H_{\oplus}$ a path of normalized Hamiltonians with $\phi_{\oplus} = \phi^1_{H_{\oplus}}$. Let $\mathcal{Y}_{\oplus} = \{\gamma_1, \ldots,\gamma_k\}$ be a collection of distinct non-degenerate $1$-periodic orbits of $H_{\oplus}$ that are pairwise freely homotopic.  Let $\mathcal{B}(\mathcal{Y}_{\oplus})$ be the braid in $S^1 \times \Sigma$ associated to $\mathcal{Y}_{\oplus}$ as in Definition \ref{def:braid}.  
Then there is $\epsilon'>0$ such that for any compactly supported Hamiltonian diffeomorphism $\phi_{\ominus}$ that is non-degenerate (in its support if $\Sigma = \D^2$) and that satisfies  $d_{\hofer}(\phi_\ominus, \phi_{\oplus}) < \epsilon'$,  there is a path $H_{\ominus}$ of normalized Hamiltonians with $\phi_{\ominus} = \phi^1_{H_{\ominus}}$ and a set of $1$-periodic orbits $\mathcal{Y}_{\ominus}$ for $H_{\ominus}$ such that 
$$ \mathcal{B}(\mathcal{Y}_{\ominus}) \text{ is freely isotopic as a braid to } \mathcal{B}(\mathcal{Y}_{\oplus}),$$
where $\mathcal{B}(\mathcal{Y}_\ominus)$ is the braid associated to $\mathcal{Y}_\ominus$. 
\end{thm}

A proof of this result is given in section \ref{sec:nondeg_stable}. 

\subsubsection{Lower semicontinuity of the $h_{\topo}$ with respect to Hofer's metric}


A consequence of Theorem \ref{thm:nondeg_stable} and the  approximation results on $h_{\topo}$ in Appendix \ref{appendix:approx} is the lower semicontinuity of the topological entropy $h_{\topo}$ with respect to $d_{\hofer}$ in dimension two. 
 
\begin{thm}\label{thm:htop_lowersemi}
Let $\Sigma$ be a closed surface equipped with a symplectic form $\omega$, then $h_{\topo}:(\Ham(\Sigma, \omega), d_\hofer) \to [0,\infty)$ is lower semicontinuous.
\end{thm}
\begin{proof}
Let $\varphi:\Sigma \to \Sigma$ be a Hamiltonian diffeomorphism. If $h_{\topo}(\varphi) = 0$ there is nothing to prove. So assume $h_{\topo}(\varphi) >0$. Let $\varepsilon>0$.
By Theorem \ref{thm:approximation} there is a hyperbolic $k$-periodic orbit $\mathcal{P} = (p_0, \ldots, p_{k-1})$ of $\varphi$, for some $k\in \N$, such that $$\Gamma_{\pi_1}([\varphi, \overline{\mathcal{P}}])>h_{\topo}(\varphi)-\epsilon,$$ where  $\overline{\mathcal{P}} = \{p_0,\ldots, p_{k-1}\}$ and $\Gamma_{\pi_1}([\varphi, \overline{\mathcal{P}}])$ is the growth rate by the action induced by  $\varphi$ on the fundamental group of $\Sigma \setminus \overline{\mathcal{P}}$, see section \ref{sec:braids} for the definition of $\Gamma_{\pi_1}$ and some of its properties.
Let $G:\Sigma \times S^1$ be a Hamiltonian with $\phi^1_G = \varphi$. The diffeomorphism $\phi_{\oplus}:=\varphi^k$ is generated by the Hamiltonian  $H_{\oplus}:\Sigma \times S^1 \to \R$ defined by $H_{\oplus}(x,t) := kG(x,kt)$. 
Let $\mathcal{Y}_{\oplus} = \{\gamma_0, \ldots, \gamma_{k-1}\}$ be the set of $1$-periodic orbits for $H_{\oplus}$ given by $\gamma_i(t) = \phi^t_{\oplus}(p_i) = \phi^{k(t+i/k)}_G(p_0)$, $i=0, \ldots, k-1$. These orbits have the same image and are in particular pairwise freely homotopic. These orbits are hyperbolic, in particular non-degenerate.
Choose $\epsilon'>0$ as in Theorem \ref{thm:nondeg_stable} with respect to  $\phi_{\oplus}, H_{\oplus}$ and $\mathcal{Y}_{\oplus}$. 
Set $\delta = \frac{\epsilon'}{k}$.
Now let $\psi$ be any strongly non-degenerate Hamiltonian diffeomorphism with $d_{\hofer}(\psi, \varphi) < \delta$. By the bi-invariance of the metric $d_{\hofer}$ it follows for $\phi_{\ominus} :=\psi^k$ that 
$d_{\hofer}(\phi_{\ominus}, \phi_{\oplus}) \leq k d_{\hofer}(\psi,\varphi) < k\delta =\epsilon'$, and hence there is a Hamiltonian  $H_{\ominus}:\Sigma \times S^1 \to \R$ with $\phi^1_{H_{\ominus}} = \phi_{\ominus}$ and a set of $1$-periodic orbits $\mathcal{Y}_{\ominus}= \{\gamma'_0, \ldots, \gamma'_{k-1}\}$ for $H_{\ominus}$ such that $\mathcal{B}(\mathcal{Y}_{\ominus})$ is isotopic as a braid to $\mathcal{B}(\mathcal{Y}_{\oplus})$. 
In particular it follows that the braid types $[\varphi^k,\overline{\mathcal{P}}]$ and $[\psi^k, \overline{\mathcal{Y}_{\ominus}}]$ coincide, where $\overline{\mathcal{Y}}_{\ominus} = \{\gamma'_0(0), \ldots, \gamma'_{k-1}(0)\}$.

We conclude that 
\begin{align*}
h_{\topo}(\psi) &= \frac{1}{k}h_{\topo}(\psi^k) \\&\geq \frac{1}{k} \Gamma_{\pi_1}([\psi^k, \overline{\mathcal{Y}_\ominus}]) \\&= \frac{1}{k}\Gamma_{\pi_1}([\varphi^k, \overline{\mathcal{P}})\\&= \frac{1}{k}\left(k\Gamma_{\pi_1}([\varphi, \overline{\mathcal{P}}])\right) \\&\geq h_{\topo}(\varphi)-\varepsilon.
\end{align*}
Since strongly non-degenerate $\psi$ are $C^{\infty}$-dense in $\Ham(\Sigma, \omega)$ and since $h_{\topo}$ is $C^{\infty}$ lower semicontinuous \cite{Newhouse}, this finishes the proof.  
\end{proof}

\begin{thm}\label{thm:htop_lowersemi_disc}
Let $\Sigma = \D$ be equipped with the standard symplectic form $\omega_0$, then $h_{\topo}:(\Ham_c(\D, \omega_0), d_\hofer) \to [0,\infty)$ is lower semicontinuous.
\end{thm}
The proof works along the lines of Theorem \ref{thm:htop_lowersemi} if one considers only compactly supported Hamiltonian diffeomorphisms.  

{ 
In the recent work \cite{CGG} \c{C}ineli-Ginzburg-G\"{u}rel showed that for a Hamiltonian diffeomorphism $\phi$ on a surface $h_{\topo}$ coincides with the barcode entropy $\hbar(\phi)$ of $\phi$ which they introduce. The barcode entropy is a measure of the complexity of the Floer barcodes of $\phi$ in the spirit of \cite{leonidpersistence}. Combining Theorems \ref{thm:htop_lowersemi} and \ref{thm:htop_lowersemi_disc} with Theorem C of \cite{CGG} we obtain the following
\begin{cor}
The barcode entropy $\hbar$ is lower semicontinuous with respect to $d_{\rm Hofer}$ on surfaces.
\end{cor}}

\section{Background}

\subsection{Recollections on Hamiltonian dynamics and Floer homology for closed surfaces different from $S^2$} \label{sec:Hamonclosedsurfaces}

\subsubsection{Surfaces different from  $S^2$ and $T^2$}\label{sec:Hamonhighergenus} 
In the following we assume that $\Sigma$ is a closed surface with $\Sigma \neq S^2$ and $\neq T^2$. The constructions in the case of non-contractible loops in $T^2$ differ slightly, so we consider the case of the torus below in \ref{sec:Hamontorus}. Let $\omega$ be a symplectic form on $\Sigma$. 
We consider a normalized time-dependent Hamiltonian $H:S^1 \times \Sigma \to \R$.
By normalized we mean that $\int_\Sigma H_t\omega=0$ for each $t\in S^1$, where $H_t(\cdot):= H(t,\cdot)$.


 
Recall that the group $\mathrm{Ham}(\Sigma,\omega)$ of Hamiltonian diffeomorphisms of $(\Sigma,\omega)$ is formed by the area-preserving diffeomorphisms of $(\Sigma,\omega)$ which are the time $1$-map of the Hamiltonian flow of some $H: S^1 \times \Sigma \to \R$. A reference for the study of $\mathrm{Ham}(\Sigma,\omega)$ is \cite{leonidsbook}.


All time-dependent Hamiltonian functions on closed surfaces considered in this paper are assumed to be normalized, as stated in the introduction.
 
 
As mentioned in the introduction, for the definition of the Hamiltonian action of non-contractible loops we fix for each free homotopy class $\alpha \in [S^1, \Sigma]$ of loops in $\Sigma$ a representative $\eta_{\alpha}:S^1 \to \Sigma$ of $\alpha$.  We define the Hamiltonian action as in the introduction.
The set of all $1$-periodic orbits of $H$ is denoted by $\mathcal{P}(H)$.
The action spectrum $\mathrm{Spec}^1(H)$ is defined as
\begin{equation}
    \mathrm{Spec}^1(H):= \{ \mathcal{A}_H(\gamma) \ | \ \gamma \mbox{ is a } \mbox{1-periodic orbit of } H \}. 
\end{equation}
 It is not hard to see that $\mathrm{Spec}^1(H)$ is a compact subset of $\R$. If $k$ is a positive integer we let 
 \begin{equation}
    \mathrm{Spec}^k(H):= \{ \mathcal{A}_H(\gamma) \ | \ \gamma \mbox{ is a } \mbox{k-periodic orbit of } H \}. 
\end{equation}
It is not hard to see that $\mathrm{Spec}^k(H):= \mathrm{Spec}^1(kH)$.
If a Hamiltonian $H$ is non-degenerate, then $\mathrm{Spec}^1(H)$ is a finite set. If $H$ is strongly non-degenerate, then $\mathrm{Spec}^k(H)$ is a finite set for every positive integer $k$.

If $\gamma$ is a $1$-periodic orbit of $\phi^t_H$ we let $\mu_{\rm CZ}(\gamma)$ be the Conley-Zehnder index of $\gamma$; see for example \cite{AD} for the definition of $\mu_{\rm CZ}$. We note that in order that $\mu_{\rm CZ}(\gamma)$ is actually well-defined for non-contractible loops on $\Sigma \neq T^2$, we fix for each $\eta_{\alpha}$ a symplectic trivialization $\Phi_{\alpha}$ of $\eta_{\alpha}^*T\Sigma$. This defines a homotopically canonical symplectic trivialization on $\gamma^*T\Sigma$. 
 
An almost complex structure $J$ on $(\Sigma, \omega)$ is called compatible if $\omega(\cdot,J\cdot)$ is a Riemannian metric on $\Sigma$.

To use Floer theory for Hamiltonian diffeomorphisms on $(\Sigma,\omega)$ we use the following setup.
We consider a $C^\infty$-smooth $S^1$-family $J_t$ of compatible almost complex structures on $\Sigma$.
The Floer operator for $(H,J_t)$ applied to a cylinder $u: \R \times S^1 \to \Sigma$ is
\begin{equation} \label{eq:defFloereq}
    \mathcal{F}_{{H},J}(u) = \partial_s u(s,t) +  J_t(u(s,t))\big(\partial_t u (s,t) - X_{{H}}(t,u(s,t)\big).
\end{equation} 

A solution of the Floer equation for $(H,J_t)$ is a cylinder $u$ such that $ \mathcal{F}_{{H},J}(u)=0$. We call such cylinders Floer cylinders. 

We assume from now on that $H$ is non-degenerate. 
The energy ${E}(u)$ of a Floer cylinder is defined by the formula $$E(u):= \int_{\R \times S^1}|\partial_s u |^2 dt ds,$$ where $ |\partial_s u(s,t) |^2 = \omega(\partial_s u(s,t), J^t (\partial_s u(s,t) ) )$. Floer showed \cite{Floer} that if a Floer cylinder has finite energy then there exist $1$-periodic orbits $\gamma$ and $\gamma'$ of $\phi^t_H$ such that $
\lim_{s\to -\infty}u(s,\cdot) = \gamma(\cdot)$, and $\lim_{s\to +\infty}u(s,\cdot) = \gamma'(\cdot)$. A well-known computation shows that $E(u) = \mathcal{A}_H(\gamma) - \mathcal{A}_H(\gamma') $.

The energy of a Floer cylinder is by definition non-negative. The only Floer cylinders with energy equal to $0$, are those which are of the form $u(s,t)= \gamma(t)$, where $\gamma$ is a $1$-periodic of $\phi^t_H$. The Floer cylinder $u_\gamma(s,t)= \gamma(t)$ is called the trivial cylinder over $\gamma$.

For two $1$-periodic orbits $\gamma$ and $\gamma'$ of ${H}$ 
we let $\mathcal{M}(\gamma, \gamma', {H},J_t)$ be the moduli space of Floer cylinders $u: \R \times S^1 \to \Sigma$ with asymptotics $
\lim_{s\to -\infty}u(s,\cdot) = \gamma(\cdot)$, and $\lim_{s\to +\infty}u(s,\cdot) = \gamma'(\cdot)$. Two Floer cylinders $u$ and $v$ represent the same element in $\mathcal{M}(\gamma, \gamma', {H},J_t)$ if there exists $s_0 \in \R $ such that $u(s+s_0, t) = v(s,t)$ for all $(s,t) \in \R \times S^1$.

As shown in \cite{FHS}, for a $C^\infty$-generic choice of $J_t$, for any choice of  $1$-periodic orbits $\gamma$ and $\gamma'$ of $\phi^t_H$ the moduli spaces $\mathcal{M}(\gamma, \gamma', {H},J_t)$ are smooth manifolds whose dimension is $\mu_{\rm CZ}(\gamma) - \mu_{\rm CZ}(\gamma') - 1$. 
We assume from now on that $J_t$ is $C^\infty$-generic in this sense: such $J_t$ will be referred to as regular.
The compactification $\overline{\mathcal{M}}(\gamma, \gamma', {H},J_t)$
as defined by Floer is the union of $\mathcal{M}(\gamma, \gamma', {H},J_t)$ with the set of broken Floer cylinders negatively asymptotic to $\gamma$ and positively asymptotic to $\gamma'$; we refer the reader to \cite{AD} and \cite{Floer}. 

Fix $1$-periodic orbits $\gamma$ and $\gamma'$ of $\phi^t_H$ and suppose that $\mu_{\rm CZ}(\gamma') = \mu_{\rm CZ}(\gamma) - 1$. In this case $\mathcal{M}(\gamma, \gamma', {H},J_t)$ is a $0$-dimensional manifold. As shown in \cite{AD,Floer}, in this case $\overline{\mathcal{M}}(\gamma, \gamma', {H},J_t)$ cannot contain broken Floer cylinders, and it follows that $\mathcal{M}(\gamma, \gamma', {H},J_t)= \overline{\mathcal{M}}(\gamma, \gamma', {H},J_t)$. We conclude that in this case $\mathcal{M}(\gamma, \gamma', {H},J_t)$ is composed of a finite set of elements. We then define 
 $$ C(\gamma,\gamma')= (\#  \mathcal{M}(\gamma, \gamma', {H},J_t)) \mod 2, $$
for any pair $\gamma$ and $\gamma'$ of $1$-periodic orbits  of $\phi^t_H$ which satisfies $\mu_{\rm CZ}(\gamma') = \mu_{\rm CZ}(\gamma) - 1$. If $\mu_{\rm CZ}(\gamma') \neq \mu_{\rm CZ}(\gamma) - 1$ we let $C(\gamma,\gamma')=0$.  

For any real number $b$ which is not in $\mathrm{Spec}^1(H)$ we let $\mathcal{P}^1_{b}(H)$ be the set of $1$-periodic orbits of $H$ which have action $<b$. If $b=+\infty$ we write just $\mathcal{P}^1(H)$. For  real  numbers $a<b$ which are not in $\mathrm{Spec}^1(H)$ we let $\mathcal{P}^1_{(a,b)}(H)$ be the set of $1$-periodic orbits with action in the interval $(a,b)$. 

Given a real number $b$ which is not in $\mathrm{Spec}^1(H)$ we define the Floer chain complex $$CF^b(H) := \bigoplus_{\gamma \in \mathcal{P}^1_{b}(H)} \Z_2 \cdot \gamma.$$
The differential $d^{J_t}: CF^b(H) \to CF^b(H)$ is defined for each $\gamma \in \mathcal{P}^1_{b}(H)$ by
\begin{equation*}
d^{J_t} (\gamma) := \sum_{\gamma' \in \mathcal{P}^1_{b}(H) } C(\gamma,\gamma')\gamma' = \sum_{\gamma' \in \mathcal{P}^1(H) } C(\gamma,\gamma')\gamma'.
\end{equation*}
The second equality is due to the fact that $C(\gamma,\gamma')\neq 0$ implies that $\mathcal{A}_H(\gamma') < \mathcal{A}_H(\gamma) < b $, since $\gamma \in \mathcal{P}^1_{b}(H)$. The differential $d^{J_t}$ is extended to all of $CF^{b}(H)$ linearly. As showed by Floer \cite{Floer} (see also \cite{AD}) $(d^{J_t})^2=0$, and we let  $HF^{b}(H)$ denote the homology of the pair $(CF^{b}(H),d^{J^t})$. The differential $d^{J_t}$ is extended to all of $CF^{(a,b)}(H)$ linearly. When the choice of $J_t$ is clear from the context we will drop $J_t$ from the notation of the differential and denote only by $d$. 

Given real numbers $a<b$ which are not in $\mathrm{Spec}^1(H)$, we let $CF^{(a,b)}(H):= \frac{CF^b(H)}{{CF^a(H)}} $. Because the differential $d^{J_t}$ maps $CF^a(H)$ to itself, we have that $(CF^a(H),d^{J_t} )$ is a sub-complex of $(CF^{b}(H),d^{J^t})$. As a consequence $d^{J^t}$ induces a differential on the quotient complex $CF^{(a,b)}(H)$, which we continue to denote by $d^{J^t}$. We then let $HF^{(a,b)}(H)$ be the homology of the pair $(CF^{(a,b)}(H),d^{J^t})$.

It is easy to see that if $\gamma \in \mathcal{P}^1_{(a,b)}(H) $ is seen as an element of $CF^{b}(H)$, then its differential $d^{J_t} (\gamma) $ is given by the formula
\begin{equation*}
d^{J_t} (\gamma) = \sum_{\gamma' \in \mathcal{P}^1_{(a,b)}(H) } C(\gamma,\gamma')\gamma'. 
\end{equation*}

We observe that the chain complexes $CF^b(H)$ considered above can be written as a direct sum over the complexes $CF_{\alpha}^b(H)$, where $\alpha$ runs over all free homotopy classes of loops in $\Sigma$ (denoted by $[S^1,\Sigma]$) and the generators for  $CF_{\alpha}^b(H)$ are the loops generating $CF^b(H)$ that also lie in the class $\alpha$. The same holds for $CF^{(a,b)}(H)$ and $HF^{(a,b)}(H)$, and we will add a lower index $\alpha$ to the chain complexes and homology groups if we want to specify the free homotopy class $\alpha$.



{\it Continuation maps and Hofer distance.} \, \, 
Let $H_\oplus$ and $H_\ominus$ be normalized Hamiltonians whose time $1$-maps are non-degenerate. A homotopy between $H_\oplus$ and $H_\ominus$, is a function $Q:\R \times S^1 \times \Sigma \to \R$ such that for some $R>0$
\begin{align} \label{eq:G properties}
 Q(s,t,p) = H_\oplus(t,p) \mbox{ for } s\leq -R, \\
    Q(s,t,p) = H_\ominus(t,p) \mbox{ for } s\geq R.
\end{align}
We let $J^\oplus_t$ and $J^\ominus_t$ be smooth $S^1$-families of compatible almost complex structures on $(\Sigma,\omega)$ which are regular for  $H_\oplus$ and $H_\ominus$, respectively. In this case we can define the homologies $HF(H_\oplus)$ and $HF(H_\ominus)$. A smooth homotopy $J^s_t$ of compatible almost complex structures on $(\Sigma, \omega)$ between $J^\oplus_t$ and $J^\ominus_t$, is a smooth $\R \times S^1$-family of compatible almost complex structures on $(\Sigma, \omega)$ such that $J^s_t=J_t^\oplus$ for sufficiently small $s$ and $J^s_t=J_t^\ominus$ for sufficiently large $s$. We then consider the Floer operator of $(Q,J^s_t)$ which applied to a cylinder $u: \R \times S^1 \to \Sigma$ is
\begin{equation}
    \mathcal{F}_{Q,J^s_t}(u) = \partial_s u(s,t) + J^s_t(\partial_t u(s,t) - X_Q(s,t,u(s,t)).
\end{equation}
Solutions of the Floer equation $(Q,J^s_t)$ are cylinders $u$ such that $\mathcal{F}_{Q,J^s_t}(u)=0$, and these are called Floer cylinders. The energy of these cylinders is defined as in the previous section, see \cite{AD}. 

If $u$ is a finite energy Floer cylinder of $(Q,J^s_t)$ then there exists $1$-periodic orbits $\gamma_\oplus$ of $\phi^t_{H_\oplus}$ and $\gamma_\ominus$ of $\phi^t_{H_\ominus}$ such that $\lim_{s \to -\infty} u(s, \cdot) = \gamma_\oplus$ and $\lim_{s \to +\infty} u(s, \cdot) = \gamma_\ominus$. For $C^\infty$-generic choices of homotopies $J^s_t$ and $Q$ the moduli spaces $\mathcal{M}(\gamma_\oplus, \gamma_\ominus, Q,J^s_t)$ of Floer cylinder of $(Q,J^s_t)$ which are negatively asymptotic to $\gamma_\oplus$ and positively asymptotic to $\gamma_\ominus $ are manifolds of dimension $\mu_{\rm CZ}(\gamma_\oplus) - \mu_{\rm CZ}(\gamma_\ominus)$, for all $1$-periodic orbits $\gamma_\oplus$ of  $\phi^t_{H_\oplus}$ and $\gamma_\ominus$ of $\phi^t_{H_\ominus}$. 

In case $\mu_{\rm CZ}(\gamma_\oplus) = \mu_{\rm CZ}(\gamma_\ominus)$, the  space $\mathcal{M}(\gamma_\oplus, \gamma_\ominus, Q,J^s_t)$ is $0$-dimensional, and using the regularity of $J^s_t$ and $Q$ and Floer compactness one obtains that $\mathcal{M}(\gamma_\oplus, \gamma_\ominus, Q,J^s_t)$ is compact and therefore a finite set of points. 
We define $K_{Q,J^s_t}(\gamma_\oplus, \gamma_\ominus):= (\# \mathcal{M}(\gamma_\oplus, \gamma_\ominus, Q,J^s_t))\mod 2$, if  $\mu_{\rm CZ}(\gamma_\oplus) = \mu_{\rm CZ}(\gamma_\ominus)$ and $K_{Q,J^s_t}(\gamma_\oplus, \gamma_\ominus)=0$ otherwise. We define the continuation map \linebreak $\Psi_{Q,J^s_t}: CF(H_\oplus) \to CF(H_\ominus)$ by 
\begin{equation} \label{eq:continuation}
    \Psi_{Q,J^s_t}(\gamma_\oplus) := \sum_{\gamma_\ominus \in \mathcal{P}(H_\ominus)} K_{Q,J^s_t}(\gamma_\oplus,\gamma_{\ominus}) \gamma_{\ominus}.
\end{equation}
As shown in \cite{AD} the $\Psi_{Q,J^s_t}$ induces a homology map which we denote by $\Psi_{Q,J^s_t}: HF(H_\oplus) \to HF(H_\ominus)$.

We will need to consider continuation maps between Floer homologies in certain action windows. This is possible under certain conditions as explained in the next proposition. 
\begin{prop} \label{prop:actionwindowsclosedsurface}
Let $\phi_\oplus$ be a non-degenerate Hamiltonian diffeomorphism in $\mathcal{H}(\Sigma,\omega)$ and $H_\oplus:S^1 \times \Sigma \to \R$ be normalized Hamiltonian generating $\phi_\oplus$. We take real numbers $a<b$ which do not belong to $\mathrm{Spec}^1(H)$ and let $\epsilon>0$ be such that all elements of $\mathrm{Spec}^1(H)$  in the interval $(a-2\epsilon, b+2\epsilon)$ are contained in $(a,b)$. Let $\phi_\ominus$ be a Hamiltonian diffeomorphism with $d_{\rm Hofer}(\phi_\oplus, \phi_\ominus)< \epsilon$. Then, there exist a normalized Hamiltonian $H_\ominus :S^1 \times \Sigma \to \R $ and homotopies $G: \R \times S^1 \times \Sigma \to \R $ between $H_\oplus$ and $H_\ominus$ and $\widehat{G}: \R \times S^1 \times \Sigma \to \R $ between $H_\ominus$ and $H_\oplus$ 
which induce continuation maps 
\begin{align*} 
& \Psi_G^b: CF^{b}(H_\oplus) \to CF^{b+\epsilon}(H_\ominus), \\ & \Psi_G^a: CF^{a - 2\epsilon}(H_\oplus) \simeq CF^{a }(H_\oplus) \to CF^{a-\epsilon}(H_\ominus),
\end{align*}
and 
\begin{align*}
& \Psi_{\widehat{G}}^b: CF^{b+ \epsilon}(H_\ominus) \to CF^{b+2\epsilon}(H_\oplus) \simeq CF^{b}(H_\oplus), \\ & \Psi_{\widehat{G}}^a: CF^{a-\epsilon  }(H_\ominus) \to CF^{a}(H_\oplus),
\end{align*}
whose compositions $$\Psi_{\widehat{G}}^b\circ \Psi_G^b: CF^{b}(H_\oplus)   \to CF^{b+ 2\epsilon}(H_\oplus) \simeq CF^{b }(H_\oplus) ,$$ and $$\Psi_{\widehat{G}}^a\circ \Psi_G^a: CF^{a}(H_\oplus) \simeq CF^{a - 2\epsilon}(H_\oplus)  \to CF^{a}(H_\oplus)$$
are both chain homotopic to the identities ${\rm id}: CF^{a}(H_\oplus)   \to CF^{a}(H_\oplus) $ and ${\rm id}: CF^{b}(H_\oplus)   \to CF^{b}(H_\oplus) $, respectively.

It follows that $G$ induces a map $$ \Psi_G:  CF^{(a,b)}(H_\oplus) \to CF^{(a-\epsilon, b+\epsilon)}(H_\ominus) $$ and $\widehat{G}$ induces a map $$\Psi_{\widehat{G}}:  CF^{(a-\epsilon, b+\epsilon)}(H_\ominus) \to CF^{(a-2\epsilon,b+2\epsilon)}(H_\oplus) \simeq CF^{(a,b)}(H_\oplus),$$ such that the composition $\Psi_{\widehat{G}}\circ \Psi_G$ is chain homotopic to the identity map ${\rm id}:  CF^{(a,b)}(H_\oplus) \to  CF^{(a,b)}(H_\oplus) $.

\end{prop}

\proof 

We start by explaining the construction of $H_\ominus$, $G$ and $\widehat{G}$. 
Because $d_{\rm Hofer}(\phi_\ominus,\phi_\oplus)<\epsilon $ there exists a normalized Hamiltonian 
\begin{equation} \label{eq:Hminus}
  F:S^1\times \Sigma \to \R   
\end{equation}
 whose time $1$-map is $\phi^{-1}_\oplus \circ \phi_\ominus$ and that satisfies 
\begin{equation}\label{eq:Hoferdistance}
    \int_0^1 (\max F_t - \min F_t) dt< \epsilon,
\end{equation}
 where for $t\in S^1$ we define $F_t:=F(t,\cdot): \Sigma \to \R$.
 
 It follows that the time $1$-map of the Hamiltonian $H_\ominus: S^1 \times \Sigma \to \R$ defined by
\begin{equation} \label{eq:F}
H_\ominus(t,p) := H_\oplus(t,p) + F_t((\phi_\oplus^t)^{-1}(p))
\end{equation}
is $\phi_\ominus$. Because $H_\oplus$ and $F$ are normalized, it follows that $H_\ominus$ is also normalized.

We define a $C^\infty$-smooth function $\beta: \R \to [0,1]$ which is non-decreasing and satisfies
\begin{align} \label{eq:beta}
&    \beta(s) = 1 \mbox{ for } s \geq -1, \\
  \nonumber & \beta(s)= 0 \mbox{ for } s \leq -2, \\
  \nonumber & \beta'(s) \leq 2.
\end{align}

We define $G : \R \times S^1  \times \Sigma \to \R$ by the formula
\begin{align} \label{eq:G}
    G(s,t,p) := \beta(s) H_\ominus(t,p) + (1-\beta(s))H_\oplus(t,p) = \\
   \nonumber  = H_\oplus(t,p) + \beta(s) F_t((\phi_\oplus^t)^{-1}(p))
\end{align}
Notice that $G$ only depends on $s$ for $-2 \leq s \leq -1$, and that $G$ is a homotopy between $H_\oplus$ and $H_\ominus$ as 
\begin{align} \label{eq:Gproperties}
    G(s,t,p) = H_\oplus(t,p) \mbox{ for } s\leq -2, \\
    G(s,t,p) = H_\ominus(t,p) \mbox{ for } s\geq -1.
\end{align}

Likewise we construct a homotopy $\widehat G : \R \times S^1  \times \Sigma \to \R$  between $H_\ominus$ and $H_\oplus$ defined by
\begin{equation} \label{eq:G-hat}
    \widehat G (s,t,p) := G(-s,t,p).
\end{equation}

We choose homotopies $J^s_t$ between $J^\oplus_t$ and $J^\ominus_t$,  and  $\widehat{J}^s_t$ between $J^\ominus_t$ and $J^\oplus_t$.  

Let $u: \R \times S^1 \to \Sigma$ be a Floer cylinder of $(G,J^s_t)$ negatively asymptotic to $\gamma_\oplus \in \mathcal{P}^1(H_\oplus)$ and positively asymptotic to $\gamma_\ominus \in \mathcal{P}^1(H_\ominus)$.
A direct computation shows that the energy $E(u)$ defined by 
\begin{equation}\label{eq:energy}
    E(u) := \int_{\R \times S^1} |\partial_s u|^2 dt ds,
\end{equation}
where $|\partial_s u(s,t)|^2 = \omega_{u(s,t)}(\partial_s u(s,t),J^{s}_{t}(\partial_s u(s,t) ))$ satisfies
\begin{equation} \label{eq:energy'}
    E(u)= \mathcal{A}_{H_\oplus}(\gamma_\oplus) - \mathcal{A}_{H_\ominus}(\gamma_\ominus) + \int_{\R \times S^1} \frac{\partial G}{\partial s}(s,t,u_a(s,t)) ds dt;
\end{equation}
see for example \cite[Section 2]{Usher}.
We conclude that 
\begin{equation} \label{eq:estimate}
    \left| \mathcal{A}_{H_\oplus}(\gamma_\oplus) - \mathcal{A}_{H_\ominus}(\gamma_\ominus) - E(u) \right| \leq \int_{\R \times S^1} \left| \frac{\partial G}{\partial s}(s,t,u_a(s,t)) \right| ds dt <  \epsilon,
\end{equation}
where the last inequality is a direct computation using the definition of $G$. 
Since $E(u)$ is positive, it follows that $\mathcal{A}_{H_\ominus}(\gamma_\ominus) \leq \mathcal{A}_{H_\oplus}(\gamma_\oplus) +\epsilon $. To define the continuation map $\Psi_{G,J^s_t}$ associated to $(G,J^s_t)$ we must take $C^\infty$-small perturbations of $(G,J^s_t)$ supported in $s \in [-2,-1]$ such that the relevant moduli spaces are regular. We assume that these perturbations are taken so that \eqref{eq:estimate} is still valid.

From this, if $\gamma_\oplus \in \mathcal{P}^1_{a}(H)= \mathcal{P}^1_{a- 2\epsilon}(H)$, then all $1$-periodic orbits appearing in the expression of   $\Psi_{G,J^s_t}(\gamma_\oplus)$ have action $< a$. It follows that $\Psi_{G,J^s_t}(CF^a(H_\oplus)) \subset CF^{a-\epsilon}(H_\ominus) $, and we we obtain a map $$ \Psi_G^a: CF^{a - 2\epsilon}(H_\oplus) \simeq CF^{a-\epsilon }(H_\oplus) \to CF^{a}(H_\ominus)$$ which descends to a map on the homologies. A similar argument implies that  $\Psi_{G,J^s_t}(CF^b(H_\oplus)) \subset CF^{b+\epsilon}(H_\ominus)$ and we obtain a map $$\Psi_{G}^b: CF^{b}(H_\oplus)  \to CF^{b+\epsilon}(H_\ominus)$$ which also descends to a map on the homologies. The fact that \linebreak $\Psi_{G,J^s_t}(CF^a(H_\oplus)) \subset CF^{a-\epsilon}(H_\ominus) $ and $\Psi_{G,J^s_t}(CF^b(H_\oplus)) \subset CF^{b+\epsilon}(H_\ominus)$ implies that  $\Psi_{G}$ induces a map $$ \Psi_G:  CF^{(a,b)}(H_\oplus) \to CF^{(a-\epsilon, b+\epsilon)}(H_\ominus).$$

The construction of the maps 
\begin{align}
& \Psi_{\widehat{G}}^b: CF^{b+ \epsilon}(H_\ominus) \to CF^{b+2\epsilon}(H_\oplus) \simeq CF^{b}(H_\oplus), \\   
& \Psi_{\widehat{G}}^a: CF^{a-\epsilon  }(H_\ominus) \to CF^{a}(H_\oplus) \simeq CF^{a-2\epsilon}(H_\oplus) , \\
& \Psi_{\widehat{G}}:  CF^{(a-\epsilon, b+\epsilon)}(H_\ominus) \to CF^{(a-2\epsilon,b+2\epsilon)}(H_\oplus) \simeq  CF^{(a,b)}(H_\oplus),
\end{align}
 follow the same strategy using an estimate similar to \eqref{eq:estimate} for the homotopy~$\widehat{G}$.

We now explain how to show that  $\Psi_{\widehat{G}}^b\circ \Psi_G^b: CF^{b}(H_\oplus)   \to CF^{b+ 2\epsilon}(H_\oplus) \simeq CF^{b }(H_\oplus) $ and $\Psi_{\widehat{G}}^a\circ \Psi_G^a: CF^{a}(H_\oplus) \simeq CF^{a - 2\epsilon}(H_\oplus)  \to CF^{a}(H_\oplus)$ are chain homotopic to the identity. 
The idea is to construct a homotopy of homotopies from the concatenation of $G$ and $\widehat{G}$ to the trivial homotopy between $H_\oplus$ and itself, and obtain the chain homotopy studying moduli spaces for the homotopy of homotopies: this is the usual method of proving the invariance of Floer homology, and it was devised by Floer. 
In our situation we want the chain homotopy to respect certain action windows, and this requires the construction of a special homotopy of homotopies $Q_a: \R \times S^1 \times \Sigma \to \R$ where  $a\in [0,+\infty)$.

To define, for for each $a \in [0,1]$,  the function $Q_a: \R \times S^1 \times \Sigma \to \R$ we first introduce a smooth auxiliary function 
\begin{equation} \label{eq:sigma}
    \sigma:[0,1] \to [0,1]
\end{equation}
such that 
\begin{eqnarray*}
    \sigma \equiv 0 \mbox{ on a neighbourhood of } 0, \\
    \sigma \equiv 1 \mbox{ on a neighbourhood of } 1.
\end{eqnarray*}
We then define
\begin{equation} \label{eq:Qasmall}
    Q_a(s,t,p) := H_\oplus(t,p) + \sigma(a) (\tau(s) F_t((\phi_\oplus^t)^{-1}(p)))
\end{equation}

We remark that the functions $Q_a$ depend smoothly on $a$, and that
for each $a \in [0,1]$, the function $Q_a$ is a homotopy from $H_\oplus$ to itself. So we refer to $(Q_a)_{a \in [0,1]}$ as a homotopy of homotopies. 
It is immediate from the definitions that $Q_0(s,t,p)=H_\oplus(t,p)$ is the trivial homotopy $H_\oplus$ to itself.

We now proceed to define, for $a \in [1,+\infty)$, the functions $Q_a: \R \times S^1 \times \Sigma \to \R$. We first let $\chi:[1,+\infty) \to [1,+\infty)$ be a smooth increasing function satisfying
\begin{align}
    \chi(1)=1 \mbox{ and all derivatives of } \chi \mbox{ vanish at } 1, \\
    \chi \mbox{ equals the identity outside a neighbourhood of } 1.
\end{align}
We then let, for each $a\in [1,+\infty)$, $Q_a$ be defined by 
\begin{align} \label{eq:Qabig}
    Q_a(s,t,p) := G(s+\chi(a),t,p) \mbox{ for } s \leq -\chi(a), \\
    Q_a(s,t,p) := H_\ominus(t,p) \mbox{ for } s \in [-\chi(a),\chi(a)], \\
    Q_a(s,t,p) := \widehat{G}(s-\chi(a),t,p) \mbox{ for } s \geq \chi(a).
\end{align}

It is clear from the definitions that for each $a\in [0,+\infty)$
\begin{align}
    Q_a(s,t,p) := H_\oplus(t,p) \mbox{ if } |s| \geq  \chi(a) +2.
\end{align}
Therefore each $Q_a$ is a homotopy from $H_\oplus$ to itself, and we can think of $(Q_a)_{a\in [0,+\infty)}$ as a homotopy of homotopies. 

We remark that 
\begin{align}
  \label{eq:convergenceatinfty}    Q_a(s + \chi(a),t,\cdot) \mbox{ converges to } \widehat{G}(s,t, \cdot)  \mbox{ in } C^\infty_{loc} \mbox{ as } a \to +\infty, \\
  \label{eq:convergenceatinfty2}  Q_a(s - \chi(a),t, \cdot) \mbox{ converges to } G(s,t,\cdot)  \mbox{ in } C^\infty_{loc} \mbox{ as } a \to +\infty.
\end{align}
So we can indeed think of $Q_{+\infty}$ as the concatenation of $G$ and $\widehat G$.

Choosing an appropriate homotopy $(J^s_t(a))_{a \in [0,+\infty)}$ of almost complex structures  and applying the usual technique in Floer homology to show its invariance, the pair of homotopies $(Q_a,J^s_t(a))$ will induce a map $S: CF(H_\oplus) \to CF(H_\oplus)$ which satisfies: 
$$\Psi_{\widehat{G},\widehat{J}^s_t} \Psi_{G,J^s_t} = \mathrm{id} + S\circ d + d \circ S,$$
where $d$ is the differential of $CF(H_\oplus)$. The map $S$ counts Floer cylinders of index $-1$ for $(Q_a,J^s_t(a))$ for the values of $a\in [0,+\infty)$ on which the moduli spaces of Floer cylinders for $(Q_a,J^s_t(a))$ are not regularly cut out. 

The main thing to be observed is that for every $a\in [0,+\infty)$ the Floer cylinders of $(Q_a,J^s_t(a))$ satisfy an estimate similar\footnote{ In order to obtain transversality, one might need to perturb the homotopies $(Q_a,J^s_t(a))$, but it is clear that for sufficiently small perturbation the inequality \eqref{eq:estimatehomotopy} will still hold.  } to the one in \eqref{eq:estimate}. More precisely if $u$ Floer cylinder used in the definition of $S$ and $\gamma$ is the negative asymptotic limit of $u$ and $\gamma'$ is its positive asymptotic limit, than we have
\begin{equation} \label{eq:estimatehomotopy}
\mathcal{A}_{H_\oplus}(\gamma') - \mathcal{A}_{H_\oplus}(\gamma) < \epsilon.   
\end{equation}
We conclude that $S(CF^{a-2\epsilon}(H_\oplus)) \subset CF^{a}(H_\oplus)$, and since $CF^{a-2\epsilon}(H_\oplus) = CF^{a}(H_\oplus)$, the map $S$ induces a map $S^a: CF^{a}(H_\oplus) \to CF^{a}(H_\oplus) $, which satisfies $$\Psi_{\widehat{G}}^a\circ \Psi_G^a = \mathrm{id} + S^a\circ d + d \circ S^a. $$ 

A similar argument shows that $S$ induces a chain homotopy between  $\Psi_{\widehat{G}}^b\circ \Psi_G^b$ and the identity. Once this is achieved, it is an elementary algebraic fact that $$\Psi_{\widehat{G}} \circ \Psi_G: CF^{(a,b)}(H_\oplus) \to CF^{(a,b)}(H_\oplus) $$  
is chain homotopic to the identity.

\qed

We will also need to understand maps $\Psi_G$ and $\Psi_{\widehat{G}}$ geometrically. Since these maps are induced by the continuation maps $\Psi_{G,J^s_t}$ and $\Psi_{\widehat{G},\widehat{J}^s_t}$ this can be obtained via the definition of continuation maps in \eqref{eq:continuation}.
Indeed, it follows from the definitions of these maps, that if $\gamma_{\oplus} \in \mathcal{P}^1_{(a,b)}(H_\oplus)$ then 
\begin{equation} \label{eq:psiG}
\Psi_G(\gamma_\oplus) = \sum_{\gamma \in \mathcal{P}^1_{(a-\epsilon,b+\epsilon)}(H_\ominus)} K_{G,J^s_t}(\gamma_\oplus,\gamma)\gamma.    
\end{equation}
Similarly, if $\gamma_{\ominus} \in \mathcal{P}^1_{(a-\epsilon,b+\epsilon)}(H_\ominus)$ then 
\begin{equation} \label{eq:psiGhat}
    \Psi_{\widehat{G}}(\gamma_\ominus) =  \sum_{\gamma \in \mathcal{P}^1_{(a,b)}(H_\oplus)} K_{\widehat{G},\widehat{J}^s_t}(\gamma_\ominus,\gamma)\gamma. 
\end{equation}
For \eqref{eq:psiGhat}, we are using that $\mathcal{P}^1_{(a,b)}(H_\oplus) = \mathcal{P}^1_{(a-2\epsilon,b+2\epsilon)}(H_\oplus)$.

\subsubsection{The case $\Sigma = T^2$}\label{sec:Hamontorus}

Let now  $\Sigma=T^2$ be the two-torus. Let $\omega$ be a symplectic form on $T^2$.
If we only consider contractible loops on $T^2$, the definition of chain complexes $CF^{(a,b)}_{[.]}(H)$ and homologies $HF^{(a,b)}_{[.]}(H)$ can be given exactly as discussed above. Here we indicate the restriction to contractible loops by a lower index $[.]$. In the case of non-contractible loops the construction has to be adapted due to the non-uniqueness of (homotopy classes of) capping cylinders. 

As above, we fix a representative  $\eta_{\alpha}$ for each free homotopy class $\alpha$ of loops in $\Sigma = T^2$, as well as a symplectic trivialization $\Phi_{\alpha}$ of $\eta_{\alpha}^*T T^2$. 
Let $\alpha$ be a non-trivial free homotopy class of loops in $T^2$. 
Let $y: S^1 \to T^2$ be a loop and $\Cyl_y:[0,1] \times S^1 \to T^2$ with $\Cyl_y(0,t) = \eta_{\alpha}(t)$ and $\Cyl_y(1,t) = y(t)$.  Denote by $[\Cyl_y]$ the homotopy class of $\Cyl_y$, where the homotopy may vary among such cylinders.   
For a Hamiltonian $H:S^1 \times T^2 \to \R$ define the action of the pair $(y,[\Cyl_y])$ to be 
\begin{equation}
    \mathcal{A}_H(y,[\Cyl_y]) := -\int_{\Cyl_y} \omega - \int_0^1 H(t,y(t))dt.
\end{equation}

The action is well-defined by Stokes' theorem. Moreover, gluing two cylinders $\Cyl_y$ and $\Cyl'_y$ along $\eta_y$ we obtain a map from $T^2$ to $T^2$ from which it follows that the action difference $\mathcal{A}_H(y,[\Cyl_y]) - \mathcal{A}_H(y,[\Cyl'_y])$ is a multiple of $\int_{T^2}\omega$. 

Given $-\infty \leq a<b\leq +\infty$ we let $\mathcal{P}^1_{(a,b); \alpha}(H)$ be the set of pairs $(\gamma,[\Cyl_\gamma])$ consiting of a $1$-periodic orbit $\gamma$ of $\phi^t_H$ representing $\alpha$ and a homotopy class of cylinders  $[\Cyl_\gamma]$ connecting $\eta_{\alpha}$ and $\gamma$ such that $\mathcal{A}_H(\gamma,[\Cyl_\gamma])\in (a,b) $.
The $1$-periodic spectrum $\mathrm{Spec}_\alpha^1(H)$ is the set of all possible actions $\mathcal{A}_H(\gamma,[\Cyl_\gamma])$ of pairs  $(\gamma,[\Cyl_\gamma]) \in  \mathcal{P}^1_{(-\infty,\infty); \alpha}(H)$.

While the Conley-Zehnder index of a $1$-periodic orbit $\gamma$ of $\phi_H$ in class $\alpha$ is not well-defined, 
it is well-defined when fixing a homotopy class $[\Cyl_\gamma]$ of cylinders connecting $\eta_\alpha$ and $\gamma$  via a symplectic trivialization of $\gamma^*T T^2$ that we obtain by extension over $\Cyl_\gamma$ of the fixed trivialization $\Phi_\alpha$ of $\eta_{\alpha}^*TT^2$. We denote it by $\mu_{\rm CZ}(\gamma, [\Cyl_\gamma])$.

Fix now (finite) real numbers $a<b$. Suppose now that $\phi_H$ is a non-degenerate Hamiltonian diffeomorphism of $(T^2,\omega)$.
We define $$ CF_\alpha^{(a,b)}(H):= \bigoplus_{(\gamma, [\Cyl_\gamma]) \in \mathcal{P}^1_{(a,b); \alpha}(H)}\Z_2 \cdot (\gamma, [\Cyl_\gamma]).$$

We choose a smooth $S^1$-family $J_t$ of compatible almost complex structures on $(T^2,\omega)$. The Floer equation of $(H,J_t)$ is defined as \eqref{eq:defFloereq}.  Given $1$-periodic orbits $(\gamma, [\Cyl_\gamma])$ and $(\gamma', [\Cyl_{\gamma'}])$ in $\mathcal{P}^1_{(a,b);\alpha}(H)$  we let $$\mathcal{M}(\gamma, [\Cyl_\gamma],\gamma', [\Cyl_{\gamma'}],H,J_t)$$ be the moduli space whose elements are Floer cylinders $u$ of $(H,J_t)$ negatively asymptotic to $\gamma$ and positively to $\gamma'$, and such that the gluing $\Cyl_\gamma \# u$ is homotopic to $\Cyl'_\gamma$. As previously, if for two Floer cylinders $u_1$ and $u_2$ of $(H,J_t)$  there is an $s_0$ that $u_1(s_0 + \cdot,\cdot)=u_2(\cdot,\cdot)$ then $u_1$ and $u_2$ represent the same element in the moduli space.

We let 
$$ C(\gamma, [\Cyl_\gamma], \gamma', [\Cyl'_\gamma]) = \#\mathcal{M}(\gamma, [\Cyl_\gamma], \gamma', [\Cyl_{\gamma'}], H, J_t)\, \,  \mod 2$$
if $\mu_{\rm CZ}(\gamma, [\Cyl_\gamma]) -1 = \mu_{\rm CZ}(\gamma', [\Cyl'_\gamma])$, and $C(\gamma, [\Cyl_\gamma], \gamma', [\Cyl_{\gamma'}]) = 0$ otherwise. We define 
$d:CF_\alpha^{(a,b)}(H) \to CF_\alpha^{(a,b)}(H)$ by letting 
$$ d(\gamma, [\Cyl_\gamma]) = 
 \sum_{(\gamma', [\Cyl_{\gamma'}]) \in \mathcal{P}^1_{(a,b);\alpha}(H)} C(\gamma,[\Cyl_\gamma],\gamma', [\Cyl_{\gamma'}]) \cdot (\gamma', [\Cyl_{\gamma'}] )$$
 for the generators and extending it to all of $CF^{(a,b)}_\alpha(H)$. 

Breaking of Floer cylinders for $(H,J^t)$ connecting some $(\gamma,[\Cyl_\gamma])$ and $(y',[\Cyl_{\gamma'}])$ appears at pairs $(\hat{\gamma}, [\hat{\Cyl}_\gamma])$ with action in the action interval $(\mathcal{A}_H({\gamma}, [\Cyl_\gamma]), \mathcal{A}_H({\gamma'}, [{\Cyl}_{\gamma'}]))$, see definition \ref{def:breakingpair} in section \ref{sec:Hamonsphere} and the discussion there in the situation of the sphere.
One shows that $d^2 = 0$ and defines $HF_\alpha^{(a,b)}(H)$ to be the homology of the chain-complex $(CF^{(a,b)}_\alpha(H),d)$.

Moreover, one can extend the definition and  properties of continuation maps from section \ref{sec:Hamonhighergenus} to the present situation, and in particular the following Proposition, an analogue to Proposition \ref{prop:actionwindowsclosedsurface}. 

\begin{prop} \label{prop:chainisotorus}
Let $\alpha$ be a free homotopy class of loops in $T^2$.
Let $\phi_\oplus$ be a non-degenerate normalized Hamiltonian diffeomorphism in $\mathcal{H}(T^2,\omega)$ and $H_\oplus:S^1 \times \Sigma \to \R$ be a Hamiltonian in $\mathrm{Ham}(T^2,\omega)$ generating $\phi_\oplus$. We take real numbers $a<b$ with $a,b \notin \mathrm{Spec}_\alpha^1(H_\oplus)$. Let $\epsilon>0$ be such that $\mathrm{Spec}_\alpha^1(H_\oplus) \cap  (a-2\epsilon, b+2\epsilon) \subset  (a,b)$. Let $\phi_\ominus$ be a Hamiltonian diffeomorphism with $d_{\rm Hofer}(\phi_\oplus, \phi_\ominus)< \epsilon$. Then, there exist a normalized Hamiltonian $H_\ominus: S^1 \times T^2 \to \R $ and homotopies $G: \R \times S^1 \times T^2 \to \R $ between $H_\oplus$ and $H_\ominus$ and $\widehat{G}: \R \times S^1 \times T^2 \to \R $ between $H_\ominus$ and $H_\oplus$ 
which induce continuation maps 
$$ \Psi_G:  CF_\alpha^{(a,b)}(H_\oplus) \to CF_\alpha^{(a-\epsilon, b+\epsilon)}(H_\ominus) $$ and $$\Psi_{\widehat{G}}:  CF_\alpha^{(a-\epsilon, b+\epsilon)}(H_\ominus) \to CF_\alpha^{(a-2\epsilon,b+2\epsilon)}(H_\oplus) \simeq CF_\alpha^{(a,b)}(H_\oplus),$$ such that the composition $\Psi_{\widehat{G}}\circ \Psi_G$ is chain homotopic to the identity map ${\rm id}:  CF_\alpha^{(a,b)}(H_\oplus) \to  CF_\alpha^{(a,b)}(H_\oplus) $.
\end{prop}

Since we directly define the chain complexes $CF^{(a,b)}(H)$ without a quotient construction, the chain maps in the proposition have to be defined directly on those chain complexes, and hence the proof of this Proposition varies a bit from our proof of Proposition  \ref{prop:actionwindowsclosedsurface}. We refer to the proof of Proposition \ref{prop:chainisodisk0}, where these adaptions are explained.

\subsubsection{The quasi-isolation property}\label{sec:quasi-iso}

We will define a property for a set of periodic orbits which we call $\epsilon$-quasi-isolation and discuss one consequence for continuation maps which we need when replacing isolation with quasi-isolation in the assumptions of the main theorems. 

Consider as a above a closed symplectic surface $(\Sigma, \omega)$ with $\Sigma \neq S^2$  (The case for  $S^2$ and $D^2$ works analogously as soon as the relevant Floer theory is defined, see sections \ref{sec:Hamonsphere} and \ref{sec:Hamsondisk})
Let $H:\Sigma \times S^1 \to \R$ be a non-degenerate Hamiltonian, and $\phi= \phi^1_H$ the Hamiltonian diffeomorphism that is generated by $H$. Let $J = J_t$ be a $S^1$-family of compatible almost complex structures. Let $\epsilon>0$. 
\begin{defn}\label{defn:quasi-isolation}
 We say that a finite set $\mathcal{Y} =\{\gamma_1, \ldots, \gamma_k\}$ of $1$-periodic orbits for $H$ that all represent the same free homotopy class $\alpha$ 
 is  \textit{$\epsilon$-quasi-isolated} (with respect to $J$) 
 if 
 \begin{enumerate}[a)]
     \item for any $i,j \in \{1, \ldots, k\}$, 
     $\Delta_H(\gamma_i, \gamma_j)$ is either equal to $0$ or $\geq \epsilon$,
         \item and there is no non-constant 
 $u:\R \times S^1 \to \Sigma$ and no $i \in \{1, \ldots, k\}$ such that 
 \begin{itemize}
     \item $\mathcal{F}_{H,J}(u) = 0$ 
     \item $E(u) < \epsilon$
     \item $\lim_{s\to \infty}u(s,t)  = \gamma_{i}(t)$ or $\lim_{s\to -\infty}u(s,t)  = \gamma_{i}(t)$. 
     \end{itemize}
     \end{enumerate}
\end{defn}

Let $\alpha$ be a free homotopy class of loops in $\Sigma$. Let now $H_{\oplus}$ be a non-degenerate Hamiltonian that generates $\phi_{\oplus}$. 
Let now $\mathcal{Y}=\{\gamma_1, \ldots, \gamma_k\}$ be a set of  $1$-periodic orbits for $H_{\oplus}$ in class $\alpha$ which is $2\epsilon$-quasi-isolated for a regular family of compatible almost complex structures $J^{\oplus}_t$.  
Assume additionally here that $\mathcal{A}_{H_{\oplus}}(\gamma_1)= \ldots = \mathcal{A}_{H_{\oplus}}(\gamma_k)$ (resp. $\mathcal{A}_{H_{\oplus}}(\gamma_1, [w_{\gamma_1}]) = \ldots =  \mathcal{A}_{H_{\oplus}}(\gamma_k, [w_{\gamma_k}])$ for some suitable cylindrical cappings) and denote this action value by $\kappa$. Then, for any $\epsilon'\leq 2\epsilon$,  the vector space $B_{\mathcal{Y}}\subset CF_{\alpha}^{(\kappa-\epsilon', \kappa+\epsilon')}(H_{\oplus})$ generated by $\gamma_1, \ldots, \gamma_k$ (resp. $(\gamma_1,[w_{\gamma_1}]), \ldots (\gamma_k,[w_{\gamma_k}])$) defines obviously a subcomplex of $(CF_{\alpha}^{(\kappa-\epsilon', \kappa+\epsilon')}(H_{\oplus}), d^{J^{\oplus}})$. 
Furthermore, we have 

\begin{prop} \label{prop:quasi_iso}
Let $\phi_\ominus$ be a non-degenerate  Hamiltonian diffeomorphism with  $d_{\rm Hofer}(\phi_\oplus, \phi_\ominus)< \epsilon$. Then there exist a normalized Hamiltonian $H_\ominus: S^1 \times \Sigma \to \R $  that generates $\phi_{\ominus}$ and homotopies $G: \R \times S^1 \times \Sigma \to \R $ between $H_\oplus$ and $H_\ominus$ and $\widehat{G}: \R \times S^1 \times \Sigma \to \R $ between $H_\ominus$ and $H_\oplus$ 
which induce chain maps 
$$\Psi^{\mathcal{Y}}_G: \mathcal{B}_{\mathcal{Y}}  \to  CF_\alpha^{(\kappa-{\epsilon},\kappa+{\epsilon})}(H_\ominus) $$ and $$\Psi^{\mathcal{Y}}_{\widehat{G}}: CF_\alpha^{(\kappa-{\epsilon},\kappa+{\epsilon})}(H_\ominus)  \to \mathcal{B}_{\mathcal{Y}} $$
such that $$\Psi^{\mathcal{Y}}_{\widehat{G}}\circ {\Psi^{\mathcal{Y}}_G} = \id.$$
\end{prop}

\begin{proof}
For convenience of notation we assume that $\Sigma$ is a closed surface different from $S^2$: if $\Sigma = T^2$ we assume moreover that $\alpha$ is the trivial free homotopy class. For the remaining cases one has to replace below the orbits $\gamma$ by pairs $(\gamma,w_\gamma)$. 
We keep the construction for $H_{\ominus}$, $G$ and $\widehat{G}$, almost complex structure $J^s_t$, $\widehat{J^s_t}$ as in the proof of Proposition \ref{prop:actionwindowsclosedsurface}. We require that $J_t^s = J_t^{\oplus}$ resp. $\widehat{J^s_t} = J_t^{\oplus}$ for $s$ sufficiently small resp. sufficiently large. Also, as in that proof define the moduli spaces  $\mathcal{M}(\gamma_{\oplus}, \gamma_{\ominus}, G,J^s_t)$
and $\mathcal{M}(\gamma_{\ominus}, \gamma_{\oplus}, \widehat{G},\widehat{J^s_t})$ for $1$-periodic orbits $\gamma_{\oplus}$ for $H_{\oplus}$ and $\gamma_{\ominus}$ for $H_{\ominus}$. 
Let $K_{G,J^s_t}(\gamma_{\oplus},\gamma_{\ominus}) := (\# \mathcal{M}(\gamma_{\oplus}, \gamma_{\ominus}, G,J^s_t) \, \mod 2)$, if $\mu_{\rm CZ}(\gamma_{\oplus}) = \mu_{\rm CZ}(\gamma_{\ominus})$ and $0$ otherwise. 
Define $\Psi^{\mathcal{Y}}_G: \mathcal{B}_{\mathcal{Y}}  \to  CF_\alpha^{(\kappa-{\epsilon},\kappa+{\epsilon})}(H_\ominus) $ via 
\begin{equation*}
\Psi^{\mathcal{Y}}_G(\gamma_i) = \sum_{\gamma' \in \mathcal{P}^1_{(\kappa-\epsilon,\kappa+\epsilon)}(H_{\ominus})} K_{G,J^s_t}(\gamma_i,\gamma) \gamma.
\end{equation*}

Define similarly $K_{\widehat{G},\widehat{J^s_t}}(\gamma_{\ominus},\gamma_{\oplus})$, 
and then 
$\Psi^{\mathcal{Y}}_{\widehat{G}}: CF_\alpha^{(\kappa-{\epsilon},\kappa+{\epsilon})}(H_\ominus)  \to \mathcal{B}_{\mathcal{Y}}$ 
via
\begin{equation*}
\Psi^{\mathcal{Y}}_{\widehat{G}}(\gamma) = \sum_{\gamma' \in \mathcal{Y}}  K_{\widehat{G},\widehat{J^s_t}}(\gamma,\gamma')\gamma'.
\end{equation*}

$\Psi^{\mathcal{Y}}_G$ and $\Psi^{\mathcal{Y}}_{\widehat{G}}$ are chain maps. To see that $\Psi^{\mathcal{Y}}_G$ is a chain map, let $\gamma_i \in \mathcal{Y}$ and $\gamma'_1,\ldots \gamma'_l \in \mathcal{P}_{(\kappa-\epsilon,\kappa+\epsilon)}(H_{\ominus})$ the orbits such that the $1$-dimensional moduli space $\mathcal{M}(\gamma_i, \gamma'_j, H, J^s_t)$ are non-empty. These moduli-spaces can be compactified, where the boundary components consist of broken Floer trajectories, and by the $2\epsilon$-quasi-isolation property and the action estimate \eqref{eq:estimate} these broken Floer trajectories are exactly those that contribute to $d \circ \Psi^{\mathcal{Y}}_G(\gamma_i)$. Since a complact one-dimensional manifold has an even number of boundary components, it follows that $d \circ \Psi^{\mathcal{Y}}_G = 0$.

Similarly one sees that $\Psi^{\mathcal{Y}}_{\widehat{G}}$ is a chain map. 

Finally, by action estimate \eqref{eq:estimatehomotopy} for Floer cylinders associated to the pair of homotopies $(Q_a, J^s_t)$ defined above and the $2\epsilon$-quasi-isolation property one observes that 
$$\Psi^{\mathcal{Y}}_{\widehat{G}}\circ {\Psi^{\mathcal{Y}}_G} = \id.$$

\end{proof}




 
 
\subsection{Hamiltonian dynamics and Floer homology on $\D$} \label{sec:Hamsondisk}

\subsubsection{Hamiltonian diffeomorphisms which are an irrational rotation near the $\partial \overline{\D}$.}  \label{sec:hamrotnearbound}

\

We consider  time-dependent Hamiltonians $H:S^1 \times \D \to \R$ on the disc, equipped with the standard symplectic form $\omega_0 := dx \wedge dy$, where $(x,y)$ are the coordinates of $\D$.
For $c \in \R_{+}$ we say  that $H$ is admissible with slope $c$ near $\partial \D$, if  there is $r_0>0$ close to $1$ 
such that in polar coordinates $ (r,\theta)$, 
\begin{itemize}
 \item $H(t,r,\theta) =  \frac{1}{2}c(r^2-1) $, for $r_0 \leq r\leq 1$, 
\end{itemize}
Admissible $H$ generate Hamiltonian diffeomorphisms $\phi: \D \to \D$.
For $c \in \R_+$, we denote by $\mathcal{H}_c(\D)$ the set of admissible Hamiltonians on $\D$ with slope $c$.
We say that $H$ is non-degenerate, if $\phi$ is non-degenerate in $\D$, and say that $H$ is strongly non-degnerate if $\phi^k$ is non-degenerate in $\D$ for all $k\in \N$.  For this to hold,  $c \in \R_+ \setminus 2\pi \Q$. 

For $c \in \R_+ $, let $\phi^0_c$ be the time $1$-map of the Hamiltonian $c(x^2 + y^2 -1)$. We define the set of Hamiltonian diffeomorphisms\footnote{Recall that every symplectomorphism of $\D$ is Hamiltonian.} $\mathrm{Ham}_c(\D)$ to be the set of Hamiltonian diffeomorphisms of $\D$ which coincide with $\phi^0_c$ on a neighbourhood of $\partial \D$.

We then have the following lemma:
\begin{lem} \label{lem:felix}
Let $c \in \R_+$.
For each element $\phi \in \mathrm{Ham}_c(\D)$ there exists a unique element of $\mathcal{H}_c(\D)$ whose time $1$-map is $\phi$.
\end{lem}
We thank Felix Schlenk for explaining to us the proof of this lemma. We give a sketch of the argument and leave it up to the reader to complete it.

\textit{Sketch of proof:}

The lemma will clearly follow if we can show that for each compactly supported Hamiltonian diffeormosphism of $\D$ there is a compactly supported time-dependent Hamiltonian $H:S^1 \times \D \to \R$ whose time $1$-map is $\phi$. 

To construct the Hamiltonian $H$ one proceeds as follows. We first construct a smooth isotopy $(f_t)_{t \in [0,1]}$ of diffeomorphisms of the disk with $f_0 = \mathrm{id}$ and $ f_1=\phi$, and with all $f_t$ supported on a fixed compact $K_{\phi}$ of $\D$: such an isotopy can be constructed via Alexander's trick. 

Using Moser's homotopy method this path can be changed into a path of area preserving maps, where the end-points of the new path are still $\mathrm{id}$ and $\phi$ and the elements of the new path still have compact support in $K_{\phi}$: the reason for this is that the vector field $X_t$ given by Moser's method vanishes where the map was already area preserving.

For each $t \in S^1$, we let $H_t$ be the the unique compactly supported function in $\D$ whose Hamiltonian vector-field is $X_t$. The time-dependent Hamiltonian $H(t,p)= H_t(p)$ is the desired one. \qed


%


We proceed to introduce some terminology. We say that a time dependent Hamiltonian $H: S^1 \times \D \to \R$ \textit{vanishes near the boundary} if there exists a compact subset $K$ of the open disk $\D$ such that $H_t$ vanishes outside $K$ for every $t\in S^1$.

Given elements $\phi_1$ and $\phi_2$ of $\mathrm{Ham}_c(\D)$, the Hofer distance $d_{\mathrm{Hofer}}(\phi_1,\phi_2)$ is defined by the formula \begin{equation}
     d_{\rm Hofer}(\phi_1,\phi_2) =  \inf \int_{0}^1 || H_t ||dt,
 \end{equation}
 where the infimum is taken over all $H:S^1 \times \D \to \R$ that vanish near the boundary and generate $\phi_1^{-1}\circ \phi_2$ as time $1$-map, and where $|| H_t ||:= \max_{p \in \D} H_t  - \min_{p \in \D} H_t$. 

In order to use Floer theory for an element $H$ of $\mathcal{H}_c(\D)$
we extend $H$ to $\R^2$ by  letting ${H}(t,x) = \frac{1}{2}c(r^2-1)$ for $r\geq r_0$. For all the iterates of time $1$-map $\phi_H^1$ of $X_H$ to be non-degenerate, it is necessary that $c \in \R \setminus 2 \pi \Q$, since otherwise every point outside of $\D$ would be a periodic point of $X_H$. 
We thus assume from now on that $c \in \R \setminus 2 \pi \Q$. This implies that all periodic points of $\phi_H^1$ are contained in a compact subset of the open disk $\D$. We say that an element $\phi$ of $\mathrm{Ham}_c(\D)$ is strongly non-degenerate if all periodic points of $\phi$ are strongly  non-degenerate.
Under our assumptions, the set of strongly non-degenerate  elements $\phi \in \mathrm{Ham}_c(\D)$ is $C^\infty$-dense in  $\mathrm{Ham}_c(\D)$. 

The Hamiltonian action $\mathcal{A}_{{H}}(y)$ of a loop $y:S^1 \to \R^2$ is defined as in \eqref{eq:defaction}.
We consider an $S^1$-family $J_t$ of compatible almost complex structures on $\R^2$ that coincide with the complex multiplication by $\mathrm{i}$ outside of $\R^2$.
The Floer equation for $(H,J_t)$ applied to a cylinder $u: \R \times S^1 \to \R^2$ is
$$ \mathcal{F}_{{H},J}(u) = \partial_s u(s,t) +  J_t(u(s,t))\big(\partial_t u (s,t) - X_{{H}}(t,u(s,t)\big)=0.
$$

For two $1$-periodic orbits $\gamma$ and $\gamma'$ of ${H}$ 
we denote $\mathcal{M}(\gamma, \gamma', {H},J_t)$ the moduli space of solutions $u: \R \times S^1 \to \R^2$, $\mathcal{F}_{\widehat{H},J}(u) = 0$ with asymptotics $
\lim_{s\to -\infty}u(s,\cdot) = \gamma(\cdot)$, and $\lim_{s\to +\infty}u(s,\cdot) = \gamma'(\cdot)$.  


In order to do Floer theory for admissible Hamiltonians on $\D$, we need to obtain compactifications of the relevant moduli spaces. The crucial step for this to work is to show that all Floer trajectories in $\mathcal{M}(\gamma, \gamma', \widehat{H},J_t)$ stay inside $D^2$; see e.g. \cite{FloerHofersymphom} or \cite[Lemma 2.2]{CieliebakOanceaSH}.
Once this is shown, one can apply the techiniques of \cite{Floer} to compactify the moduli space $\mathcal{M}(\gamma, \gamma', \widehat{H},J_t)$. Once this has been observed, one can construct the Floer homology of a Hamiltonian in $H$ in $\mathrm{Ham}_c(\D)$ as in Section \ref{sec:Hamonclosedsurfaces}. We define $CF^a(H)$, $CF^{(a,b)}(H)$ and the differential $d^{J_t}$ exactly as in Section \ref{sec:Hamonclosedsurfaces}. The homologies $HF^{a}(H)$ and $HF^{(a,b)}(H)$ are those of the pairs  $(CF^a(H),d^{J_t}) $ and  $(CF^{(a,b)}(H),d^{J_t})$, respectively.


Given now two admissible non-degenerate Hamiltonians $H_{\oplus}$ and $H_{\ominus}$ in $\mathcal{H}_c(\D)$, we let $J_{t}$ be such that all Floer cylinders with finite energy of  $(H_\oplus,J_t)$ and $(H_\oplus,J_t)$ are Fredholm regular.
A homotopy $G: \R \times S^1 \times \D$ of Hamiltonians between $H_{\oplus}$ and $H_{\ominus}$ is called admissible if there exists a compact subset $K$ of the open disk $\D$ such that for every $s\in \R$ the function $G_s$ coincides with $c(r^2 - 1)$ on the complement of $K$.
 
Given an admissible homotopy $Q: \R \times S^1 \times \D$ between $H_\oplus$ and $H_\ominus$ and regular compatible $\R \times S^1$-dependent family $J^s_t$,  we consider the moduli spaces composed of Floer cylinders from $1$-periodic orbits of $H_\oplus$ to $1$-periodic orbits of $H_\ominus$. Again, the almost complex structures $J^s_t$ are assumed to coincide with complex multiplication outside $\D$ for all $s$ and $t$. This forces the images of all relevant Floer cylinders to be contained in $\D$ and is the crucial step that allows us to compactify these moduli spaces. For $C^\infty$-generic pairs $(Q,J^s_t)$ one obtains a continuation map $\Psi_{Q,J^s_t}: CF(H_\oplus) \to CF(H_\ominus)$ which passes to a homology map $\Psi_{Q,J^s_t}: HF(H_\oplus) \to HF(H_\ominus)$.
The proof of the following proposition is identical to the one of Proposition \ref{prop:actionwindowsclosedsurface}.

\begin{prop} \label{prop:chainisodiskc}
Let $\phi_\oplus$ be a non-degenerate Hamiltonian diffeomorphism in $\mathcal{H}_c(\D)$ and $H_\oplus:S^1 \times \Sigma \to \R$ be a Hamiltonian in $\mathrm{Ham}_c(\D)$ generating $\phi_\oplus$. We take real numbers $a<b$ which do not belong to $\mathrm{Spec}^1(H_\oplus)$ and let $\epsilon>0$ be such that all elements of $\mathrm{Spec}^1(H_\oplus)$  in the interval $(a-2\epsilon, b+2\epsilon)$ are contained in $(a,b)$. Let $\phi_\ominus$ be a non-degnerate Hamiltonian diffeomorphism with $d_{\rm Hofer}(\phi_\oplus, \phi_\ominus)< \epsilon$. Then, there exist a normalized Hamiltonian $H_\ominus: S^1 \times \D \to \R $ and homotopies $G: \R \times S^1 \times \D \to \R $ between $H_\oplus$ and $H_\ominus$ and $\widehat{G}: \R \times S^1 \times \D \to \R $ between $H_\ominus$ and $H_\oplus$ 
which induce continuation maps 
\begin{align*} 
& \Psi_G^b: CF^{b}(H_\oplus) \to CF^{b+\epsilon}(H_\ominus), \\ & \Psi_G^a: CF^{a - 2\epsilon}(H_\oplus) \simeq CF^{a }(H_\oplus) \to CF^{a-\epsilon}(H_\ominus),
\end{align*}
and 
\begin{align*}
& \Psi_{\widehat{G}}^b: CF^{b+ \epsilon}(H_\ominus) \to CF^{b+2\epsilon}(H_\oplus) \simeq CF^{b}(H_\oplus), \\ & \Psi_{\widehat{G}}^a: CF^{a-\epsilon  }(H_\ominus) \to CF^{a}(H_\oplus),
\end{align*}
whose compositions $$\Psi_{\widehat{G}}^b\circ \Psi_G^b: CF^{b}(H_\oplus)   \to CF^{b+ 2\epsilon}(H_\oplus) \simeq CF^{b }(H_\oplus) ,$$ and $$\Psi_{\widehat{G}}^a\circ \Psi_G^a: CF^{a}(H_\oplus) \simeq CF^{a - 2\epsilon}(H_\oplus)  \to CF^{a}(H_\oplus)$$
are both chain homotopic to the identities ${\rm id}: CF^{a}(H_\oplus)   \to CF^{a}(H_\oplus) $ and ${\rm id}: CF^{b}(H_\oplus)   \to CF^{b}(H_\oplus) $, respectively.

It follows that $G$ induces a map $$ \Psi_G:  CF^{(a,b)}(H_\oplus) \to CF^{(a-\epsilon, b+\epsilon)}(H_\ominus) $$ and $\widehat{G}$ induces a map $$\Psi_{\widehat{G}}:  CF^{(a-\epsilon, b+\epsilon)}(H_\ominus) \to CF^{(a-2\epsilon,b+2\epsilon)}(H_\oplus) \simeq CF^{(a,b)}(H_\oplus),$$ such that the composition $\Psi_{\widehat{G}}\circ \Psi_G$ is chain homotopic to the identity map ${\rm id}:  CF^{(a,b)}(H_\oplus) \to  CF^{(a,b)}(H_\oplus) $.

\end{prop}

Moreover, one obtains analogously to Proposition \ref{prop:quasi_iso} with the same proof 
\begin{prop}\label{prop:quasi-isoD}
Let $\phi_{\oplus}$, $H_{\oplus}$ as above. Let $\epsilon>0$ and let $\mathcal{Y}_{\oplus}=\{\gamma_1,\ldots, \gamma_k\}$ be a set of $1$-periodic orbits for $H_{\oplus}$ with  $\mathcal{A}_{H_{\oplus}}(\gamma_1)=\ldots\mathcal{A}_{H_{\oplus}}(\gamma_k) = \kappa$ and that are $2\epsilon$-quasi-isolated (defined analogously as in section \ref{sec:quasi-iso}). Then $\mathcal{Y}_{\oplus}$ generate a subcomplex $\mathcal{B}_{\mathcal{Y}}$ in $CF^{(\kappa-2\epsilon, \kappa+2\epsilon)}(H_{\oplus})$. 
Moreover, if $\phi_{\ominus} \in \mathcal{H}_c(\D)$ is non-degenerate with $d_{\hofer}(\phi_{\ominus}, \phi_{\oplus}) < \epsilon$, then there exist a normalized Hamiltonian that generates $\phi_{\ominus}$ and homotopies $G:\R \times S^1 \times \D \to \R$ between $H_{\oplus}$ and $H_{\ominus}$ and $G:\R \times S^1 \times \D \to \R$ between $H_{\ominus}$ and $H_{\oplus}$ which induce chain maps 
$$\Psi^{\mathcal{Y}}_G: \mathcal{B}_{\mathcal{Y}} \to CF_{\alpha}^{(\kappa-\epsilon,\kappa+\epsilon)}(H_{\ominus})$$ and 
$$\Psi^{\mathcal{Y}}_{\widehat{G}}:  CF_{\alpha}^{(\kappa-\epsilon,\kappa+\epsilon)}(H_{\ominus})\to \mathcal{B}_{\mathcal{Y}} $$ such that 
$$\Psi^{\mathcal{Y}}_{\widehat{G}}\circ \Psi^{\mathcal{Y}}_{{G}} = \id.$$ 
\end{prop}

\subsubsection{Floer homology for compactly supported Hamiltonians in $\D$}

We consider the group $\mathcal{H}_0(\D)$ of compactly supported Hamiltonian diffeomorphisms of $(\D,dx \wedge dy)$; i.e. area preserving diffeomorphisms which coincide with the identity in some neighbourhood of $\partial \D$. As we showed in the proof of Lemma \ref{lem:felix}, for each element $\phi\in \mathcal{H}_0(\D)$ there exists a Hamiltonian $H$ which vanishes near the boundary and that generates $\phi$, in the sense that the time $1$-map of $H$ is $\phi$. We let $\mathrm{Ham}_0(\D)$ be the set of Hamiltonians on $\D$ which vanish near the boundary.

We consider $S^1$-dependent almost complex structures $J_t$  which admit extensions to $\R^2$ which coincide with the complex multiplication by $\mathrm{i}$ outside of $\R^2$.

The Hamiltonian action of a loop is defined as in \eqref{eq:defaction}.
One defines the Hofer distance of two elements $\phi_1$ and $\phi_2$ of $\mathrm{Ham}_c(\D)$ as in the previous section. 
 
In order to define the Floer homology of $H$ we cannot follow the steps in Section \ref{sec:Hamonclosedsurfaces}. The reason for this is that $H$ has degenerate periodic orbit: all $1$-periodic orbits of $\phi_H^1$ contained in the neighbourhood of $\partial \overline{\D}$ where $H$ vanishes are degenerate. We thus follow a more geometric approach, which is explained in \cite[Section 3]{Ginzburg2010}.

Before presenting the construction we introduce some terminology. Let ${u}_n:\R \times S^1 \to W $ be a sequence of Floer cylinders of some pair $(\overline{H},J_t)$, where $(W,\overline{\omega})$ is a symplectic manifold. A $1$-periodic orbit $\gamma$ of $\phi_{\overline{H}}$ is called a breaking orbit, if there exists a sequence $s_n$ such that $u_n(s_n,\cdot) : S^1 \to W$ converges in $C^0$ to $\gamma: S^1 \to W$: because of elliptic regularity, the convergence is actually in $C^\infty$. 

We say that a Hamiltonian diffeomorphism in $\mathcal{H}_0(\D)$ is non-degenerate, if every $1$-periodic orbit whose action is $\neq 0$ is non-degenerate. This set is $C^\infty$-dense in $\mathcal{H}_0(\D)$. If $H$ is a Hamiltonian generating a non-degenerate Hamiltonian diffeomorphism $\phi$, then the only accumulation point of  $\mathrm{Spec}^1(H)$ is $0$, and for any $c>0$ the number of elements of $\mathcal{P}^1(H)$ with action in $\R \setminus [-c,c]$ is finite. 

Let now $\phi_H$ be a non-degenerate Hamiltonian diffeomorphism in $\mathcal{H}_0(\D)$ which is generated by a Hamiltonian $H$, and let $a<b$  be real numbers which are not in $\mathrm{Spec}^1(H)$ and such that every $1$-periodic in $\mathcal{P}^1_{(a,b)}(H)$ is non-degenerate. It follows that $0$ is not contained in $[a,b]$, and that $\mathcal{P}^1_{(a,b)}(H)$  is a finite set. We define $CF^{(a,b)}(H)$ to be the $\Z_2$-vector space generated by $\mathcal{P}^1_{(a,b)}(H)$, i.e.
\begin{equation}
    CF^{(a,b)}(H) \bigoplus_{\gamma \in \mathcal{P}^1_{(a,b)}(H)} \Z_2 \cdot \gamma.
\end{equation}

 
We first observe that   given any choice of smooth $S^1$-family of compatible almost complex structures $J_t$ on $(\D,dx\wedge dy)$ which coincides with the complex multiplication by $\mathrm{i}$ outside of $\D$, then the maximum principle implies that a Floer cylinder whose asymptotic limits are in the interior of $\D$ must be contained inside $\D$.
 
We then notice that given any choice of smooth $S^1$-family of compatible almost complex structures $J_t$ on $(\D,dx\wedge dy)$, if $u_n$ is a sequence of Floer cylinders of $(H,J_t)$ whose negative and positive asymptotic limits are in $\mathcal{P}^1_{(a,b)}(H)$ then any breaking orbit of $u_n$ has action in $(a,b)$. Using the fact that all orbits in $\mathcal{P}^1_{(a,b)}(H)$ are  non-degenerate and the techniques of \cite{Floer}, it is possible to compactify all moduli spaces $\mathcal{M}(\gamma,\gamma',H,J_t)$,  where $\gamma$ and $\gamma'$ are in $\mathcal{P}^1_{(a,b)}(H)$. The compactified moduli space $\overline{\mathcal{M}}(\gamma,\gamma',H,J_t)$ is formed by the union of $\mathcal{M}(\gamma,\gamma',H,J_t)$ and of broken Floer cylinders of $(H,{J}_t)$ from $\gamma$ to $ \gamma'$. For any broken Floer cylinder $\mathbf{u}$ which is negatively and positively asymptotic to orbits  $\mathcal{P}^1_{(a,b)}(H) $, its breaking orbits must also be in $\mathcal{P}^1_{(a,b)}(H) $. The reason is that the action of these breaking orbits must be smaller than that of the  negative limit of $\mathbf{u}$ and bigger than that of the positive limit of $\mathbf{u}$.

We then invoke the results of \cite{FHS} and choose a generic $S^1$-family $J_t$ so that for any $\gamma$ and $\gamma'$ in $\mathcal{P}^1_{(a,b)}(H)$ the moduli space $\mathcal{M}(\gamma,\gamma',H,J_t)$ is a manifold of dimension $\mu_{\rm CZ}(\gamma) - \mu_{\rm CZ}(\gamma')-1$.
In this situation $\mathcal{M}(\gamma,\gamma',H,J_t)$ is a finite set of points if $\mu_{\rm CZ}(\gamma) - 1 = \mu_{\rm CZ}(\gamma')$.

Let $\gamma \in \mathcal{P}^1_{(a,b)}(H)$. We define for $\gamma' \in \mathcal{P}^1_{(a,b)}(H)$ with  $\mu_{\rm CZ}(\gamma) - 1 = \mu_{\rm CZ}(\gamma')$ the number $$C(\gamma,\gamma') := \#( \mathcal{M}(\gamma,\gamma',H,J_t)) \mod 2. $$
If $\mu_{\rm CZ}(\gamma) - 1 \neq \mu_{\rm CZ}(\gamma')$ we let $C(\gamma,\gamma') =0$. 
With these preliminaries we define the differential $d^{J_t}: CF^{(a,b)}(H) \to CF^{(a,b)}(H)$ by letting for each  $\gamma \in \mathcal{P}^1_{(a,b)}(H) $
\begin{equation*}
d^{J_t} (\gamma) = \sum_{\gamma' \in \mathcal{P}^1_{(a,b)}(H) } C(\gamma,\gamma')\gamma'. 
\end{equation*}
The differential $ d^{J_t}$ is extended to all of $CF^{(a,b)}(H)$ linearly.

Using regularity of $(H,J_t)$  and the fact that for broken Floer cylinder $\mathbf{u}$ which is negatively and positively asymptotic to orbits  $\mathcal{P}^1_{(a,b)}(H) $, its breaking orbits must also be in $\mathcal{P}^1_{(a,b)}(H) $, the proof that $(d^{J_t})^2=0$ is the same as the one for the analogous statement for the case of closed surfaces of positive genus.

We are ready to state the following proposition, analogous to Proposition \ref{prop:actionwindowsclosedsurface}, for Hamiltonian diffeomorphisms in $\mathcal{H}_0(\D)$.

\begin{prop} \label{prop:chainisodisk0}
Let $\phi_\oplus$ be a non-degenerate Hamiltonian diffeomorphism in $\mathcal{H}_0(\D)$ and $H_\oplus:S^1 \times \D \to \R$ be a Hamiltonian in $\mathrm{Ham}_0(\D)$ generating $\phi_\oplus$. We take real numbers $a<b$ which do not belong to $\mathrm{Spec}^1(H_\oplus)$ and $0 \notin [a,b]$. Let $\epsilon>0$ be such that all elements of $\mathrm{Spec}^1(H_\oplus)$  in the interval $(a-2\epsilon, b+2\epsilon)$ are contained in $(a,b)$. Let $\phi_\ominus$ be a non-degenerate Hamiltonian diffeomorphism with $d_{\rm Hofer}(\phi_\oplus, \phi_\ominus)< \epsilon$. Then, there exist a normalized Hamiltonian $H_\ominus: S^1 \times \D \to \R $ and homotopies $G: \R \times S^1 \times \D \to \R $ between $H_\oplus$ and $H_\ominus$ and $\widehat{G}: \R \times S^1 \times \D \to \R $ between $H_\ominus$ and $H_\oplus$ 
which induce continuation maps 
$$ \Psi_G:  CF^{(a,b)}(H_\oplus) \to CF^{(a-\epsilon, b+\epsilon)}(H_\ominus) $$ and $$\Psi_{\widehat{G}}:  CF^{(a-\epsilon, b+\epsilon)}(H_\ominus) \to CF^{(a-2\epsilon,b+2\epsilon)}(H_\oplus) \simeq CF^{(a,b)}(H_\oplus),$$ such that the composition $\Psi_{\widehat{G}}\circ \Psi_G$ is chain homotopic to the identity map ${\rm id}:  CF^{(a,b)}(H_\oplus) \to  CF^{(a,b)}(H_\oplus) $.

\end{prop}

The proof is a variation of the proof of Proposition \ref{prop:actionwindowsclosedsurface}. We provide a sketch of the proof and explain the necessary adjustments. 

\textit{Sketch of proof.} 

We start by explaining the construction of $H_\ominus$, $G$ and $\widehat{G}$. 
Because $d_{\rm Hofer}(\phi_\ominus,\phi_\oplus)<\epsilon $ there exists a Hamiltonian
\begin{equation} \label{eq:H-minus}
  F:S^1\times \D \to \R   
\end{equation}
which vanishes near the boundary whose time $1$-map is $\phi^{-1}_\oplus \circ \phi_\ominus$ and that satisfies 
\begin{equation}\label{eq:Hofer-distance}
    \int_0^1 (\max F_t - \min F_t) dt< \epsilon,
\end{equation}
 where for $t\in S^1$ we define $F_t:=F(t,\cdot): \Sigma \to \R$. We then define $H_\ominus$, $G$ and $\widehat{G}$ as in equations \eqref{eq:F}, \eqref{eq:G} and \eqref{eq:G-hat}. 
 
We first explain why $G$ induces a chain map $$\Psi_G:  CF^{(a,b)}(H_\oplus) \to CF^{(a-\epsilon, b+\epsilon)}(H_\ominus).$$ For this we first choose a  homotopy of almost complex structures $J^s_t$ between $J_\oplus$ and $J_\ominus$. We assume that for each fixed $s \in \R$ the almost complex structures $J^s_t$ coincide with the complex multiplication by $\mathrm{i}$ outside of $\D$. This guarantees that a Floer cylinder of $(G,J^s_t)$ whose negative asymptotic limit is an orbit $\gamma_\oplus \in \mathcal{P}^1_{(a,b)}(H_\oplus)$ and positive asymptotic limit is an orbit in $\gamma_\ominus \in \mathcal{P}^1_{(a-\epsilon,b+\epsilon)}(H_\ominus)$ must have its image contained in $\D$.

Then we notice that if $u$ is a Floer cylinder of $(G,J^s_t)$ whose negative asymptotic limit is an orbit $\gamma_\oplus \in \mathcal{P}^1_{(a,b)}(H_\oplus)$ and negatively asymptotic to a periodic orbit $\gamma_\ominus  $ of $H_\ominus$ then we must have
\begin{equation} \label{eq:estimateactionimportandisk}
   \mathcal{A}_{H_\ominus}(\gamma_\ominus) \leq \mathcal{A}_{H_\oplus}(\gamma_\oplus) +\epsilon . 
\end{equation}
This implies that $\gamma_\ominus \in \mathcal{P}^1_{(a-\epsilon,b+\epsilon)}(H_\ominus)$.
Moreover, using this same inequality one shows that if $u_n$ is a sequence of Floer cylinders of $(G,J^s_t)$ whose negative asymptotic limit is in $\mathcal{P}^1_{(a,b)}(H_\oplus)$ and whose positive asymptotic limit is in  $\mathcal{P}^1_{(a-\epsilon,b+\epsilon)}(H_\ominus)$, then any breaking orbit of this sequence is either in $\mathcal{P}^1_{(a,b)}(H_\oplus)$ or in  $\mathcal{P}^1_{(a-\epsilon,b+\epsilon)}(H_\ominus)$. 
This observations allow us to apply Floer compactness to compactify moduli spaces $\mathcal{M}(\gamma_\oplus,\gamma_\ominus,G,J^s_t)$ where $\gamma_\oplus \in \mathcal{P}^1_{(a,b)}(H_\oplus)$ and $\gamma_\ominus \in \mathcal{P}^1_{(a-\epsilon,b+\epsilon)}(H_\ominus)$. The compactification will be formed by elements of $\mathcal{M}(\gamma_\oplus,\gamma_\ominus,G,J^s_t)$ and broken Floer cylinders of $(G,J^s_t)$ from $\gamma_\oplus$ to $\gamma_\ominus$. The breaking orbits of these broken Floer cylinders must be either in $\mathcal{P}^1_{(a,b)}(H_\oplus)$ or in  $\mathcal{P}^1_{(a-\epsilon,b+\epsilon)}(H_\ominus)$. This follows from the estimate \eqref{eq:estimateactionimportandisk}.

We then choose $J^s_t$ in a generic way so that the moduli space $\mathcal{M}(\gamma_\oplus,\gamma_\ominus,G,J^s_t) $ considered in the previous paragraph are manifolds of dimension $\mu_{\rm CZ}(\gamma_\oplus) - \mu_{\rm CZ}(\gamma_\ominus)$. We then define $K_{G,J^s_t}(\gamma_\oplus, \gamma_\ominus):= (\# \mathcal{M}(\gamma_\oplus, \gamma_\ominus, G,J^s_t))\mod 2$, if  $\mu_{\rm CZ}(\gamma_\oplus) = \mu_{\rm CZ}(\gamma_\ominus)$ and $K_{Q,J^s_t}(\gamma_\oplus, \gamma_\ominus)=0$ otherwise. The map $\Psi_G: CF^{(a,b)}(H_\oplus) \to CF^{(a-\epsilon, b+\epsilon)}(H_\ominus)$ is given by
\begin{equation} \label{eq:defiGdisk}
    \Psi_G(\gamma_\oplus):= \sum_{\gamma \in \mathcal{P}^1_{(a-\epsilon,b+\epsilon)}(H_\ominus)}  K_{G,J^s_t}(\gamma_\oplus, \gamma) \gamma
\end{equation}

With this, and using the fact the Floer cylinders in these spaces must converge to broken cylinders whose breaking orbits are in $\mathcal{P}^1_{(a,b)}(H_\oplus)$ or in  $\mathcal{P}^1_{(a-\epsilon,b+\epsilon)}(H_\ominus)$, one proves that the map $\Psi_G$ induces a map on homology.

A similar construction is done to define the map $\Psi_{\widehat{G}}: CF^{(a-\epsilon, b+\epsilon)}(H_\ominus) \to  CF^{(a-2\epsilon,b+ 2 \epsilon)}(H_\oplus) \simeq CF^{(a,b)}(H_\oplus) $. For this we first choose a  homotopy of almost complex structures $\widehat{J}^s_t$ between $J_\ominus$ and $J_\oplus$. We assume that for each fixed $s \in \R$ the almost complex structures $\widehat{J}^s_t$ coincide with the complex multiplication by $\mathrm{i}$ outside of $\D$. To define $\Psi_{\widehat{G}}$ we study moduli spaces $\mathcal{M}(\gamma_\ominus, \gamma_\oplus, \widehat{G},\widehat{J}^s_t))$, where $\gamma_\oplus \in \mathcal{P}^1_{(a,b)}(H_\oplus)$ and  $\gamma_\ominus \in \mathcal{P}^1_{(a-\epsilon,b+\epsilon)}(H_\ominus)$. Using an estimate analogous to \eqref{eq:estimateactionimportandisk} one shows that breaking orbits of sequences of elements in $\mathcal{M}(\gamma_\ominus, \gamma_\oplus, \widehat{G},\widehat{J}^s_t))$ must be in $\mathcal{P}^1_{(a-2\epsilon,b+2\epsilon)}(H_\oplus)$ or in  $\mathcal{P}^1_{(a-\epsilon,b+\epsilon)}(H_\ominus)$. However, by our assumption $\mathcal{P}^1_{(a-2\epsilon,b+2\epsilon)}(H_\oplus) = \mathcal{P}^1_{(a,b)}(H_\oplus)$: this shows why this assumption is crucial for us to be able to define the map $\Psi_{\widehat{G}}$ from $CF^{(a-\epsilon, b+\epsilon)}(H_\ominus)$ to $ CF^{(a,b)}(H_\oplus) $.

The proof of the fact that $\Psi_{\widehat{G}} \circ \Psi_{{G}} : CF^{(a,b)}(H_\oplus) \to CF^{(a,b)}(H_\oplus)  $ is chain-homotopic to the identity, follows the same scheme of the analogous statement in Proposition \ref{prop:actionwindowsclosedsurface}. That is, we define the homotopy of homotopies $(Q_a,J^s_t(a))_{a \in [0,+\infty)}$ and study the relevant $1$-dimensional moduli spaces of Floer cylinders of the homotopy of homotopies $(Q_a,J^s_t(a))_{a \in [0,+\infty)}$ with asymptotic limits in $\mathcal{P}^1_{(a,b)}(H_\oplus)$. Again, in order to show that we can define the chain-homotopy map in the appropriate action windows one shows that all breaking orbits for sequences elements in these moduli spaces are in $\mathcal{P}^1_{(a,b)}(H_\oplus)$ or in  $\mathcal{P}^1_{(a-\epsilon,b+\epsilon)}(H_\ominus)$. This follows from estimates similar to \eqref{eq:estimateactionimportandisk} for the carefully constructed homotopy $(Q_a)_{a \in [0,+\infty)}$.

 \qed

\subsection{Hamiltonian dynamics and Floer homology on $S^2$} \label{sec:Hamonsphere}
 
Let $\omega$ be a symplectic form on $S^2$ and assume that $\int_{S^2} \omega = 8$. In order to define Floer homology for Hamiltonians on $(S^2,\omega)$ we need to make certain adaptations. In particular, because $\pi_2(S^2) \neq 0$ there are constraints on the action windows for which Floer homology can be defined, because of the possibility of bubbling off of holomorphic spheres. Notice, that because $\int_{S^2} \omega = 8$ it follows that any for any sphere $S$ in $(S^2,\omega)$ its $\omega$-integral $\int_S \omega$ is a multiple of $8$: this integral clearly only depends on the free homotopy class of $S$.

Let $H: S^1 \times S^2 \to \R$ be a normalized Hamiltonian and $\phi_H$ its time $1$-map.  
Because  $\pi_2(S^2) \neq 0$ the $H$-action of a closed curve $y$ of $\phi^t_H$ is not well-defined but depends of the choice of a capping $\mathcal{D}_y$ of $y$.  We thus the define for a pair $(y,\mathcal{D}_y)$
\begin{equation}
    \mathcal{A}_H(y,\mathcal{D}_y) := -\int_{\mathcal{D}_y} \omega + \int_0^1 H(t,y(t))dt.
\end{equation}

If $y$ is a closed curve and $\mathcal{D}_y$ and $\mathcal{D}'_y$ are two cappings of $y$, then $$\mathcal{A}_H(y,\mathcal{D}'_y) - \mathcal{A}_H(y,\mathcal{D}_y)= \int_{\mathcal{D}_y \# -\mathcal{D}'_y} \omega,$$
where $\mathcal{D}'_y \# -\mathcal{D}_y$ is the sphere obtained by gluing $\mathcal{D}'_y$ and $-\mathcal{D}'_y$. It follows that the difference $\mathcal{A}_H(y,\mathcal{D}_y) - \mathcal{A}_H(y,\mathcal{D}_y)$ is always a multiple of $8$. Since $\mathcal{A}_H(y,\mathcal{D}_y)$ only depend on the homotopy class of $\mathcal{D}_y$ we will define for a pair $(y,[\mathcal{D}_y])$ where $[\mathcal{D}_y])$ denotes the homotopy class of $\mathcal{D}_y$ the action
$\mathcal{A}_H(y,[\mathcal{D}_y]) $ by 
\begin{equation} \label{eq:actionS2}
    \mathcal{A}_H(y,[\mathcal{D}_y]) := -\int_{\mathcal{D}_y} \omega + \int_0^1 H(t,y(t))dt.
\end{equation}
Given real numbers $a<b$ we let $\mathcal{P}^1_{(a,b)}(H)$ be the set of pairs $(\gamma,[\mathcal{D}_\gamma])$ where $\gamma$ is a $1$-periodic orbit of $\phi^t_H$ and $[\mathcal{D}_\gamma]$ is a homotopy class of cappings of $\gamma$ such that $\mathcal{A}_H(\gamma,[\mathcal{D}_\gamma])\in (a,b) $.

The $1$-periodic spectrum $\mathrm{Spec}^1(H)$ is the set of all possible actions $\mathcal{A}_H(\gamma,[\mathcal{D}_\gamma])$ of pairs  $(\gamma,[\mathcal{D}_\gamma])$ where $\gamma$ is a $1$-periodic orbit of $\phi_H$ and $[\mathcal{D}_\gamma]$ is a homotopy class of cappings of $\gamma$.

Similarly, the Conley-Zehnder index of $1$-periodic orbit $\gamma$ of $\phi_H$ is not well-defined but if we fix a homotopy class $[\mathcal{D}_\gamma]$ of cappings of $\gamma$ the Conley-Zehnder index $\mu_{\rm CZ}(\gamma, [\mathcal{D}_\gamma])$ is well-defined. 

Suppose now that $\phi_H$ is a non-degenerate Hamiltonian diffeomorphism of $(S^2,\omega)$.
We now fix real numbers $a<b$ such that $|b-a| \leq \frac{1}{4}$. Once this is done we define $$ CF^{(a,b)}(H):= \bigoplus_{(\gamma, [\mathcal{D}_\gamma]) \in \mathcal{P}^1_{(a,b)}(H)}\Z_2 \cdot (\gamma, [\mathcal{D}_\gamma]).$$

We now choose a smooth $S^1$-family $J_t$ of compatible almost complex structures on $(S^2,\omega)$. The Floer equation of $(H,J_t)$ is defined as \eqref{eq:defFloereq}.  Given  $(\gamma, [\mathcal{D}_\gamma])$ and $(\gamma', [\mathcal{D}_{\gamma'}])$ in $\mathcal{P}^1_{(a,b)}(H)$  we let $$\mathcal{M}(\gamma, [\mathcal{D}_\gamma],\gamma', [\mathcal{D}_{\gamma'}],H,J_t)$$ be the moduli space whose elements are Floer cylinders $u$ of $(H,J_t)$ negatively asymptotic to $\gamma$ and positively to $\gamma'$, and such that the gluing $u \# - \mathcal{D}_{\gamma'}$ is homotopic to $\mathcal{D}_\gamma$. As previously, if for two Floer cylinders $u_1$ and $u_2$ of $(H,J_t)$  there is an $s_0$ that $u_1(s_0 + \cdot,\cdot)=u_2(\cdot,\cdot)$ then $u_1$ and $u_2$ represent the same element in the moduli space.

We first explain why for any sequence $u_n$ of elements of the moduli space $\mathcal{M}(\gamma, [\mathcal{D}_\gamma],\gamma', [\mathcal{D}_{\gamma'}],H,J_t)$, the gradients of $u_n$ are uniformly bounded. If this was not the case, then bubbling analysis would imply that a non-constant holomorphic sphere $v: (S^2,i) \to (S^2,J_{t_0})$ bubbles off of the sequence $u_n$ for some $t_0 \in S^1$. But, reasoning as in the proof of \cite[Lemma 6.6.2]{AD} one obtains that the energy of $v$ is bounded from above by $4\sup_{k \in \N}E(u_k) \leq 4|b-a| \leq 1$. Since any non constant holomorphic sphere in $(S^2,\omega)$ must have energy $ \geq 8$ we conclude that $v$ is constant, which is a contradiction. It follows from Floer's techniques \cite{Floer} that since the gradients of any sequence $u_n$ of elements of the moduli space are uniformly bounded, any such sequence must converge to a broken Floer cylinder.  
Moreover, since the action decreases along Floer cylinders, we know that any breaking pair (see Definition \ref{def:breakingpair}) $(\widehat \gamma,[\mathcal{D}_{\widehat \gamma}])$ of the sequence $u_n$ in $\mathcal{P}^1_{(a,b)}(H)$. We choose now $J_t$ generically, so that $\mathcal{M}(\gamma, [\mathcal{D}_\gamma],\gamma', [\mathcal{D}_{\gamma'}],H,J_t)$ are manifolds whose dimension is $\mu_{\rm CZ}(\gamma, [\mathcal{D}_\gamma])- \mu_{\rm CZ}(\gamma', [\mathcal{D}_{\gamma'}]) -1$. 

\begin{defn} \label{def:breakingpair}
If $u_n$ is a sequence of elements of the moduli space \linebreak $\mathcal{M}(\gamma, [\mathcal{D}_\gamma],\gamma', [\mathcal{D}_{\gamma'}],H,J_t)$, then a pair $(\widehat \gamma,[\mathcal{D}_{\widehat \gamma}])$ is called a breaking pair for $u_n$ if:
\begin{itemize}
    \item there exists a sequence $s_n$ such that $u_n(s_n,\cdot)$ converges in $C^\infty$ to $\widehat \gamma$,
    \item and for sufficiently large $n$ such that $u_n(s_n,\cdot)$ is contained in a small tubular neighbourhood $U_{\widehat \gamma}$ of $\widehat \gamma$, the capping of $\widehat \gamma$ obtained by gluing  $\mathrm{Cyl}_n \#  u_n([s_n,+\infty) \times S^1) \# - \mathcal{D}_{\gamma'} $ is in the homotopy class $[\mathcal{D}_{\widehat \gamma}]$, where $\mathrm{Cyl}_n$ is any cylinder from $\widehat \gamma$ to $u_n(s_n,\cdot)$ which is contained in the tubular neighbourhood $U_{\widehat \gamma}$.
\end{itemize}
\end{defn}

We let  $$C(\gamma, [\mathcal{D}_\gamma],\gamma', [\mathcal{D}_{\gamma'}] )= \# \mathcal{M}(\gamma, [\mathcal{D}_\gamma],\gamma', [\mathcal{D}_{\gamma'}],H,J_t) \mod 2$$ if $\mu_{\rm CZ}(\gamma, [\mathcal{D}_\gamma])- 1 = \mu_{\rm CZ}(\gamma', [\mathcal{D}_{\gamma'}])$, and $C(\gamma, [\mathcal{D}_\gamma],\gamma', [\mathcal{D}_{\gamma'}] )=0$ otherwise. With this we are ready to define $d^{J_t}: CF^{(a,b)}(H) \to CF^{(a,b)}(H)$ by letting $$d^{J_t}(\gamma, [\mathcal{D}_\gamma]) = \sum_{(\gamma', [\mathcal{D}_{\gamma'}]) \in \mathcal{P}^1_{(a,b)}(H)} C(\gamma, [\mathcal{D}_\gamma],\gamma', [\mathcal{D}_{\gamma'}] ) \cdot (\gamma', [\mathcal{D}_{\gamma'}] )$$ for the generators and extending it algebraically to all of $CF^{(a,b)}(H)$.
Once we know that breaking pairs of the relevant moduli spaces are contained in $\mathcal{P}^1_{(a,b)}(H)$ the proof that $(d^{J_t})^2 = 0$ is as the one of the analogous statement for Floer homology on closed surfaces of positive genus. We then let $HF^{(a,b)}(H)$ be the homology of the chain-complex $(CF^{(a,b)}(H),d^{J_t})$.
 
We are now ready to state the following propostion, analogous to Proposition \ref{prop:actionwindowsclosedsurface},  for Hamiltonian diffeomorphisms on $(S^2,\omega)$.

\begin{prop} \label{prop:chainisosphere}
Let $\phi_\oplus$ be a non-degenerate Hamiltonian diffeomorphism in $\mathcal{H}(S^2,\omega)$ and $H_\oplus:S^1 \times S^2 \to \R$ be a Hamiltonian in $\mathrm{Ham}(S^2,\omega)$ generating $\phi_\oplus$. We take real numbers $a<b$ that do not belong to $\mathrm{Spec}^1(H_\oplus)$ and such that and $b-a< \frac{1}{4}$. Let $\epsilon>0$ be such that $b-a + 2\epsilon < \frac{1}{4}$ and that all elements of $\mathrm{Spec}^1(H_\oplus)$  in the interval $(a-2\epsilon, b+2\epsilon)$ are contained in $(a,b)$. Let $\phi_\ominus$ be a non-degenerate Hamiltonian diffeomorphism with $d_{\rm Hofer}(\phi_\oplus, \phi_\ominus)< \epsilon$. Then, there exist a normalized Hamiltonian $H_\ominus: S^1 \times S^2 \to \R $ and homotopies $G: \R \times S^1 \times S^2 \to \R $ between $H_\oplus$ and $H_\ominus$ and $\widehat{G}: \R \times S^1 \times S^2 \to \R $ between $H_\ominus$ and $H_\oplus$ 
which induce continuation maps 
$$ \Psi_G:  CF^{(a,b)}(H_\oplus) \to CF^{(a-\epsilon, b+\epsilon)}(H_\ominus) $$ and $$\Psi_{\widehat{G}}:  CF^{(a-\epsilon, b+\epsilon)}(H_\ominus) \to CF^{(a-2\epsilon,b+2\epsilon)}(H_\oplus) \simeq CF^{(a,b)}(H_\oplus),$$ such that the composition $\Psi_{\widehat{G}}\circ \Psi_G$ is chain homotopic to the identity map ${\rm id}:  CF^{(a,b)}(H_\oplus) \to  CF^{(a,b)}(H_\oplus) $.

\end{prop}

The proof is similar to the one of Proposition \ref{prop:chainisodisk0}.
There are two main things to be observed. Firstly, because of our choice of the size of the action windows $(a,b)$, $(a-\epsilon,b+\epsilon)$ and $(a-2\epsilon,b+2\epsilon)$ the relevant Floer cylinders have small energy and this precludes bubbling. Secondly, all breaking pairs that appear for sequences of elements for the relevant moduli spaces are in $\mathcal{P}^1_{(a,b)}(H_\oplus)$ or $\mathcal{P}^1_{(a-\epsilon,b+\epsilon)}(H)$.

Finally, analogous to Proposition \ref{prop:quasi_iso} one obtains
\begin{prop}\label{prop:quasi-isoS}
Let $\phi_{\oplus} \in \Ham(S^2, \omega)$ be non-degenerate, and $H_{\oplus}\in \mathrm{Ham}(S^2,\omega)$ generating $\phi_{\oplus}$. Let $\epsilon>0$ and let $\mathcal{Y}=\{\gamma_1, \ldots, \gamma_k\}$ be a set of $1$-periodic orbits for $H_{\oplus}$ which is $2\epsilon$-quasi-isolated (analogously defined as in section \ref{sec:quasi-iso}), and for which there are disc cappings $D_{\gamma_1}, \ldots, D_{\gamma_k}$ such that $\mathcal{A}_{H_{\oplus}}(\gamma_1, D_{\gamma_1}) = \ldots \mathcal{A}_{H_{\oplus}}(\gamma_k, D_{\gamma_k})= \kappa$.  Then the pairs $(\gamma_1, D_{\gamma_1})$, $\ldots$, $(\gamma_k, D_{\gamma_k})$ generate a subcomplex $\mathcal{B}_{\mathcal{Y}}$ in $CF^{(\kappa-2\epsilon,\kappa+2\epsilon)}$. Moreover, if $\phi_{\ominus} \in \Ham(S^2, \omega)$ is non-degenerate with $d_{\hofer}(\phi_{\ominus}, \phi_{\oplus})< \epsilon$ there are $H_{\ominus} \in \mathrm{Ham}(S^2,\omega)$ generating $\phi_{\ominus}$ and homotopies $G$ from $H_{\oplus}$ to $H_{\ominus}$ and $\widehat{G}$ from $H_{\ominus}$ to $H_{\oplus}$ 
which induce chain maps 
$$\Psi^{\mathcal{Y}}_G: \mathcal{B}_{\mathcal{Y}} \to CF_{\alpha}^{(\kappa-\epsilon,\kappa+\epsilon)}(H_{\ominus})$$ and 
$$\Psi^{\mathcal{Y}}_{\widehat{G}}:  CF_{\alpha}^{(\kappa-\epsilon,\kappa+\epsilon)}(H_{\ominus})\to \mathcal{B}_{\mathcal{Y}} $$ such that 
$$\Psi^{\mathcal{Y}}_{\widehat{G}}\circ \Psi^{\mathcal{Y}}_{{G}} = \id.$$ 
\end{prop}

\section{Floer cylinders and holomorphic curves} \label{sec:holcylinders}

In this section we recall a construction due to Gromov which shows that Floer cylinders on a symplectic surface $(\Sigma, \omega)$ for a pair $(H,J^s_t)$ are in bijective correspondence with holomorphic cylinders on the $4$-manifold $\R \times S^1 \times \Sigma$ endowed with a almost complex structure which is constructed out of $(H,J^s_t)$. The construction can be performed for symplectic manifolds of any dimension, but we restrict our attention to surfaces because this is the case on which we are interested. 
We consider coordinates $(s,t,p) $ on $\R \times S^1 \times \Sigma$.

On a compact symplectic surface $(\Sigma,\omega)$ we consider a $C^\infty$-smooth Hamiltonian $H:\R \times S^1 \times \Sigma \to \R$ which is normalized and a $C^\infty$-smooth family of $\R \times S^1$-dependent almost complex structures $J^s_t$ which are compatible with $(\Sigma,\omega)$. We assume that there exists $s_H>0$ such that
\begin{equation} \label{eq:Hasymp}
    H \mbox{ and  } J^s_t \mbox{ does not depend on } s \mbox{ if } |s| \geq s_H.
\end{equation}
Recall that 
\begin{itemize}
    \item  if $\Sigma$ is closed normalized means that $\int_\Sigma H^s_t\omega=0$ for each $(s,t) \in \R \times S^1$, where $H^s_t(\cdot):= H(t,\cdot)$,
    \item if $\partial \Sigma \neq 0$ normalized means that $H^s_t$ vanishes in a neighbourhood of $\partial \Sigma $ for all $(s,t) \in \R \times S^1$.
\end{itemize}
  We denote by $H_s:S^1 \times \Sigma \to \R$ the function $H(s,\cdot,\cdot)$.

For each $(s,t)$ fixed the function $H$ defines a $\R \times S^1$ family of  vector fields $X_H$ on $\Sigma$ by requesting $\iota_{X_H(s,t,\cdot)}= d_\Sigma H(s,t,\cdot)$, where $d_\Sigma H(s,t,\cdot)$ is the differential of $H(s,t, \cdot) : \Sigma \to \R$ in $\Sigma$. When there is no danger of confusion we write $d$ for $d_\Sigma$. 

The Floer operator $\mathcal{F}_{H,J}$ for the pair $(H,J^s_t)$ applied to a cylinder $u: \R \times S^1 \to \Sigma$ is 
\begin{equation} \label{eq:Floereq1}
    \mathcal{F}_{H,J}(u) = \partial_s u(s,t) +  J^s_t(u(s,t))\big(\partial_t u (s,t) - X_H(s,t,u(s,t)\big).
\end{equation}
A Floer cylinder is a map $u: \R \times S^1 \to \Sigma$ such that 
\begin{equation}\label{eq:Floereq2}
    \mathcal{F}_{H,J}(u) = 0.
\end{equation}

Associated to the pair $(H,J^s_t)$ we construct an almost complex structure $\widetilde{J}$ on $\R \times S^1 \times \Sigma$. In order to define $\widetilde{J}$ we first introduce a natural decomposition of the tangent bundle $T(\R \times S^1 \times \Sigma)$. 
\begin{itemize}
    \item Associated to the coordinates $s$ and $t$ we have tangent vectors $\partial_s$ and $\partial_t$ at the tangent space of every point  $(s_0,t_0,p_0) \in \R \times S^1 \times \Sigma$.
    \item Consider the family of surfaces $\{s\}\times \{t\} \times \Sigma$ which foliates $\R \times S^1 \times \Sigma$. A vector $v \in T_{(s_0,t_0,p_0)}(\R \times S^1 \times \Sigma)$ is called horizontal if it belongs to $\{s_0\}\times \{t_0\} \times T_{p_0}\Sigma$. Let $H_{(s_0,t_0,p_0)}$ be the sub-space of horizontal vectors in $T_{(s_0,t_0,p_0)}(\R \times S^1 \times \Sigma)$.
\end{itemize}
It is clear that $T_{(s_0,t_0,p_0)}(\R \times S^1 \times \Sigma) = \R \partial_s \oplus \R \partial_t \oplus H_{(s_0,t_0,p_0)}$.

Let $\Pi_\Sigma: \R \times S^1 \times \Sigma \to \Sigma$ be the projection on the third coordinate. Then, the restriction $L_{(s_0,t_0,p_0)}:= (D\Pi_\Sigma)_{(s_0,t_0,p_0)}|_{H_{(s_0,t_0,p_0)}}$ of $(D\Pi_\Sigma)_{(s_0,t_0,p_0)}$ to $H_{(s_0,t_0,p_0)}$ is an isomorphism between $H_{(s_0,t_0,p_0)}$ and $T_{p_0}\Sigma$. We denote its inverse by $L_{(s_0,t_0,p_0)}^{-1}$. 

At a point $(s_0,t_0,p_0)$  $\widetilde{J}(s_0,t_0,p_0)$ is defined by the following formulas:
\begin{align*}
    \widetilde{J}(s_0,t_0,p_0) v := L_{(s_0,t_0,p_0)}^{-1} \circ J^{s_0}_{t_0}(p_0) \circ L_{(s_0,t_0,p_0)}(v) \mbox{ for } v \in H_{(s_0,t_0,p_0)} , \\
   \widetilde{J}(s_0,t_0,p_0) \partial_s := \partial_t + L_{(s_0,t_0,p_0)}^{-1}(X_H(s_0,t_0,p_0)).
\end{align*}
These two equations completely determine $\widetilde{J}(s_0,t_0,p_0)$ and one deduces from them that
\begin{equation*}
    \widetilde{J}(s_0,t_0,p_0) \partial_t = -\partial_s - L_{(s_0,t_0,p_0)}^{-1} \circ J^{s_0}_{t_0}(p_0) (X_H(s_0,t_0,p_0)).
\end{equation*}

Notice that $\widetilde{J}(s_0,t_0,p_0)$ leaves the horizontal sub-space $H_{(s_0,t_0,p_0)}$ invariant and that, $\widetilde{J}(s_0,t_0,p_0)$ restricted to $H_{(s_0,t_0,p_0)}$ is the pullback of $J^{s_0}_{t_0}(p_0)$ by the map $L_{(s_0,t_0,p_0)}$. Because $\widetilde{J}$ leaves invariant the horizontal sub-spaces, the surfaces $\{s_0\}\times \{t_0\} \times \Sigma$ are holomorphic for $\widetilde{J}$.

The next proposition gives the promised relation between Floer cylinders for $(H,J^s_t)$ and holomorphic cylinders on $(\R \times S^1 \times \Sigma,\widetilde{J})$. For this we let $(s,t)$ be coordinates on $\R \times S^1$ and let $j$ be the complex structure on $\R \times S^1$ that satisfies $j\partial_s= \partial_t$.

\begin{prop} \label{prop:Floertoholmorphic}
A map $u:\R \times S^1 \to \Sigma$ is a Floer cylinder for $(H,J^s_t)$, if and only if, its lift $\widetilde{u}: (\R \times S^1,j) \to (\R \times S^1 \times \Sigma,\widetilde{J})$ defined as $\widetilde{u}(s,t) := (s,t,u(s,t))$ is a holomorphic curve. If the $1$-periodic orbits $\gamma$ of $H_{s_H}$ and $\gamma'$ of $H_{-s_H}$  are respectively the negative and  positive limits of the Floer cylinders $u$, then the lift $\widetilde{u}$ is negatively asymptotic to the suspension $\{(t,\gamma(t)) \ | \ t \in S^1 \}$ of $\gamma$ and positively asymptotic to the suspension $\{(t,\gamma'(t)) \ | \ t \in S^1 \}$ of $\gamma'$; for $s_H$ as in \eqref{eq:Hasymp}.
\end{prop}

\proof: The proof is a direct computation. See for example section 4.12 of \cite{EKP}. \qed

\begin{rem}
One can also show that any holomorphic cylinder in $(\R \times S^1 \times \Sigma,\widetilde{J})$ that has one positive puncture asymptotic to the suspension $\{(t,\gamma(t)) \ | \ t \in S^1 \}$ of a $1$-periodic orbit $\gamma$ of $H_{s_H}:= H(s_H,\cdot, \cdot)$ is the lift of a Floer cylinder of $(H,J^{s}_t)$; see again section 4.12 of  \cite{EKP}. However, we do not need this result and Proposition \ref{prop:Floertoholmorphic} is enough for all our arguments. 
\end{rem}

\section{Proof of theorems \ref{thm:disk} and \ref{thm:disk0}} \label{sec:proofthmdisk}

In this section we prove theorems \ref{thm:disk} and \ref{thm:disk0}.

\textit{Proof of Theorem \ref{thm:disk0}:}

\textit{Step 1.}

We first define $\mathrm{Spec}(\mathcal{Y}_\oplus):= \cup_{1 \leq i \leq k} \mathcal{A}_{H_\oplus}(\gamma_i)$. 
For each number $\kappa \in \mathrm{Spec}(\mathcal{Y}_\oplus)$ we denote by $\mathcal{Y}^\kappa_\oplus \subset \mathcal{Y}_\oplus$ the subset of elements of $\mathcal{Y}_\oplus$ whose action is $\kappa$. We denote by $n_\kappa$ the cardinality of $\mathcal{Y}^\kappa_\oplus $. We denote by $\{\gamma^\kappa_1,...,\gamma^\kappa_{n_\kappa}\}$ the elements of $\mathcal{Y}^\kappa_\oplus $.\footnote{Since $\mathcal{Y}^\kappa_\oplus  \subset \mathcal{Y}_\oplus = \{\gamma_1,...,\gamma_k\} $, we are actually renaming the elements of $\mathcal{Y}_\oplus$.}  

Fix $\epsilon>0$ as in the statement of the theorem and then consider for every $\kappa \in \mathrm{Spec}(\mathcal{Y}_\oplus)$ the Floer homology $HF^{(\kappa - \epsilon, \kappa+ \epsilon)}(H_\oplus)$. We choose the smooth $S^1$-family $J_t$ to be regular for all chain-complexes $CF^{(\kappa - \epsilon, \kappa+ \epsilon)}(H_\oplus)$ and $CF^{(\kappa - 2\epsilon, \kappa+ 2\epsilon)}(H_\ominus)$, so that we can use always the same pairs $(H_\oplus,J_t)$ to the define the Floer differential $d^{J_t}$ on $CF^{(\kappa - \epsilon, \kappa+ \epsilon)}(H_\oplus)$ and $(H_\ominus,J_t)$ to define the Floer differential $d^{J_t}$ on $CF^{(\kappa - 2\epsilon, \kappa+ 2\epsilon)}(H_\ominus)$.  We notice that since all the elements of $CF^{(\kappa - \epsilon, \kappa+ \epsilon)}(H_\oplus)$ have the same action the differential vanishes on $CF^{(\kappa - \epsilon, \kappa+ \epsilon)}(H_\oplus)$. Moreover, since $\mathcal{Y}_\oplus$ is $100\epsilon$-isolated, we conclude that the rank of $CF^{(\kappa - \epsilon, \kappa+ \epsilon)}(H_\oplus)$ and $HF^{(\kappa - \epsilon, \kappa+ \epsilon)}(H_\oplus)$ coincide with $\# (\mathcal{Y}^\kappa_\oplus )$.

We now apply Proposition \ref{prop:chainisodisk0} for the chain-complexes $CF^{(\kappa - \epsilon, \kappa+ \epsilon)}(H_\oplus)$ and $CF^{(\kappa - 2\epsilon, \kappa+ 2\epsilon)}(H_\ominus)$ for all $\kappa \in \mathrm{Spec}(\mathcal{Y}_\oplus)$. Notice that the construction of homotopies $G$ and $\widehat{G}$ does not depend on $\kappa$, but only on $\epsilon$ and $H_\ominus$. We can thus choose generically the same pairs $(G,J^s_t)$ and $(\widehat{G},\widehat{J}^s_t)$ to induce maps from $CF^{(\kappa - \epsilon, \kappa+ \epsilon)}(H_\oplus)$ to $CF^{(\kappa - 2\epsilon, \kappa+ 2\epsilon)}(H_\ominus)$, and from  $CF^{(\kappa - 2\epsilon, \kappa+ 2\epsilon)}(H_\ominus)$ to $CF^{(\kappa - \epsilon, \kappa+ \epsilon)}(H_\oplus)$, for all $\kappa \in \mathrm{Spec}(\mathcal{Y}_\oplus)$. 

We then obtain chain maps $\Psi_G : CF^{(\kappa - \epsilon, \kappa+ \epsilon)}(H_\oplus) \to CF^{(\kappa - 2\epsilon, \kappa+ 2\epsilon)}(H_\ominus) $ and $\Psi_{\widehat{G}}:CF^{(\kappa - 2\epsilon, \kappa+ 2\epsilon)}(H_\ominus) \to CF^{(\kappa - \epsilon, \kappa+ \epsilon)}(H_\oplus)$. Since $\Psi_{\widehat{G}} \circ \Psi_G: CF^{(\kappa - \epsilon, \kappa+ \epsilon)}(H_\oplus) \to CF^{(\kappa - \epsilon, \kappa+ \epsilon)}(H_\oplus) $ is chain isotopic to the identity and the differential $d^{J_t}$ vanishes, we conclude that the chain map   $\Psi_{\widehat{G}} \circ \Psi_G: CF^{(\kappa - \epsilon, \kappa+ \epsilon)}(H_\oplus) \to CF^{(\kappa - \epsilon, \kappa+ \epsilon)}(H_\oplus) $ is the identity.
It follows that $\Psi_G: CF^{(\kappa - \epsilon, \kappa+ \epsilon)}(H_\oplus) \to CF^{(\kappa - 2\epsilon, \kappa+ 2\epsilon)}(H_\ominus) $ is injective and that \linebreak $\Psi_{\widehat{G}}: CF^{(\kappa - 2\epsilon, \kappa+ 2\epsilon)}(H_\ominus) \to CF^{(\kappa - \epsilon, \kappa+ \epsilon)}(H_\oplus)$ is surjective.

Since $\Psi_g$ is injective we know that the dimension of $CF^{(\kappa - 2\epsilon, \kappa+ 2\epsilon)}(H_\ominus)$ is $\geq n_\kappa$. We let the $m_\kappa \geq n_\kappa$ be the dimension of $CF^{(\kappa - 2\epsilon, \kappa+ 2\epsilon)}(H_\ominus)$, and since $CF^{(\kappa - 2\epsilon, \kappa+ 2\epsilon)}(H_\ominus)$ is the $\Z_2$-vector space over $\mathcal{P}^1_{(\kappa - 2\epsilon, \kappa+ 2\epsilon)}(H_\ominus)$, we conclude that $ \# \mathcal{P}^1_{(\kappa - 2\epsilon, \kappa+ 2\epsilon)}(H_\ominus) = m_\kappa$. We denote by  $\{\sigma^\kappa_1,...,\sigma^\kappa_{m_\kappa} \}$ the elements of  $\mathcal{P}^1_{(\kappa - 2\epsilon, \kappa+ 2\epsilon)}(H_\ominus)$. 

We now introduce some terminology. For each $j\in \{1,...,n_\kappa\}$, the image $\Psi_G(\gamma^\kappa_j)$ can be written in a unique way as a sum of orbits in $\mathcal{P}^1_{(\kappa - 2\epsilon, \kappa+ 2\epsilon)}(H_\ominus)$: we use here that we are working with $\Z_2$-coefficients. We say that an orbit $\sigma^\kappa_l$ appears in $\Psi_G(\gamma^\kappa_j)$, if the orbit $\sigma^\kappa_l$ is one of the orbits which appear in the expression of $\Psi_G(\gamma^\kappa_j)$ written in the base $\mathcal{P}^1_{(\kappa - 2\epsilon, \kappa+ 2\epsilon)}(H_\ominus)$ of $CF^{(\kappa - 2\epsilon, \kappa+ 2\epsilon)}(H_\ominus)$. 

We define in a similar way orbits of $\mathcal{P}^1_{(\kappa - \epsilon, \kappa+ \epsilon)}(H_\oplus)$ which appear in $\Psi_{\widehat{G}}(\sigma^\kappa_l)$ for each $\sigma^\kappa_l \in \mathcal{P}^1_{(\kappa - 2\epsilon, \kappa+ 2\epsilon)}(H_\ominus)$. 
We make the following claim:
\begin{claim} \label{claim1}
It is possible to choose an injective map $\mathfrak{f}_\kappa: \{1,...,n_\kappa\} \to \{1,...,m_\kappa\}$ and a bijective map $\mathfrak{g}_\kappa: \{1,...,n_\kappa\} \to \{1,...,n_\kappa\}$ such that for each $i \in \{1,...,n_\kappa\}$:
\begin{itemize}
    \item  the orbit $\sigma^\kappa_{\mathfrak{f}_\kappa(i)}$ appears in $\Psi_G(\gamma^\kappa_i) $, and the orbit $\gamma^\kappa_{\mathfrak{g}_\kappa(i)}$ appears in $\Psi_{\widehat{G}}(\sigma^\kappa_{\mathfrak{f}_\kappa(i)}) $.
\end{itemize}
\end{claim}
It is clear that the claim will follow from the following combinatorial lemma, whose proof is presented in Appendix \ref{appendix:comblemma}.

\begin{lem} \label{lem:combinatorial}
Let $V$ and $R$ be finite dimensional $\Z_2$-vector spaces whose dimensions we denote by $n$ and $m$, respectively. Let $\{v_1,...,v_{n}\}$ and $\{r_1,...,r_{m}\}$ be basis of $V$ and $R$, respectively, and $\mathfrak{F}: V \to R$ and $\mathfrak{G}: R \to V $ be linear maps such that $\mathfrak{G} \circ \mathfrak{F}$ is an isomorphism. Then, it is possible to find an injective map $\mathfrak{f}: \{1,...,n\} \to \{1,...,m\}$ and a bijective map $\mathfrak{g}: \{1,...,n\} \to \{1,...,n\}$ such that for each $i\in \{1,...,n\}$:
\begin{itemize}
    \item  the element $r_{\mathfrak{f}(i)}$ appears in $\mathfrak{F}(v_i) $, and the element $v_{\mathfrak{g}(i)}$ appears in $\mathfrak{G}(r_{\mathfrak{f}(i)}) $.
\end{itemize}

\end{lem}

\textit{Step 2.}

For each $\gamma^\kappa_i \in \mathcal{Y}^\kappa_\oplus$ we consider the $1$-periodic orbit $\sigma^\kappa_{\mathfrak{f}_\kappa(i)}$. Since $\sigma^\kappa_{\mathfrak{f}_\kappa(i)}$ appears in $\Psi_{G}(\gamma^\kappa_i)$, there exists a Floer cylinder $u^1_{\kappa,i}$ of $(G,J^s_t)$  which is negatively asymptotic to $\gamma^\kappa_i$ and positively asymptotic to $\sigma^\kappa_{\mathfrak{f}_\kappa(i)}$. 

Similarly, since $\gamma^\kappa_{\mathfrak{g}_\kappa(i)}$ appears in $\Psi_{\widehat{G}}(\sigma^\kappa_{\mathfrak{f}_\kappa(i)})$, there exists a Floer cylinder $u^2_{\kappa,i}$ of $(\widehat{G},\widehat{J}^s_t)$ which is negatively asymptotic to $\sigma^\kappa_{\mathfrak{f}_\kappa(i)}$ and positively asymptotic to $\gamma^\kappa_{\mathfrak{g}_\kappa(i)}$.

Since the maps $\mathfrak{g}_{\kappa}$ are permutations, there exists a natural number $M\geq 1$ such that for every $\kappa \in \mathrm{Spec}^1(\mathcal{Y}_\oplus)$ the $M$-th iterate $\mathfrak{g}^M_{\kappa} $  equals the identity.
For each pair $\kappa \in \mathrm{Spec}^1(\mathcal{Y}_\oplus)$ and $i \in {1,...,n_\kappa}$, we let $\widetilde{u}^1_{\kappa,i}: \R \times S^1 \to \R \times S^1 \times \D$ be the lift of the Floer cylinder ${u}^1_{\kappa,i}$. The cylinder $\widetilde{u}^1_{\kappa,i}$ is $\widetilde{J}_{G,J^s_t}$-holomorphic, where $\widetilde{J}_{G,J^s_t}$ is the almost complex structure on $\R \times S^1 \times \D$ constructed in Section \ref{sec:holcylinders}. 

\begin{claim} \label{claim2}
We claim that the cylinders $\widetilde{u}^1_{\kappa,i}$ have no intersections. More precisely, we have that
\begin{itemize}
    \item if $\kappa \neq \kappa'$ are elements of $\mathrm{Spec}^1(\mathcal{Y}_\oplus)$, then for every $i \in \{1,...,n_\kappa\}$ and $j \in \{1,...,n_{\kappa'}\}$, the cylinders $\widetilde{u}^1_{\kappa,i}$ and $\widetilde{u}^1_{\kappa',j}$ have no intersections,
    \item if $\kappa \in \mathrm{Spec}^1(\mathcal{Y}_\oplus)$ and $i\neq j$ are elements of $\{1,...,n_{\kappa'}\}$, then $\widetilde{u}^1_{\kappa,i}$ and $\widetilde{u}^1_{\kappa,j}$ have no intersections. 
\end{itemize}

\end{claim}

Before proving the claim we explain why it implies Theorem \ref{thm:disk0}.
Denote by $\mathcal{Y}_{\ominus}$ the collection of orbits $\sigma^\kappa_{\mathfrak{f}_\kappa(i)}$, $\kappa \in \mathrm{Spec}^1(\mathcal{Y}_\oplus), i \in \{1,...,n_\kappa\}$. 
For each $\kappa \in \mathrm{Spec}^1(\mathcal{Y}_\oplus)$, $i \in \{1,...,n_\kappa\}$ and $s\in \R$ we let $\xi^s_{\kappa,i}: S^1 \to S^1 \times \D$ be defined by $$ \xi^s_{\kappa,i}(\cdot) := \widetilde{u}^1_{\kappa,i}(s,\cdot). $$
It is easy to see that $\xi^s_{\kappa,i}$ is a knot embedded in $S^1 \times \D$ for each $s\in \R$ which intersect each of disks $\{t\} \times \D$ transversely and only once. 
We also define $\xi^\oplus_{\kappa,i}: S^1 \to S^1 \times \Sigma$ by
$$ \xi^\oplus_{\kappa,i}(t):= (t,\gamma^\kappa_i(t)),$$
and $\xi^\ominus_{\kappa,i}: S^1 \to S^1 \times \Sigma$ by 
$$ \xi^\ominus_{\kappa,i}(t):= (t,\sigma^\kappa_{\mathfrak{f}_\kappa(i)}(t)).$$
The braid $\mathcal{B}(\mathcal{Y}_\oplus) $ equals the disjoint union $ \cup_{\kappa \in \mathrm{Spec}^1(\mathcal{Y}_\oplus)} \cup_{i \in \{1,...,n_\kappa\}}\xi^\oplus_{\kappa,i}$. 

Since $\mathfrak{g}_\kappa$ is always a bijection for each $\kappa$, the knots $\xi^\ominus_{\kappa,i}$ are also disjoint. It follows that $\mathcal{B}(\mathcal{Y}_\ominus) = \cup_{\kappa \in \mathrm{Spec}^1(\mathcal{Y}_\oplus)} \cup_{i \in \{1,...,n_\kappa\}}\xi^\ominus_{\kappa,i}$ is a braid with the same number of strands as $\mathcal{B}(\mathcal{Y}_\oplus)$.

The asymptotic behaviour of the Floer cylinders $\widetilde{u}^1_{\kappa,i}$ tells us that 
\begin{itemize}
    \item $\xi^s_{\kappa,i}$ converges in $C^\infty$ to $\xi^\oplus_{\kappa,i}$,
    \item $\xi^s_{\kappa,i}$ converges in $C^\infty$ to $\xi^\ominus_{\kappa,i}$.
\end{itemize}
It follows from this that letting $\xi^{-\infty}_{\kappa,i}:=\xi^\oplus_{\kappa,i}$ and $\xi^{+\infty}_{\kappa,i}:=\xi^\ominus_{\kappa,i}$, the families $(\xi^s_{\kappa,i})_{s \in \overline{\R}}$ define isotopies between  $\xi^\oplus_{\kappa,i}$ and $\xi^\ominus_{\kappa,i}$. 

For each $s \in \overline{\R} $, let  $\mathcal{B}^s = \cup_{\kappa \in \mathrm{Spec}^1(\mathcal{Y}_\oplus)} \cup_{i \in \{1,...,n_\kappa\}}\xi^s_{\kappa,i}$. To prove that $(\mathcal{B}^s)_{s \in \overline{R}}$ gives a braid isotopy between $\mathcal{B}(\mathcal{Y}_\oplus) $ and $\mathcal{B}(\mathcal{Y}_\ominus) $, it suffices then to show that 
\begin{itemize}
\item if $\kappa \neq \kappa'$ are elements of $\mathrm{Spec}^1(\mathcal{Y}_\oplus)$, then for every $i \in \{1,...,n_\kappa\}$ and $j \in \{1,...,n_{\kappa'}\}$, the knots $\xi^s_{\kappa,i}$ and $\xi^s_{\kappa',j}$ are disjoint for each $s \in \R$,
\item if $\kappa \in \mathrm{Spec}^1(\mathcal{Y}_\oplus)$ and $i\neq j$ are elements of $\{1,...,n_{\kappa'}\}$, then the knots $\xi^s_{\kappa,i}$ and $\xi^s_{\kappa,j}$ are disjoint for each $s \in \R$. 
\end{itemize}
But from the definition of the knots $\xi^s_{\kappa_i}$, it is clear that these two conditions are equivalent to Claim \ref{claim2}.

\textit{Step 3.}

We thus proceed to prove Claim \ref{claim2}, which will imply the theorem.

We consider now for each choice of $\kappa \in \mathrm{Spec}^1(\mathcal{Y}_\oplus)$ and $i \in \{1,...,n_\kappa\}$ the ordered $(2M-2)$-tuple $$(u^1_{\kappa,i},u^2_{\kappa,i}, u^1_{\kappa,\mathfrak{g}(i)},u^2_{\kappa,\mathfrak{g}(i)},u^1_{\kappa,\mathfrak{g}^2(i)},u^2_{\kappa,\mathfrak{g}^2(i)},...,u^2_{\kappa,\mathfrak{g}(i)},u^1_{\kappa,\mathfrak{g}^{M-1}(i)},u^2_{\kappa,\mathfrak{g}^{M-1}(i)}).$$

Because of the asymptotic behaviour of the maps $\widetilde{u}^\nu_{\kappa,i}$ they can be compactified to cylinders $\overline{u}^\nu_{\kappa,i}: \overline{\R} \times S^1 \to \overline{\R} \times S^1 \times \D $.

We now fix, once and for all, a homeomorphism $\mathcal{L}: [0,{1}] \to  \overline{\R}$ which is smooth in the interior of $[0,1]$ and satisfies $\mathcal{L}(0)= -\infty$ and $\mathcal{L}(1)= +\infty$. For $l \in \{0,...,M-1\}$ and $\nu \in \{1,2\}$ we let $\mathcal{L}^\nu_l: [2l + (\nu-1) ,2l + \nu] \to  \overline{\R}$ the homeomorphisms given by  $\mathcal{L}^\nu_l(u)=\mathcal{L}(u-(2l + (\nu -1)))$.

Using the map $\mathcal{L}^\nu_l$ we obtain homeomorphisms $${\mathfrak{L}^\nu_l}: [2l + (\nu-1) ,2l + \nu] \times S^1  \to \overline{\R} \times S^1$$ given by ${\mathfrak{L}^\nu_l}(s,t) = ({\mathcal{L}^\nu_l}(s),t)$ and $${\mathfrak{N}^\nu_l}: \overline{\R} \times S^1 \times \D \to  [2l + (\nu-1) ,2l + \nu] \times S^1  \times \D $$ given by $\mathfrak{N}^\nu_l(s,t,p) = ((\mathcal{L}^\nu_l)^{-1}(s),t,p)$.
We then define $$\overline{v}^{\nu,l}_{\kappa,i}:  [2l + (\nu-1) ,2l + \nu] \times S^1  \to [2l + (\nu-1) ,2l + \nu] \times S^1  \times \D$$ by $$ \overline{v}^{\nu,l}_{\kappa,i} = \mathfrak{N}^\nu_l \circ \widetilde{u}^\nu_{\kappa,\mathfrak{g}^l_\kappa(i)} \circ {\mathcal{L}^\nu_l}.  $$  

We notice that if $i\neq i'$, intersections between two cylinders $\overline{v}^{\nu,l}_{\kappa,i}$ and $\overline{v}^{\nu',l'}_{\kappa',i'}$ can only occur if $l=l'$. In this case, any such intersection is positive, because of positivity of intersection for holomorphic curves and the method we used to construct these cylinders from the holomorphic cylinders $\widetilde{u}^\nu_{\kappa,\mathfrak{g}^l_\kappa(i)} $.

We are ready to define $$ \overline{U}^\kappa_i: [0,2M-2] \times S^1 \to [0,2M-2] \times S^1 \times \D$$ by the formula $\overline{U}^\kappa_i(s,t)= \overline{v}^{\nu,l}_{\kappa,i}$ if $s \in [2l + (\nu-1) ,2l + \nu]$. The cylinders $\overline{U}^\kappa_i$ should be thought of as the concatenation of the cylinders forming the ordered tuple $2M-2$-tuple $$(\overline{u}^1_{\kappa,i},\overline{u}^2_{\kappa,i}, \overline{u}^1_{\kappa,\mathfrak{g}(i)},\overline{u}^2_{\kappa,\mathfrak{g}(i)},\overline{u}^1_{\kappa,\mathfrak{g}^2(i)},\overline{u}^2_{\kappa,\mathfrak{g}^2(i)},...,\overline{u}^2_{\kappa,\mathfrak{g}(i)},\overline{u}^1_{\kappa,\mathfrak{g}^{M-1}(i)},\overline{u}^2_{\kappa,\mathfrak{g}^{M-1}(i)}).$$ Because the positive asymptotic limit of an element of the tuple coincides with the negative asymptotic limit of the next element, the map 
$\overline{U}^\kappa_i$ is indeed continuous, and smooth when $s$ is in the interior of the intervals $[2l + (\nu-1) ,2l + \nu]$.

\textit{Step 4.}

We finish the proof of Claim \ref{claim2}. 

We argue by contradiction. We assume that there exist $\kappa,\kappa'$ and $i,j$ such that either $\kappa \neq \kappa'$ or $i \neq j$ and that $\widetilde{u}^1_{\kappa,i}$ and $\widetilde{u}^1_{\kappa',j}$ intersect. In this case the long cylinders $\overline{U}^\kappa_i$ and $\overline{U}^{\kappa'}_j$ must also intersect. 

We claim that all intersections of these cylinders count positively. Indeed, these intersections must occur in the open sets $(2l + (\nu-1) ,2l + \nu ) \times S^1  \times \D$ where the cylinders are smooth and holomorphic, positivity of intersections for holomorphic curves imply that these intersections count positively. The reason why the intersections can only occur in these open sets is that the $l$-th element of the tuples $$(\overline{u}^1_{\kappa,i},\overline{u}^2_{\kappa,i}, \overline{u}^1_{\kappa,\mathfrak{g}(i)},\overline{u}^2_{\kappa,\mathfrak{g}(i)},\overline{u}^1_{\kappa,\mathfrak{g}^2(i)},\overline{u}^2_{\kappa,\mathfrak{g}^2(i)},...,\overline{u}^2_{\kappa,\mathfrak{g}(i)},\overline{u}^1_{\kappa,\mathfrak{g}^{M-1}(i)},\overline{u}^2_{\kappa,\mathfrak{g}^{M-1}(i)})$$ and $$(\overline{u}^1_{\kappa',j},\overline{u}^2_{\kappa',j}, \overline{u}^1_{\kappa',\mathfrak{g}(j)},\overline{u}^2_{\kappa,\mathfrak{g}(j)},\overline{u}^1_{\kappa',\mathfrak{g}^2(j)},\overline{u}^2_{\kappa',\mathfrak{g}^2(j)},...,\overline{u}^2_{\kappa',\mathfrak{g}^2(j)},\overline{u}^1_{\kappa',\mathfrak{g}^{M-1}(j)},\overline{u}^2_{\kappa',\mathfrak{g}^{M-1}(j)})$$ have different positive and negative asymptotic limits.

Thus, the assumption that $\widetilde{u}^1_{\kappa,i}$ and $\widetilde{u}^1_{\kappa',j}$ intersect implies that $\overline{U}^\kappa_i$ and $\overline{U}^{\kappa'}_j$ have positive intersection number. Observe that $\overline{U}^\kappa_i(2M-2,\cdot)= \gamma^\kappa_i(\cdot)$ and $\overline{U}^\kappa_i(0,\cdot)= \gamma^\kappa_i(\cdot)$. Similarly, $\overline{U}^{\kappa'}_j(2M-2,\cdot)= \gamma^{\kappa'}_j(\cdot)$ and $\overline{U}^{\kappa'}_j(0,\cdot)= \gamma^{\kappa'}_j(\cdot)$.

Let $\overline{V}^\kappa_i: [0,2M-2] \times S^1 \to [0,2M-2] \times S^1 \times \D$ be the trivial cylinder over $\xi^\oplus_{\kappa,i}$, given by $\overline{V}^\kappa_i(s,t)= (s,t,\gamma^\kappa_i(t))$.
Because $\D$ is contractible, there exist a homotopy $H^\kappa_i: [0,1] \times [0,2M-2] \times S^1 \to  [0,2M-2] \times S^1 \times \D $ satisfying:
\begin{itemize}
    \item $H^\kappa_i(a,2M-2,\cdot) = (2M-2, \xi^\oplus_{\kappa,i}(\cdot))$ for every $a\in [0,1]$,
    \item $H^\kappa_i(a,0,\cdot) = (0,\xi^\oplus_{\kappa,i}(\cdot))$ for every $a\in [0,1]$,
    \item $H^\kappa_i(0,\cdot,\cdot) = \overline{V}^\kappa_i(\cdot,\cdot)$ and $H^\kappa_i(1,\cdot,\cdot)= \overline{U}^\kappa_i(\cdot,\cdot)$.
\end{itemize}
We refer to $H^\kappa_i$ as a homotopy of cylinders with the same boundary between $\overline{V}^\kappa_i$ and $\overline{U}^\kappa_i$.
We consider similarly a homotopy of cylinders with the same boundary between $\overline{V}^{\kappa'}_j$ and $\overline{U}^{\kappa'}_j$.

Since either $\kappa \neq \kappa'$ or $i \neq j$, the knots $\xi^\oplus_{\kappa,i}$ and $\xi^\oplus_{\kappa',j}$ are disjoint: it follows that the cylinders $H^\kappa_i(a,\cdot,\cdot)$ and $H^{\kappa'}_j(a,\cdot,\cdot)$ have disjoint boundaries. The intersection number of the cylinders $H^\kappa_i(a,\cdot,\cdot)$ and $H^{\kappa'}_j(a,\cdot,\cdot)$ does not depend  on $a$. But, this intersection number is clearly $0$ for $a=0$. But this is in contradiction with the positivity of the intersection number of  $\overline{U}^{\kappa'}_j$ and $\overline{U}^\kappa_i$, which had followed from the assumption that $\widetilde{u}^1_{\kappa,i}$ and $\widetilde{u}^1_{\kappa',j}$ intersect. 

This contradiction shows that $\widetilde{u}^1_{\kappa,i}$ and $\widetilde{u}^1_{\kappa',j}$ do not intersect, finishing the proof of Claim \ref{claim2} and completing the proof of Theorem \ref{thm:disk0}. \qed

\textit{Proof of Theorem \ref{thm:disk}.} 

The proof of Theorem \ref{thm:disk} is identical to the proof of Theorem \ref{thm:disk0}, with the only modification that one applies Proposition \ref{prop:chainisodiskc} where in the previous proof one applied Proposition \ref{prop:chainisodisk0}. \qed

\section{Proof of Theorem \ref{thm:main} }

To prove  Theorem \ref{thm:main} we must consider several different cases.
During the proof we will refer to the proof steps and arguments used in the proof of Theorem \ref{thm:disk}, so the reader should first read the proof of that theorem. 

We start treating the case in which $\Sigma=S^2$.

\subsection{Proof of Theorem \ref{thm:main} in case $\Sigma = S^2$.} \label{sec:proofsphere}

\

\textit{Step 1}.

Note that one can choose capping discs $\mathcal{D}_{\gamma_i}, i=1, \ldots k$ of the orbits in $\mathcal{Y}_{\oplus}$ such that $\mathcal{A}_{H_\oplus}(\gamma_i, \mathcal{D}_{\gamma_i}) - \mathcal{A}_{H_\oplus}(\gamma_j, \mathcal{D}_{\gamma_j}) = \Delta(\gamma_i,\gamma_j)$, for all $i,j \in \{1, \ldots, k\}$. 
Denote by  $\widetilde{\mathcal{Y_{\oplus}}} = \{(\gamma_1, \mathcal{D}_{\gamma_1}), \ldots, (\gamma_k, \mathcal{D}_{\gamma_k})\}$ the set of the resulting pairs, and by $\mathrm{Spec}(\widetilde{\mathcal{Y}_\oplus}):= \cup_{1 \leq i \leq k} \mathcal{A}_{H_\oplus}(\gamma_i, \mathcal{D}_{\gamma_i})$ their action values. 
For each number $\kappa \in \mathrm{Spec}(\widetilde{\mathcal{Y}_\oplus})$ we denote by $\widetilde{\mathcal{Y}^\kappa_\oplus} \subset \widetilde{\mathcal{Y}_\oplus}$ the subset of elements of $\widetilde{\mathcal{Y}_\oplus}$ whose action is $\kappa$. We denote by $n_\kappa$ the cardinality of $\widetilde{\mathcal{Y}^\kappa_\oplus} $, and let $\{(\gamma^\kappa_1,\mathcal{D}_{\gamma^\kappa_1}),...,(\gamma^\kappa_{n_\kappa},\mathcal{D}_{\gamma^\kappa_{n_\kappa}})\}$ be the elements of $\widetilde{\mathcal{Y}^\kappa_\oplus} $.

Fix $\epsilon>0$ as in the statement of the theorem and then consider for every $\kappa \in \mathrm{Spec}(\widetilde{\mathcal{Y}_\oplus})$ the Floer homology $HF^{(\kappa - \epsilon, \kappa+ \epsilon)}(H_\oplus)$. We choose the smooth $S^1$-family $J_t$ to be regular for all chain-complexes $CF^{(\kappa - \epsilon, \kappa+ \epsilon)}(H_\oplus)$ and $CF^{(\kappa - 2\epsilon, \kappa+ 2\epsilon)}(H_\ominus)$, so that we can use always the same pairs $(H_\oplus,J_t)$ to define the Floer differential $d^{J_t}$ on $CF^{(\kappa - \epsilon, \kappa+ \epsilon)}(H_\oplus)$ and $(H_\ominus,J_t)$ to define the Floer differential $d^{J_t}$ on $CF^{(\kappa - 2\epsilon, \kappa+ 2\epsilon)}(H_\ominus)$.  We notice that since all the elements of $CF^{(\kappa - \epsilon, \kappa+ \epsilon)}(H_\oplus)$ have the same action the differential vanishes on $CF^{(\kappa - \epsilon, \kappa+ \epsilon)}(H_\oplus)$. Moreover, since $\mathcal{Y}_\oplus$ is $100\epsilon$-isolated, we conclude that the rank of $CF^{(\kappa - \epsilon, \kappa+ \epsilon)}(H_\oplus)$ and $HF^{(\kappa - \epsilon, \kappa+ \epsilon)}(H_\oplus)$ coincide with $\# (\widetilde{\mathcal{Y}^\kappa_\oplus} )$.

We now apply Proposition \ref{prop:chainisosphere} for the chain-complexes $CF^{(\kappa - \epsilon, \kappa+ \epsilon)}(H_\oplus)$ and $CF^{(\kappa - 2\epsilon, \kappa+ 2\epsilon)}(H_\ominus)$ for all $\kappa \in \mathrm{Spec}(\widetilde{\mathcal{Y}_\oplus})$. Notice that the construction of homotopies $G$ and $\widehat{G}$ does not depend on $\kappa$, but only on $\epsilon$ and $H_\ominus$. We can thus choose generically the same pairs $(G,J^s_t)$ and $(\widehat{G},\widehat{J}^s_t)$ to induce maps from $CF^{(\kappa - \epsilon, \kappa+ \epsilon)}(H_\oplus)$ to $CF^{(\kappa - 2\epsilon, \kappa+ 2\epsilon)}(H_\ominus)$, and from  $CF^{(\kappa - 2\epsilon, \kappa+ 2\epsilon)}(H_\ominus)$ to $CF^{(\kappa - \epsilon, \kappa+ \epsilon)}(H_\oplus)$, for all $\kappa \in \mathrm{Spec}(\widetilde{\mathcal{Y}_\oplus})$. 

We then obtain chain maps $\Psi_G : CF^{(\kappa - \epsilon, \kappa+ \epsilon)}(H_\oplus) \to CF^{(\kappa - 2\epsilon, \kappa+ 2\epsilon)}(H_\ominus) $ and $\Psi_{\widehat{G}}:CF^{(\kappa - 2\epsilon, \kappa+ 2\epsilon)}(H_\ominus) \to CF^{(\kappa - \epsilon, \kappa+ \epsilon)}(H_\oplus)$. Since $\Psi_{\widehat{G}} \circ \Psi_G: CF^{(\kappa - \epsilon, \kappa+ \epsilon)}(H_\oplus) \to CF^{(\kappa - \epsilon, \kappa+ \epsilon)}(H_\oplus) $ is chain isotopic to the identity and the differential $d^{J_t}$ vanishes, we conclude that the chain map   $\Psi_{\widehat{G}} \circ \Psi_G: CF^{(\kappa - \epsilon, \kappa+ \epsilon)}(H_\oplus) \to CF^{(\kappa - \epsilon, \kappa+ \epsilon)}(H_\oplus) $ is the identity.
It follows that $\Psi_G: CF^{(\kappa - \epsilon, \kappa+ \epsilon)}(H_\oplus) \to CF^{(\kappa - 2\epsilon, \kappa+ 2\epsilon)}(H_\ominus) $ is injective and that \linebreak $\Psi_{\widehat{G}}: CF^{(\kappa - 2\epsilon, \kappa+ 2\epsilon)}(H_\ominus) \to CF^{(\kappa - \epsilon, \kappa+ \epsilon)}(H_\oplus)$ is surjective.

Since $\Psi_g$ is injective we know that the dimension of $CF^{(\kappa - 2\epsilon, \kappa+ 2\epsilon)}(H_\ominus)$ is $\geq n_\kappa$. We let the $m_\kappa \geq n_\kappa$ be the dimension of $CF^{(\kappa - 2\epsilon, \kappa+ 2\epsilon)}(H_\ominus)$, and since $CF^{(\kappa - 2\epsilon, \kappa+ 2\epsilon)}(H_\ominus)$ is the $\Z_2$-vector space over $\mathcal{P}^1_{(\kappa - 2\epsilon, \kappa+ 2\epsilon)}(H_\ominus)$, we conclude that $ \# \mathcal{P}^1_{(\kappa - 2\epsilon, \kappa+ 2\epsilon)}(H_\ominus) = m_{\kappa}$. We write   $\{(\sigma^\kappa_1,\mathcal{D}_{\sigma^\kappa_1}),...,(\sigma^\kappa_{m_\kappa},\mathcal{D}_{\sigma^\kappa_{m_\kappa}}) \}$ the elements of  $\mathcal{P}^1_{(\kappa - 2\epsilon, \kappa+ 2\epsilon)}(H_\ominus)$. 

 For each $j\in \{1,...,n_\kappa\}$, the image $\Psi_G((\gamma^\kappa_j,\mathcal{D}_{\gamma^\kappa_j}))$ can be written in a unique way as a sum of pairs in $\mathcal{P}^1_{(\kappa - 2\epsilon, \kappa+ 2\epsilon)}(H_\ominus)$. We say that a pair $(\sigma^\kappa_l,\mathcal{D}_{\sigma^\kappa_l})$ appears in $\Psi_G((\gamma^\kappa_j,\mathcal{D}_{\gamma^\kappa_j}))$, if the pair $(\sigma^\kappa_l,\mathcal{D}_{\sigma^\kappa_l})$ is one of the pairs which appear in the expression of $\Psi_G((\gamma^\kappa_j,\mathcal{D}_{\gamma^\kappa_j}))$ written in the base $\mathcal{P}^1_{(\kappa - 2\epsilon, \kappa+ 2\epsilon)}(H_\ominus)$ of $CF^{(\kappa - 2\epsilon, \kappa+ 2\epsilon)}(H_\ominus)$. 

We define in a similar way pairs of $\mathcal{P}^1_{(\kappa - \epsilon, \kappa+ \epsilon)}(H_\oplus)$ which appear in $\Psi_{\widehat{G}}((\sigma^\kappa_l,\mathcal{D}_{\sigma^\kappa_l})$ for each $(\sigma^\kappa_l,\mathcal{D}_{\sigma^\kappa_l})\in \mathcal{P}^1_{(\kappa - 2\epsilon, \kappa+ 2\epsilon)}(H_\ominus)$. Using Lemma \ref{lem:combinatorial} we know that:

\begin{fact} \label{fact1}
It is possible to choose an injective map $\mathfrak{f}_\kappa: \{1,...,n_\kappa\} \to \{1,...,m_\kappa\}$ and a bijective map $\mathfrak{g}_\kappa: \{1,...,n_\kappa\} \to \{1,...,n_\kappa\}$ such that for each $i \in \{1,...,n_\kappa\}$:
\begin{itemize}
    \item  the pair $(\sigma^\kappa_{\mathfrak{f}_\kappa(i)},\mathcal{D}_{\sigma^\kappa_{\mathfrak{f}_\kappa(i)}})$ appears in $\Psi_G((\gamma^\kappa_i,\mathcal{D}_{\gamma^\kappa_i})) $, and the orbit \linebreak $(\gamma^\kappa_{\mathfrak{g}_\kappa(i)},\mathcal{D}_{\gamma^\kappa_{\mathfrak{g}_\kappa(i)}})$ appears in $\Psi_{\widehat{G}}((\sigma^\kappa_{\mathfrak{f}_\kappa(i)},\mathcal{D}_{\sigma^\kappa_{\mathfrak{f}_\kappa(i)}})) $.
\end{itemize}
\end{fact}

\textit{Step 2.}

For each $(\gamma^\kappa_i,\mathcal{D}_{\gamma^\kappa_i}) \in \mathcal{Y}^\kappa_\oplus$ we consider the pair  $(\sigma^\kappa_{\mathfrak{f}_\kappa(i)},\mathcal{D}_{\sigma^\kappa_{\mathfrak{f}_\kappa(i)}})$. Since \linebreak $(\sigma^\kappa_{\mathfrak{f}_\kappa(i)},\mathcal{D}_{\sigma^\kappa_{\mathfrak{f}_\kappa(i)}})$ appears in $\Psi_{G}((\gamma^\kappa_i,\mathcal{D}_{\gamma^\kappa_i}))$, there exists a Floer cylinder $u^1_{\kappa,i}$ of $(G,J^s_t)$  which is negatively asymptotic to $(\gamma^\kappa_i,\mathcal{D}_{\gamma^\kappa_i})$ and positively asymptotic to $(\sigma^\kappa_{\mathfrak{f}_\kappa(i)},\mathcal{D}_{\sigma^\kappa_{\mathfrak{f}_\kappa(i)}})$. 

Similarly, there exists a Floer cylinder $u^2_{\kappa,i}$ of $(\widehat{G},\widehat{J}^s_t)$ which is negatively asymptotic to $(\sigma^\kappa_{\mathfrak{f}_\kappa(i)},\mathcal{D}_{\sigma^\kappa_{\mathfrak{f}_\kappa(i)}})$ and positively asymptotic to $(\gamma^\kappa_{\mathfrak{g}_\kappa(i)},\mathcal{D}_{\gamma^\kappa_{\mathfrak{g}_\kappa(i)}})$.

For each pair $\kappa \in \mathrm{Spec}^1(\mathcal{Y}_\oplus)$ and $i \in {1,...,n_\kappa}$, we let $\widetilde{u}^1_{\kappa,i}: \R \times S^1 \to \R \times S^1 \times \D$ be the lift of the Floer cylinder ${u}^1_{\kappa,i}$. The cylinder $\widetilde{u}^1_{\kappa,i}$ is $\widetilde{J}_{G,J^s_t}$-holomorphic, where $\widetilde{J}_{G,J^s_t}$ is the almost complex structure on $\R \times S^1 \times \D$ constructed in Section \ref{sec:holcylinders}. 
Reasoning as in Step 2 of the proof of Theorem \ref{thm:disk0} (see Section \ref{sec:proofthmdisk}), we conclude that Theorem \ref{thm:main} in the case $\Sigma=S^2$ will follow if we can show that the cylinders $u^1_{\kappa,i}$ and $u^1_{\kappa',j}$ have no intersections if either $\kappa\neq \kappa'$ or $i\neq j$.

\ 

\textit{Step 3.}

Since the maps $\mathfrak{g}_{\kappa}$ are permutations, there exists a natural number $M\geq 1$ such that for every $\kappa \in \mathrm{Spec}^1(\mathcal{Y}_\oplus)$ the $M$-th iterate $\mathfrak{g}^M_{\kappa} $  equals the identity. 


We consider now for each choice of $\kappa \in \mathrm{Spec}^1(\mathcal{Y}_\oplus)$ and $i \in \{1,...,n_\kappa\}$ the ordered $(2M-2)$-tuple $$(u^1_{\kappa,i},u^2_{\kappa,i}, u^1_{\kappa,\mathfrak{g}(i)},u^2_{\kappa,\mathfrak{g}(i)},u^1_{\kappa,\mathfrak{g}^2(i)},u^2_{\kappa,\mathfrak{g}^2(i)},...,u^2_{\kappa,\mathfrak{g}(i)},u^1_{\kappa,\mathfrak{g}^{M-1}(i)},u^2_{\kappa,\mathfrak{g}^{M-1}(i)}).$$

Because of the asymptotic behaviour of the maps $\widetilde{u}^\nu_{\kappa,i}$ they can be compactified to cylinders $\overline{u}^\nu_{\kappa,i}: \overline{\R} \times S^1 \to \overline{\R} \times S^1 \times \D $.

We now fix, once and for all, a homeomorphism $\mathcal{L}: [0,{1}] \to  \overline{\R}$ which is smooth in the interior of $[0,1]$ and satisfies $\mathcal{L}(0)= -\infty$ and $\mathcal{L}(1)= +\infty$. For $l \in \{0,...,M-1\}$ and $\nu \in \{1,2\}$ we let $\mathcal{L}^\nu_l: [2l + (\nu -1)  ,2l + \nu] \to  \overline{\R}, u \mapsto \mathcal{L}(u - (2l + (\nu-1)))$.

Using the map $\mathcal{L}^\nu_l$ we obtain homeomorphisms between $${\mathfrak{L}^\nu_l}: [2l + (\nu-1) ,2l + \nu ] \times S^1  \to \overline{\R} \times S^1$$ given by ${\mathfrak{L}^\nu_l}(s,t) = ({\mathcal{L}^\nu_l}(s),t)$ and $${\mathfrak{N}^\nu_l}: \overline{\R} \times S^1 \times S^2 \to  [2l + (\nu-1) ,2l + \nu ] \times S^1  \times S^2 $$ given by $\mathfrak{N}^\nu_l(s,t,p) = ((\mathcal{L}^\nu_l)^{-1}(s),t,p)$.
We then define $$\overline{v}^{\nu,l}_{\kappa,i}:  [2l + (\nu-1) ,2l + \nu ] \times S^1  \to [2l + (\nu-1) ,2l + \nu] \times S^1  \times S^2$$ by $$ \overline{v}^{\nu,l}_{\kappa,i} = \mathfrak{N}^\nu_l \circ \widetilde{u}^\nu_{\kappa,\mathfrak{g}^l_\kappa(i)} \circ {\mathcal{L}^\nu_l}.  $$  

We are ready to define $$ \overline{U}^\kappa_i: [0,2M-2] \times S^1 \to [0,2M-2] \times S^1 \times \D$$ by the formula $\overline{U}^\kappa_i(s,t)= \overline{v}^{\nu,l}_{\kappa,i}$ if $s \in [2l + \nu ,2l + \nu +1]$. One should think of $\overline{U}^\kappa_i$ as the concatenation of the cylinders  $\overline{v}^{\nu,l}_{\kappa,i}$.

We consider $$ {W}^{\kappa}_i: [0,2M-2] \times S^1 \to S^2 $$ be the projection of $\overline{U}^\kappa_i$ to $S^2$, i.e. $ {W}^{\kappa}_i := \Pi_{S^2} \circ \overline{U}^\kappa_i$ where $\Pi_{S^2}: [0,2M-2] \times S^1 \times S^2 \to S^2 $ is the projection onto the third coordinate. Let $${v}^{\nu,l}_{\kappa,i} := {u}^\nu_{\kappa,\mathfrak{g}^l_\kappa(i)} \circ {\mathcal{L}^\nu_l}.  $$ The cylinder ${v}^{\nu,l}_{\kappa,i}$ is obtained by compactifying the Floer cylinder $u^\kappa_{i}$ and taking its domain to be $[2l + (\nu-1) ,2l + \nu ]$. It is then clear that $${W}^{\kappa}_i := (\overline{v}^{2,M}_{\kappa,i} \# \overline{v}^{1,M}_{\kappa,i} ) \# (\overline{v}^{2,M-1}_{\kappa,i} \# \overline{v}^{1,M-1}_{\kappa,i}) \# ... \# (\overline{v}^{2,2}_{\kappa,i} \# \overline{v}^{1,2}_{\kappa,i} ) \# (\overline{v}^{2,1}_{\kappa,i} \# \overline{v}^{1,1}_{\kappa,i}). $$ 

We need the following definition. 
\begin{defn}
Let $(\gamma,[D_\gamma]) \in  \mathcal{P}^1(H)$ and $(\gamma',[D_{\gamma'}]) \in \mathcal{P}^1(H')$ for Hamiltonians $H:S^1 \times S^2 \to \R$ and $H':S^1 \times S^2 \to \R$, and let $V:[0,K] \times S^1 \to S^2  $ be a cylinder\footnote{Here $K>0$ is some positive real number.} such that $V: \{0\} \times S^1 = \gamma$ and $V: \{K\} \times S^1 = \gamma'$. We say that $V$ is a cylinder from $(\gamma,[D_\gamma])$ and $(\gamma',[D_{\gamma'}])$ if the disk $D_\gamma \# V $ obtained by gluing $D_\gamma$  and $V $ is homotopic to $D_{\gamma'}$ among disks filling $\gamma'$.
\end{defn}

\begin{claim} \label{claim:cylinder}
The cylinder ${W}^{\kappa}_i: [0,2M-2] \times S^1 \to S^2$ is a cylinder from $(\gamma^\kappa_i,[\mathcal{D}_{\gamma^\kappa_i}])$ to $(\gamma^\kappa_i,[\mathcal{D}_{\gamma^\kappa_i}])$.

\end{claim}

To prove this claim we need the following straightforward fact. 
\begin{fact} \label{fact2}
Let $a<b<c$ be real numbers,  $V_1:[a,b] \to  S^2$ be a cylinder from $(\gamma,[D_\gamma])$ to $(\gamma',[D_{\gamma'}])$, and $V_2:[b,c] \to S^2$ be a cylinder from $(\gamma',[D_{\gamma'}])$ to $(\gamma'',[D_{\gamma''}])$. Then, the concatenated cylinder $ V_2\# V_1 : [a,c] \to S^2 $ is a cylinder from $(\gamma,[D_\gamma])$ to $(\gamma'',[D_{\gamma''}])$.
\end{fact}

By construction, the Floer cylinders $u^1_{\kappa,\mathfrak{g}^l(i)}$ belong to the moduli space $\mathcal{M}((\gamma^\kappa_{\mathfrak{g}^{l}(i)}, [\mathcal{D}_{\gamma^\kappa_{\mathfrak{g}^{l}(i)}}]),(\sigma^\kappa_{\mathfrak{f}_\kappa(\mathfrak{g}^{l}(i))},[\mathcal{D}_{\sigma^\kappa_{\mathfrak{f}_\kappa(\mathfrak{g}^{l}(i))}}]), G,J^s_t)$. It follows that its compactification ${v}^{1,l}_{\kappa,i}$ is a cylinder from $(\gamma^\kappa_{\mathfrak{g}^{l}(i)}, [\mathcal{D}_{\gamma^\kappa_{\mathfrak{g}^{l}(i)}}])$ to $(\sigma^\kappa_{\mathfrak{f}_\kappa(\mathfrak{g}^{l}(i))},[\mathcal{D}_{\sigma^\kappa_{\mathfrak{f}_\kappa(\mathfrak{g}^{l}(i))}}])$. Similarly, one obtains that ${v}^{2,l}_{\kappa,i}$ is a cylinder from $(\sigma^\kappa_{\mathfrak{f}_\kappa(\mathfrak{g}^{l}(i))},[\mathcal{D}_{\sigma^\kappa_{\mathfrak{f}_\kappa(\mathfrak{g}^{l}(i))}}])$ to $(\gamma^\kappa_{\mathfrak{g}^{l+1}(i)}, [\mathcal{D}_{\gamma^\kappa_{\mathfrak{g}^{l+1}(i)}}])$.

Applying Fact \ref{fact2} multiple times we obtain that the concatenation $${W}^{\kappa}_i := (\overline{v}^{2,M}_{\kappa,i} \# \overline{v}^{1,M}_{\kappa,i} ) \# (\overline{v}^{2,M-1}_{\kappa,i} \# \overline{v}^{1,M-1}_{\kappa,i}) \# ... \# (\overline{v}^{2,2}_{\kappa,i} \# \overline{v}^{1,2}_{\kappa,i} ) \# (\overline{v}^{2,1}_{\kappa,i} \# \overline{v}^{1,1}_{\kappa,i}) $$ is a cylinder from $(\gamma^\kappa_{\mathfrak(i)}, [\mathcal{D}_{\gamma^\kappa_{\mathfrak(i)}}])$ to $(\gamma^\kappa_{\mathfrak{g}^{M}(i)}, [\mathcal{D}_{\gamma^\kappa_{\mathfrak{g}^{M}(i)}}]) = (\gamma^\kappa_{\mathfrak(i)}, [\mathcal{D}_{\gamma^\kappa_{\mathfrak(i)}}])$.

This proves Claim \ref{claim:cylinder}.

\ 

\textit{Step 4:}

Because  ${W}^{\kappa}_i$ is a cylinder from $(\gamma^\kappa_i,[\mathcal{D}_{\gamma^\kappa_i}])$ to $(\gamma^\kappa_i,[\mathcal{D}_{\gamma^\kappa_i}])$, it is homotopic to the trivial cylinder ${W}^{\kappa,i}_{\rm triv}: [0,2M-2] \times S^1 \to S^2$ over $\gamma^\kappa_i$ given by the formula $${W}^{\kappa,i}_{\rm triv}(s,t) = \gamma^\kappa_i(t), $$
among cylinders which map $\{0\} \times S^1$ and $\{2M-2\} \times S^1$ to $\gamma^\kappa_i$. 

We explain one way to prove the existence of such a homotopy. The fact that ${W}^{\kappa}_i$ is a cylinder from $(\gamma^\kappa_i,[\mathcal{D}_{\gamma^\kappa_i}])$ to $(\gamma^\kappa_i,[\mathcal{D}_{\gamma^\kappa_i}])$ implies that the sphere $S := \mathcal{D}_{\gamma^\kappa_i} \# {W}^{\kappa}_i \# (-\mathcal{D}_{\gamma^\kappa_i})$ is contractible.  
Let $p_0:= \gamma_i^\kappa(0)$. The cylinder ${W}^{\kappa}_i$ can be thought of as a loop in the free loop space $\Lambda(S^2)$ starting and ending at the $\gamma_i^\kappa$, and we denote by $[{W}^{\kappa}_i]$ the element of $\pi_1(\Lambda(S^2),\gamma_i^\kappa )$ represented by ${W}^{\kappa}_i$. Let $\Omega_{p_0}(S^2)$ be the based loop space of $S^2$ with basepoint $p_0$.  From the fibration exact sequence for homotopy groups associated to the fibration $\Omega_{p_0}(S^2) \xhookrightarrow{} \Lambda(S^2) \to S^2 $ we know that the inclusion $\Omega_{p_0}(S^2) \xhookrightarrow{} \Lambda(S^2)$ induces an isomorphism of fundamental groups. It follows that we can homotopy ${W}^{\kappa}_i$ among loops in $\Lambda(S^2)$ starting and ending at $\gamma_i^\kappa$, to a loop $\breve{W}^{\kappa}_i$ completely contained in $\Omega_{p_0}(S^2)$. It is clear that the sphere $\breve{S} := \mathcal{D}_{\gamma^\kappa_i} \# \breve{W}^{\kappa}_i \# (-\mathcal{D}_{\gamma^\kappa_i})$ is contractible since it is homotopic to $S$. 

We will now use the isomorphism between $\pi_2(S^2 )$ and $\pi_1(\Omega(S^2))$, where $\Omega({S^2})$ is the based loop space of ${S^2}$. This isomorphism is not canonical, as it depends on the choice of loop in $S^2$ which will be the base point of $\pi_1(\Omega({S^2}))$ together with a disk in $S^2$ capping this loop. We choose the base point of $\pi_1(\Omega(S^2))$ to be the loop $\gamma^\kappa_i$ and the capping disk to be $\mathcal{D}_{\gamma^\kappa_i}$. The isomorphism between $\pi_1(\Omega({S^2}),\gamma^\kappa_i )$ and $\pi_2(S^2,p_0 )$ then identifies $[\breve{W}^\kappa_i] \in \pi_1(\Omega(S^2),\gamma^\kappa_i)$ with $[\breve{S}] \in \pi_2(S^2,p_0 )$. 
Since $[\breve{S}] = 0 \in \pi_2(S^2,p_0 ) $ we obtain that  $[\breve{W}^{\kappa}_i] =0  \in \pi_1(\Omega(S^2))$, which implies that $\breve{W}^{\kappa}_i$ is homotopic to the trivial cylinder over $\gamma^\kappa_i$ among cylinders which are positively and negatively asymptotic to $\gamma^\kappa_i$. Since, as we saw in the previous paragraph, ${W}^{\kappa}_i$ is homotopic to $\breve{W}^{\kappa}_i$ among cylinders which are positively and negatively asymptotic to $\gamma^\kappa_i$, it is also homotopic to the trivial cylinder over $\gamma^\kappa_i$ among cylinders positively and negatively asymptotic to $\gamma^\kappa_i$. This gives us the promised homotopy.


The graph lift of the homotopy between ${W}^{\kappa}_i$ and ${W}^{\kappa,i}_{\rm triv}$ gives a homotopy between $\overline{U}^\kappa_i$ and the trivial trivial cylinder $\overline{U}^{\kappa,i}_{\rm triv}:[0,2M-2] \times S^1 \to [0,2M-2] \times S^1 \times S^2 $ defined by $$\overline{U}^{\kappa,i}_{\rm triv}(s,t) = (s,t, \gamma^\kappa_i(t)).$$

The cylinders $\overline{U}^{\kappa,i}_{\rm triv}$ and $\overline{U}^{\kappa',j}_{\rm triv}$ are disjoint if either $\kappa\neq \kappa'$ or $i\neq j$, which implies that their algebraic intersection number is $0$. The homotopy between 
$\overline{U}^{\kappa}_i$ and $\overline{U}^{\kappa,i}_{\rm triv}$ and the homotopy between $\overline{U}^{\kappa'}_j$ and $\overline{U}^{\kappa',j}_{\rm triv}$ then imply that the algebraic intersection number of $\overline{U}^{\kappa}_i$ and $\overline{U}^{\kappa'}_j$ is $0$, if either  $\kappa\neq \kappa'$ or $i\neq j$.

Reasoning as in Step 2 of the proof of Theorem \ref{thm:disk0} we conclude that any intersection point of $\overline{U}^{\kappa}_i$ and $\overline{U}^{\kappa'}_j$ counts positively, from where we obtain that $\overline{U}^{\kappa}_i$ and $\overline{U}^{\kappa'}_j$ are disjoint. This implies that $\widetilde{u}^1_{\kappa,i}$ and  $\widetilde{u}^1_{\kappa',j}$ have no intersection, if either  $\kappa\neq \kappa'$ or $i\neq j$. 

Reasoning as in step 2 of the proof of Theorem \ref{thm:disk0} (see Section~\ref{sec:proofthmdisk} ) one obtains an isotopy between the braid $\mathcal{B}(\mathcal{Y}_\oplus)$ associated to $\mathcal{Y}_\oplus$, and the braid  $\mathcal{B}(\mathcal{Y}_\ominus)$ associated to $\mathcal{Y}_\ominus= \cup_{\kappa \in \mathrm{Spec}(\widetilde{\mathcal{Y}_\oplus)}} \cup^{n_\kappa}_{i=1} \{\sigma^\kappa_{\mathfrak{f}(i)}\}$. \qed

\subsection{Proof of Theorem \ref{thm:main} in case $\Sigma$ has positive genus and the elements of $\mathcal{Y}_\oplus$ are contractible.}

\ 

\textit{Step 1:}

We define $\mathrm{Spec}(\mathcal{Y}_\oplus):= \cup_{1 \leq i \leq k} \mathcal{A}_{H_\oplus}(\gamma_i)$. 
For each number $\kappa \in \mathrm{Spec}(\mathcal{Y}_\oplus)$ we denote by $\mathcal{Y}^\kappa_\oplus \subset \mathcal{Y}_\oplus$ the subset of elements of $\mathcal{Y}_\oplus$ whose action is $\kappa$. We denote by $n_\kappa$ the cardinality of $\mathcal{Y}^\kappa_\oplus $. We denote by $\{\gamma^\kappa_1,...,\gamma^\kappa_{n_\kappa}\}$ the elements of $\mathcal{Y}^\kappa_\oplus $.

We find the maps $\mathfrak{f}_\kappa: \{1,...,n_\kappa\} \to \{1,...,m_\kappa\}$ and $\mathfrak{g}_\kappa:\{1,...,n_\kappa\} \to \{1,...,n_\kappa\}$ as in Step 1 of the proof of Theorem \ref{thm:disk0}. We then find the periodic orbits $\sigma^\kappa_{\mathfrak{f}(i)}$ of $H_{\ominus}$.

Following the same reasoning in Step 3 of Section \ref{sec:proofsphere} we construct the cylinders $$\overline{U}^\kappa_i: [0,2M-2] \times S^1 \to [0,2M-2] \times S^1 \times \Sigma, $$ and $$W^\kappa_i: [0,2M-2] \times S^1 \to \Sigma.$$

We will show that the cylinders $\overline{U}^\kappa_i $ and $\overline{U}^{\kappa'}_j$ are disjoint if either $\kappa \neq \kappa'$ or $i \neq j$. Once we establish this, the Theorem will follow by a reasoning identical to the one in Step 2 of the proof of Theorem \ref{thm:disk0} (presented in Section~\ref{sec:proofthmdisk}).

\ 

\textit{Step 2:}

We fix natural numbers $\kappa, \kappa',i,j$, and let either $\kappa \neq \kappa'$ or $i \neq j$.
As explained in Section \ref{sec:proofsphere} the intersection points of $\overline{U}^\kappa_i $ and $\overline{U}^{\kappa'}_j$ always count positively. Therefore $\overline{U}^\kappa_i $ and $\overline{U}^{\kappa'}_j$ are disjoint if, and only if, their algebraic intersection number is $0$.

We reason  by contradiction assuming that $\overline{U}^\kappa_i $ and $\overline{U}^{\kappa'}_j$ intersect. Let $\widetilde{\Sigma}$ be the universal covering of $\Sigma$: this induces an obvious covering $[0,2M-2] \times S^1 \times \widetilde{\Sigma}$ of $ [0,2M-2] \times S^1 \times {\Sigma}$. It also induces an obvious covering $ S^1 \times \widetilde{\Sigma}$ of $S^1 \times {\Sigma}$.

We can thus find lifts $\widetilde{U}^\kappa_i(0) : [0,2M-2] \times S^1 \to [0,2M-2] \times S^1 \times \widetilde{\Sigma} $ and $\widetilde{U}^{\kappa'}_j(0) :[0,2M-2] \times S^1 \to [0,2M-2] \times S^1 \times \widetilde{\Sigma}$ of these cylinders to the covering $[0,2M-2] \times S^1 \to [0,2M-2] \times S^1 \times \widetilde{\Sigma}$ of $[0,2M-2] \times S^1 \to [0,2M-2] \times S^1 \times {\Sigma}$, which intersect each other. The existence of these lifts uses the fact that  $\overline{U}^\kappa_i $ and $\overline{U}^{\kappa'}_j$ are asymptotic to contractible periodic orbits of $H_\oplus$: if the orbits were non-contractible the lifts would not exist.

Associated to the covering $\widetilde{\Sigma}$ of $\Sigma$ there is the group $\Gamma$ of deck transformations on $\widetilde{\Sigma}$. The group   $\Gamma$  is also the group of deck transformations of the covering $[0,2M-2] \times S^1 \times \widetilde{\Sigma}$ of $[0,2M-2] \times S^1 \times {\Sigma}$, and of the covering $S^1 \times \widetilde{\Sigma}$ of $ S^1 \times {\Sigma}$. 

We denote by $\xi^{\oplus}_{\kappa,i}$ the knot in $S^1 \times \Sigma$ given by $$\xi^{\oplus}_{\kappa,i}:= \{(t,\gamma^\kappa_i(t) \ | \ t \in S^1\}.$$ In an identical way we define $\xi^{\oplus}_{\kappa',j}$.

The cylinder $\widetilde{U}^\kappa_i(0)$ is negatively asymptotic to a lift $\xi^{\oplus}_{\kappa,i}(0)$ of $\xi^{\oplus}_{\kappa,i}$ in $S^1 \times \widetilde{\Sigma}$. It is positively asymptotic to a lift $\xi^{\oplus}_{\kappa,i}(1)$ of $\xi^{\oplus}_{\kappa,i}$ in $S^1 \times \widetilde{\Sigma}$. 

Let $\mathcal{T}^\kappa_i \in \Gamma$ be the deck transformation such that  $$\mathcal{T}^\kappa_i(\xi^{\oplus}_{\kappa,i}(0))= \xi^{\oplus}_{\kappa,i}(1).$$
We now define for each $l\in \Z$ 
\begin{equation}
    \xi^{\oplus}_{\kappa,i}(l) = (\mathcal{T}^\kappa_i)^l(\xi^{\oplus}_{\kappa,i}(0)),
\end{equation}
and 
\begin{equation}
\widetilde{U}^\kappa_i(l)= (\mathcal{T}^\kappa_i)^l(\widetilde{U}^\kappa_i(0)).
\end{equation}
It is clear that $\widetilde{U}^\kappa_i(l)$ is negatively asymptotic to $\xi^{\oplus}_{\kappa,i}(l-1)$ and positively asymptotic to $\xi^{\oplus}_{\kappa,i}(l)$.

Using $\widetilde{U}^{\kappa'}_j(0)$ we define in a similar way $\mathcal{T}^{\kappa'}_j$ and the sequence of lifts $\xi^{\oplus}_{\kappa',j'}(l)$ and $\widetilde{U}^{\kappa'}_j(l)$ for $l\in \Z$.

\

\textit{Step 3:}

Since $\Sigma$ is a surface of positive genus its universal cover $\widetilde{\Sigma}$ is diffeomorphic to the plane. We can thus consider the lifts of $\xi^{\oplus}_{\kappa,i}$ and of $\xi^{\oplus}_{\kappa',j}$ to $S^1 \times \widetilde{\Sigma}$ as knots on the plane, and define their crossing number of a pair of knots $\mathrm{Cross}$ as in Section 5 of \cite{BraidFloer}. Using positivity of intersections the authors showed in Lemma 5.4 of \cite{BraidFloer} that because $\widetilde{U}^\kappa_i(0)$ and $\widetilde{U}^{\kappa'}_j(0)$ have positive intersection number,  the crossing numbers of $\mathrm{Cross}(\xi^{\oplus}_{\kappa,i}(0),\xi^{\oplus}_{\kappa',j}(0))$ and $\mathrm{Cross}(\xi^{\oplus}_{\kappa,i}(1),\xi^{\oplus}_{\kappa',j}(1))$ satisfy 
\begin{equation} \label{eq:crossing1}
    \mathrm{Cross}(\xi^{\oplus}_{\kappa,i}(0),\xi^{\oplus}_{\kappa',j}(0)) > \mathrm{Cross}(\xi^{\oplus}_{\kappa,i}(1),\xi^{\oplus}_{\kappa',j}(1)).
\end{equation}

Moreover using the fundamental Lemma 5.4 of \cite{BraidFloer} we obtain the following fact:
\begin{fact}  \label{factCross}
If $\widetilde{U}^\kappa_i(0)$ and $\widetilde{U}^{\kappa'}_j(0)$ intersect, then $$\mathrm{Cross}(\xi^{\oplus}_{\kappa,i}(l-1),\xi^{\oplus}_{\kappa',j}(l-1)) > \mathrm{Cross}(\xi^{\oplus}_{\kappa,i}(l),\xi^{\oplus}_{\kappa',j}(l)).$$

If $\widetilde{U}^\kappa_i(0)$ and $\widetilde{U}^{\kappa'}_j(0)$ do not intersect, then 
$$\mathrm{Cross}(\xi^{\oplus}_{\kappa,i}(l-1),\xi^{\oplus}_{\kappa',j}(l-1)) = \mathrm{Cross}(\xi^{\oplus}_{\kappa,i}(l),\xi^{\oplus}_{\kappa',j}(l)).$$
\end{fact}

We conclude from this fact that the sequence $(\mathrm{Cross}(\xi^{\oplus}_{\kappa,i}(l),\xi^{\oplus}_{\kappa',j}(l)))_{l\in \Z}$ is a non-increasing sequence of integers.

Endow $\Sigma$ with an auxiliary Riemannian metric $g$ and let $\widetilde{g}$ be the induced metric on the universal cover $\widetilde{\Sigma}$. Take a Riemannian metric $g_0$ on $S^1$ and consider $g'$ to be the product metric of $g_0$ and $\widetilde{g}$ on $S^1 \times \widetilde{\Sigma}$. This Riemannian metric induces a distance function $d$ on $S^1 \times \widetilde{\Sigma}$ which we will fix from now on. Remark that $\Gamma$ acts on $S^1 \times \widetilde{\Sigma}$ by isometries.

Fix a lift $\widetilde{\xi}^{\oplus}_{\kappa,i}$. There are only finitely many lifts of $\widetilde{\xi}^{\oplus}_{\kappa',j}$ to $S^1 \times \widetilde{\Sigma}$ which have non-zero crossing number with $\widetilde{\xi}^{\oplus}_{\kappa,i}$. The reason for this is that for a lift of $\widetilde{\xi}^{\oplus}_{\kappa',j}$ to have non-zero crossing number with $\widetilde{\xi}^{\oplus}_{\kappa,i}$, its projection to the the plane $\widetilde{\Sigma}$ must intersect the projection of $\widetilde{\xi}^{\oplus}_{\kappa,i}$ to $\widetilde{\Sigma}$: clearly, there are only finitely many lifts of $\widetilde{\xi}^{\oplus}_{\kappa',j}$ with this property.
It follows that there exists a constant $\mathrm{D}>0$ such that any lift $\widetilde{\xi}$ of ${\xi}^{\oplus}_{\kappa',j}$ such that $d(\widetilde{\xi}^{\oplus}_{\kappa,i},\widetilde{\xi}) > D$ satisfies $$\mathrm{Cross}(\widetilde{\xi}^{\oplus}_{\kappa,i},\widetilde{\xi})=0.$$

Because $\Gamma$ acts by isometry we obtain that this constant $D$ does not depend on the lift of ${\xi}^{\oplus}_{\kappa,i}$. This gives us the following fact:
\begin{fact} \label{factD}
There exists a constant $D>0$ such that if $\widehat{\xi}_{\kappa,i}$ is a lift of ${\xi}^{\oplus}_{\kappa,i}$ and $\widehat{\xi}_{\kappa',j}$ is a lift of ${\xi}^{\oplus}_{\kappa',j}$ such that $d({\xi}^{\oplus}_{\kappa',j},\widehat{\xi}_{\kappa',j})>D$ then $$ \mathrm{Cross}(\widehat{\xi}_{\kappa,i},\widehat{\xi}_{\kappa',j})=0.$$ 
\end{fact}

The next fact follows easily from the observation above that there are only finitely many lifts of $\widetilde{\xi}^{\oplus}_{\kappa',j}$ to $S^1 \times \widetilde{\Sigma}$ which have non-zero crossing number with $\widetilde{\xi}^{\oplus}_{\kappa,i}$.
\begin{fact} \label{factbound}
Given any lift $\widetilde{\xi}^\oplus_{\kappa,i}$ of ${\xi}^{\oplus}_{\kappa,i}$ we have $$\max_{\widetilde{\xi} \mbox{ is a lift of } {\xi}^{\oplus}_{\kappa',j} }\{ |\mathrm{Cross}(\widetilde{\xi}^\oplus_{\kappa,i},\widetilde{\xi})|  \}< +\infty. $$ 
\end{fact}

\ 

\textit{Step 4:}

We now have to deal with two cases:
\begin{itemize}
    \item[(A)] The deck transformations $\mathcal{T}^{\kappa}_i$ and $\mathcal{T}^{\kappa'}_j$ do not coincide.
    \item[(B)] The deck transfomations $\mathcal{T}^{\kappa}_i$ and $\mathcal{T}^{\kappa'}_j$ coincide.
\end{itemize}
We start treating the case (A).

\textit{Case (A):}

Since $\mathcal{T}^{\kappa}_i$ and $\mathcal{T}^{\kappa'}_j$ do not coincide, we can find a positive integer $N>0$ such that $$d((\mathcal{T}^{\kappa}_i)^l(\xi^{\oplus}_{\kappa,i}(0)),(\mathcal{T}^{\kappa'}_j)^l(\xi^{\oplus}_{\kappa',j}(0))) > D \mbox{ for all } l\geq N,$$
and 
$$d((\mathcal{T}^{\kappa}_i)^l(\xi^{\oplus}_{\kappa,i}(0)),(\mathcal{T}^{\kappa'}_j)^l(\xi^{\oplus}_{\kappa',j}(0))) > D \mbox{ for all } l\leq -N,$$
where $D>0$ is the constant given by Fact \ref{factD}.

In particular, we conclude that 
$$ \mathrm{Cross}(\xi^{\oplus}_{\kappa,i}(N),\xi^{\oplus}_{\kappa',j}(N))= \mathrm{Cross}(\xi^{\oplus}_{\kappa,i}(-N),\xi^{\oplus}_{\kappa',j}(-N))=0.$$ 

But, as observed above, Fact \ref{factCross} implies that the sequence \linebreak $(\mathrm{Cross}(\xi^{\oplus}_{\kappa,i}(l),\xi^{\oplus}_{\kappa',j}(l)))_{l\in \Z}$ is a non-increasing sequence of integers. We conclude that for every integer $-N \leq p \leq N$ we have $$ \mathrm{Cross}(\xi^{\oplus}_{\kappa,i}(p),\xi^{\oplus}_{\kappa',j}(p))=0. $$
But this is in contradiction with \eqref{eq:crossing1} which we derived from the fact that  $\widetilde{U}^\kappa_i(0)$ and $\widetilde{U}^{\kappa'}_j(0)$ have positive intersection number. This concludes the proof in case (A).

\textit{Case (B):}

In this case we let $\mathcal{T}:= \mathcal{T}^{\kappa}_i = \mathcal{T}^{\kappa'}_j$. Since $\widetilde{U}^\kappa_i(0)$ and $\widetilde{U}^{\kappa'}_j(0)$ have positive intersection the same is true for $\widetilde{U}^\kappa_i(l) = \mathcal{T}^l(\widetilde{U}^\kappa_i(0))$ and $\mathcal{T}^l(\widetilde{U}^{\kappa'}_j(l))$ for all $l \in \Z$.
Lemma 5.4 of \cite{BraidFloer} then implies that for all $l \in \Z$
\begin{equation} \label{eq:crossing}
    \mathrm{Cross}(\xi^{\oplus}_{\kappa',j}(l-1),\xi^{\oplus}_{\kappa',j}(l-1)) > \mathrm{Cross}(\xi^{\oplus}_{\kappa',j}(l),\xi^{\oplus}_{\kappa',j}(l)).
\end{equation}

It is then clear that the sequence $(\mathrm{Cross}(\xi^{\oplus}_{\kappa,i}(l),\xi^{\oplus}_{\kappa',j}(l)))_{l\in \Z}$ is unbounded. But this is in contradiction with Fact \ref{factbound}. This establishes the theorem in case (B), and concludes the proof of Theorem \ref{thm:main} under the hypothesis the $\Sigma$ has positive genus and $\mathcal{Y}_\oplus$ is formed by contractible $1$-periodic orbits of $H_\oplus$.  \qed

\subsection{Proof of Theorem \ref{thm:main} in case $\Sigma$ has genus $\geq 2$ and the elements of $\mathcal{Y}_\oplus$ are non-contractible.}

\

\textit{Outline of the proof:}

The proof in this case is identical to the proof of Theorem \ref{thm:disk0}. The reason for this is that if $\gamma$ is a non-contractible loop in $\Sigma$ then any cylinder in $\Sigma$ which is both positively and negatively asymptotic to $\gamma$ is homotopic to the trivial cylinder over $\gamma$: this is the case because $\Sigma$ is atoroidal since it has genus $\geq 2$. 

We use the same notation as in the proof of Theorem \ref{thm:disk0}.
We then construct for each $\kappa \in \mathrm{Spec}(\mathcal{Y}_\oplus)$ and $i\in \{1,...,n_\kappa\}$ the cylinder $\overline{U}^\kappa_i$ as in the proof of Theorem \ref{thm:disk0} and obtain from the previous observation that it is homotopic to the trivial cylinder $\overline{U}^{\kappa,i}_{\rm triv}:[0,2M-2] \times S^1 \to [0,2M-2] \times S^1 \times S^2 $ defined by $$\overline{U}^{\kappa,i}_{\rm triv}(s,t) = (s,t, \gamma^\kappa_i(t)).$$ 
Once this is obtained the proof of the present case of Theorem \ref{thm:main} follows the arguments presented in Step 2 of the the proof of Theorem \ref{thm:disk0} (presented in Section \ref{sec:proofthmdisk}). \qed

\subsection{Proof of Theorem \ref{thm:main} in case $\Sigma = T^2$ and the elements of $\mathcal{Y}_\oplus$ are non-contractible.}

\

\textit{Outline of the proof:} 

The proof in this case is similar to the one presented in Section \ref{sec:proofsphere} for the case $\Sigma= S^2$. Namely, we cannot guarantee by purely topological reasons that the cylinder $\overline{U}^\kappa_i$ is homotopic to the trivial cylinder $\overline{U}^{\kappa,i}_{\rm triv}$ among cylinders which are positively and negatively asymptotic to $\gamma^\kappa_i$, since $T^2$ is toroidal.

However, using an argument identical to the one presented in Step 3 of Section \ref{sec:proofsphere}, we can show that if $[\mathrm{Cyl}_{\gamma^\kappa_i}]$ is a choice of capping for $\gamma^\kappa_i$, then the cylinder ${W}^\kappa_i$ (which is the projection of $\overline{U}^\kappa_i$ to $T^2$) is a cylinder from $(\gamma_i^\kappa,[\mathrm{Cyl}_{\gamma^\kappa_i}])$ to itself. It follows from this that ${W}^\kappa_i$ represents the trivial element in $\pi_1(\Omega_{[\gamma^\kappa_i]}(T^2,\gamma^\kappa_i))$, where $\Omega_{[\gamma^\kappa_i]}(T^2,\gamma^\kappa_i)$ is the connected component of the loop space of $T^2$ that contains $\gamma^\kappa_i$. We conclude that ${W}^\kappa_i$ is homotopic to the trivial cylinder ${W}^{\kappa,i}_{\rm triv}$ over $\gamma^\kappa_i$. Lifting this homotopy we obtain the desired homotopy between $\overline{U}^\kappa_i$  and the trivial cylinder $\overline{U}^{\kappa,i}_{\rm triv}$ among cylinders which are positively and negatively asymptotic to $\gamma^\kappa_i$. Once this is obtained the proof is completed by using the argument of Step 2 of the proof of Theorem \ref{sec:proofthmdisk} (presented in Section \ref{sec:proofthmdisk}). \qed

\section{A variant of Theorems \ref{thm:main} and \ref{thm:disk}\label{sec:replacing_assumption}}

We explain in this section why the assumption of the orbits to be $100\epsilon$-separated in the main theorems can be replaced by a slightly weaker but more technical assumption of being $100\epsilon$-quasi-isolated. 

\begin{mainthma}{\ref{thm:main}*}\label{thm:maina}
The statement in Theorem \ref{thm:main} holds true if the assumption that $\mathcal{Y}_{\oplus}$ is $100\epsilon$-isolated is replaced by the assumption that $\mathcal{Y}_{\oplus}$ is $100\epsilon$-quasi-isolated for some pair $(H_{\oplus}, J_{\oplus})$, where $J_\oplus$ is a $S^1$ family of compatible almost complex structures. 
\end{mainthma}

\begin{mainthma}{\ref{thm:disk}*}\label{thm:diska}
The statement in Theorem \ref{thm:disk} holds true if the assumption that $\mathcal{Y}_{\oplus}$ is $100\epsilon$-isolated is replaced by the assumption that $\mathcal{Y}_{\oplus}$ is $100\epsilon$-quasi-isolated for some pair $(H_{\oplus}, J_{\oplus})$, where $J_\oplus$ is a $S^1$ family of compatible almost complex structures. 
\end{mainthma}

The proofs of Theorems \ref{thm:main} and $\ref{thm:disk}$ also apply here with small modifications. 
\begin{proof}[Proof of Theorem \ref{thm:diska}]
Let $\phi_{\oplus}$, $H_{\oplus}$, $\mathcal{Y}_{\oplus}$ and $J_{\oplus}$ as in the theorem for some $\epsilon>0$, and let $\phi_{\ominus}$ be non-degenerate Hamiltonian diffeomorphism with $d_{\hofer}(\phi_{\ominus}, \phi_{\oplus}) < \epsilon$.  
As in the proof of Theorem \ref{thm:disk} one writes  $\mathcal{Y}_{\oplus}$ as a disjoint union of sets $\mathcal{Y}_{\oplus}^{\kappa}$ of orbits $\gamma\in \mathcal{Y}_{\oplus}$ with  $\mathcal{A}_{H_{\oplus}}(\gamma) = \kappa$. 
One then considers for each $\kappa$ with $\mathcal{Y}_{\oplus}^{\kappa} \neq \emptyset$ the $\Z_2$ vector space  $\mathcal{B}_{\mathcal{Y}^{\kappa}_{\oplus}}$ generated by $\mathcal{Y}_{\oplus}^{\kappa}$ as subcomplexes of $CF^{(\kappa-100\epsilon, \kappa + 100\epsilon)}(H_{\oplus})$, as done in Section \ref{sec:quasi-iso}. 
Proposition \ref{prop:quasi-isoD} gives then $H_{\ominus}$ generating $\phi_{\ominus}$ and  homomorphisms $\Psi^{\mathcal{Y}}_G: \mathcal{B}_{\mathcal{Y}^{\kappa}_{\oplus}} \to CF^{(\kappa - 99\epsilon, \kappa + 99 \epsilon)}(H_{\ominus})$ and 
$\Psi^{\mathcal{Y}}_{\widehat{G}}:CF^{(\kappa - 99\epsilon, \kappa + 99 \epsilon)}(H_{\ominus}) \to \mathcal{B}_{\mathcal{Y}^{\kappa}_{\oplus}}$, with $\Psi^{\mathcal{Y}}_{\widehat{G}}\circ \Psi^{\mathcal{Y}}_G = \id$.
One then can proceed in the same way as in the proof of Theorem \ref{thm:disk}, using the maps
$\Psi^{\mathcal{Y}}_G$ and $\Psi^{\mathcal{Y}}_{\widehat{G}}$ instead of the maps $\Psi_G$ and $\Psi_{\widehat{G}}$. 
Note that the condition that $|\mathcal{A}_{H_{\oplus}}(\gamma) - \mathcal{A}_{H_{\oplus}}(\gamma')|\notin(0,100\epsilon)$ for $\gamma, \gamma'\in \mathcal{Y}_{\oplus}$  guarantees that the first part of Claim 2 in the proof of Theorem \ref{thm:disk} will again hold. 

\end{proof}

\begin{proof}[Proof of Theorem \ref{thm:maina}]
Here again, the proof of Theorem \ref{thm:main} applies with the same  small modification. Instead of the homomorphisms $\Psi_{G}$ and $\Psi_{\widehat{G}}$ that are guaranteed, for an $100\epsilon$-isolated set of orbits with the same action, by Proposition \ref{prop:actionwindowsclosedsurface} ($\Sigma$ is a closed surface) resp.  
Proposition \ref{prop:chainisosphere} ($\Sigma = S^2$), one uses the homomorphims $\Psi^{\mathcal{Y}}_{G}$ and $\Psi^{\mathcal{Y}}_{\widehat{G}}$ that are given in Proposition \ref{prop:quasi_iso} resp. Proposition \ref{prop:quasi-isoS}, for an $100\epsilon$-quasi-isolated set of orbits of the same action.

\end{proof}


\section{Proof of Theorem \ref{thm:nondeg_stable}}\label{sec:nondeg_stable}

In this section we deduce Theorem \ref{thm:nondeg_stable} from Theorems \ref{thm:maina} and \ref{thm:diska}. 

We start the with the following proposition. Consider first $\Sigma$ to be a closed surface. And let  $H:\Sigma \times S^1  \to \R$ be a normalized Hamiltonian, and $\phi = \phi^1_H$ the diffeomorphism generated by $H$. 
Let $\mathcal{Y}= \{\gamma_1, \ldots, \gamma_k\}$ be a collection of pairwise freely homotopic, non-degenerate $1$-periodic orbits for $H$. 

\begin{prop}\label{prop:first_perturb}
There are open disks $U_i\subset \Sigma$, $i=1, \ldots,n$ with $\gamma_i(0)\in U_i$, $i=1, \ldots, k$,  $\varepsilon'>0$, a non-degenerate Hamiltonian $H'$ that generates a diffeomorphism $\phi'= \phi^1_{H'}$, and a $S^1$-family of compatible almost complex structures $J'=J'_t$ with the property that  
\begin{itemize}
\item $d_{\hofer}(\phi', \phi) < {\varepsilon'}$,
\item $\phi'|_{U_i} = \phi|_{U_i}$, $i=1, \ldots k$, 
\item $\mathcal{Y}$ is $200\varepsilon'$- quasi-isolated for $(H', J')$ 
\end{itemize}
\end{prop}

\begin{proof}
Note that the orbits in $\mathcal{Y}$ are topologically isolated among all the $1$-periodic orbits for $H$. 
It is a variant of the standard argument for the genericity of non-degenerate Hamiltonian diffeomorphisms that for sufficiently small neigbourhoods $V_i$ of $\{(t,\gamma_i(t))\, |\,  t \in S^1 \}$ in $S^1 \times \Sigma$, there is a sequence of non-degenerate Hamiltonians $H_j:S^1 \times \Sigma \to \R$, $j\in \N$, with  $H_j|_{V_i} = H|_{V_i}$, for all $i=1, \ldots, k$, $j\in \N$, and such that $H_j$ converge to $H$ in $C^{\infty}$. We may choose the $V_i$ such that $U_i := {\rm pr}_2(V_i \cap (\{0\} \times \Sigma ))\subset \Sigma$ are discs, where ${\rm pr}_2$ is the projection to the second coordinate, and such that $\gamma_i(0)$ is the only point $z\in U_i$ with $\phi(z)= z$.
We can choose a sequence $J_j$, $j\in \N$,  of $S^1$ families of compatible almost complex structures on $\Sigma$ such that the pairs $(H_j,J_j)$ are regular, and such that $J_j$ converges with all derivatives to a $S^1$-family of compatible almost complex structures $J$.  

We claim that there is $\varepsilon'>0$ and  $j\in \N$, such that $H':=H_j$ and $J'=J_j$  satisfy the asserted properties from the proposition.
The condition $a)$ in the definition of $\varepsilon'$-quasi-isolation is clearly satisfied for any sufficiently small $\varepsilon'$, hence if the above does not hold, then one can pass to a subsequence of $H_j$ for which there are non-constant $u_{j}: \R \times S^1\to M$ that satisfy $\mathcal{F}_{H_{j},J_{j}} (u_j)= 0$ such that $u_j$ are positive or negative asymptotic to some $\gamma_i$, and for which $E(u_j) \to 0$. Here the energies of $u_j$ are defined with respect to $J_j$.  W.l.o.g and by passing further to a subsequence we may assume that $u_{j}$ are all negatively  asymptotic to $\gamma_{i_0}$ for some fixed $i_0\in  \{1, \ldots, k\}$. For any non-constant Floer trajectory, the periodic orbits that appear as its negative and positive asymptotics do not coincide, so for any $j\in \N$, ${u_j}$ is positively asympotic to some $1$-periodic orbit for $H_j$ which is different from $\gamma_{i_0}$.  
It follows that there is $s_j$ such that $u_j(s_j,0)$ lies in the boundary $\partial U_{i_0}$ of $U_{i_0}$. We may assume that $s_j = 0$ for all $j \in \N$,  and, by passing to a further subsequence, that $u_j(0,0)$ converges to some $x_0 \in \partial U_{i_0}$. 
The first derivatives of $u_j$ are uniformely bounded. Otherwise,  by a bubbling-off argument, since $E(u_j) \to 0$, one will find a non-constant holomorphic sphere with zero energy, a contradiction. With this, one proves, as it is standard in Floer theory,  that up to a further subsequence, $(u_j)_{j\in \N}$ uniformely converges with all their derivatives on compact subsets. Any limit curve will satisfy $\mathcal{F}_{H,J}(u) = 0, E(u) = 0$. In particular for such $u$ we have that $\lim_{s\to \infty}u(s,t) = \lim_{s\to -\infty}u(s,t)$ is a periodic orbit for $H$.  But by the above we in particular find such $u$ with $u(0,0) = x_0\in \partial U_{i_0}$, which contradicts the fact that $\phi(z) \neq z$ for all $z \in \partial U_{i_0}$.
\end{proof}

We have the following variant of the above proposition. For that consider  $\Sigma = \D$. Let $H \in \mathcal{H}_0(\D)$ and $\phi \in \Ham_0(\D)$ generated by $H$. 
Let $\mathcal{Y}= \{\gamma_1, \ldots, \gamma_k\}$ be a collection of non-degenerate $1$-periodic orbits for $H$. For $0< \rho< 1$ we denote by $\D_{1-\rho}$ the set $\{ (r,\theta) \in \D \, |\, 1-\rho> r \geq 0\},$ where $(r,\theta)$ denote the polar coordinates. 
\begin{prop}\label{prop:first_perturbD}
There are open disks $U_i\subset \D$, $i=1, \ldots,n$ with $\gamma_i(0)\in U_i$, $i=1, \ldots, k$,  and $\varepsilon'>0$, $c>0$, $\rho_0>0$ such that for each $0 < \rho< \rho_0$ there is a non-degenerate Hamiltonian $H' \in \mathcal{H}_c(\D)$  that generates a diffeomorphism $\phi'= \phi^1_{H'}\in \Ham_c(\D)$, and there is a $S^1$-family of compatible almost complex structures $J'=J'_t$ with the property that  
\begin{itemize}
\item $d_{C^2}(H|_{\D_{1-\rho}}, H'|_{\D_{1-\rho}}) < {\varepsilon'}$,
\item $\phi'|_{U_i} = \phi|_{U_i}$, $i=1, \ldots k$, 
\item there are no periodic orbits of $\phi'$ in $\D \setminus \D_{1-\rho}$,
\item $\mathcal{Y}$ is $200\varepsilon'$- quasi-isolated for $(H', J')$ 
\end{itemize}
\end{prop}
\begin{proof}
Since the orbits in $\mathcal{Y}$ are non-degenerate, they are contained in the interior of the support of $H$. Hence we can find for sufficiently small $c>0$,  sets $U_i$, $i=1,\ldots, k$ and a sequence $H_j\in \mathcal{H}_c(\D)$ with  $$d_{C^2}(H_j|_{\D_{1-\frac{1}{j}}}, H|_{\D_{1-\frac{1}{j}}}) \to 0$$ as $j\to \infty$, and that coincide with $H$ on a sufficiently small neighbourhood of the suspension of $\gamma_1, \ldots, \gamma_k$,  and we can find a sequence of compatible almost complex structures $J_j$ such that the pairs $(H_j,J_j)$ are regular, and such that $J_j$ converges to a compatible almost complex structure $J$. We can additionally assume that there are no $1$-periodic orbits for $H_j$ in the complement of $\D_{1-\frac{1}{j}}$.  

Now, similarly as above, we claim that there is $\varepsilon'>0$ and $j_0 \in \N$, such that for all $j>j_0$, 
$\mathcal{Y}$ is $200\varepsilon'$- quasi-isolated for $(H_j,J_j)$. 
If this is not the case then there is subsequence of  $H_j$ and a sequence of $u_{j}: \R \times S^1\to \R^2$ such that  $\mathcal{F}_{H_{j},J_{j}} (u_j)= 0$, $E(u_j) \to 0$, and w.l.o.g $u_{j}$ are all negative asymptotic to $\gamma_{i_0}$ for some fixed $i_0\in  \{1, \ldots, k\}$. As before, this gives rise to a contradiction. 
\end{proof}

\begin{proof}[Proof of Theorem \ref{thm:nondeg_stable}]

We start with the case that $\Sigma$ is a closed surface. 
Let $H_{\oplus}, \phi_{\oplus}, \mathcal{Y}_{\oplus} = \{\gamma_1, \ldots, \gamma_k\}$ given as in the theorem. 

We apply Proposition \ref{prop:first_perturb} with $H=H_{\oplus}$, $\mathcal{Y} = \mathcal{Y}_{\oplus}$, and obtain $\varepsilon'$ as well as a Hamiltonian  $H'$ and a $S^1$ family of compatible almost complex structures $J'$, such that the conclusions of the proposition hold for $\varepsilon', J', H'$ and $\mathcal{Y}$. 

Apply now Theorem \ref{thm:main}* with now $
H_{\oplus} = H'$, $J_{\oplus} = J'$, $\mathcal{Y}_{\oplus} = \mathcal{Y}$  $\epsilon = 2\varepsilon'$. Since any ball $B\in \Ham(\Sigma, \omega)$ of radius $2\varepsilon'$ with a center that lies in the ball $B'$ of radius $\varepsilon'$ contains $B'$, the  conclusion of Theorem   \ref{thm:main}* imply now Theorem \ref{thm:nondeg_stable} in the case of  closed surfaces.

 Let now $\Sigma = \D$. Let $H_{\oplus} \in \mathcal{H}_0(\D)$ non-degenerate and $\phi_{\oplus}$ generated by $H_{\oplus}$.  Let $\mathcal{Y}_{\oplus}$ be a collection of non-degenerate $1$-period orbits for $H_{\oplus}$. 
 We apply Proposition \ref{prop:first_perturbD} with $H= H_{\oplus}$, $\mathcal{Y}= \mathcal{Y}_{\oplus}$ and obtain $\varepsilon'>0$, $c>0$ $\rho_0>0$, such that for each $0<\rho<\rho_0$ there is a pair $(H', J')$ such that the conclusions of the proposition hold. 
 
 Let $\epsilon= 2\varepsilon'$. Let $\phi \in \Ham_0(\D)$ be non-degenerate in its support with 
 $d_{\hofer}(\phi, \phi_{\oplus}) < \epsilon$. 
 Then there is $\rho_1$ such that $\supp \phi\in \D_{1-\rho_1}$. 
 Choose $0< \rho_2< \min\{\rho_0, \rho_1\}$. 
 By a sufficiently $C^2$ small approximation of $H'$ we can  now find a non-degenerate Hamiltonian $H_{\ominus} \in \mathcal{H}_{c}(\D^2)$ that generates a diffeomorphism $\phi_{\ominus}\in \Ham_c(\D^2)$ such that 
 $F(t,p) := H_\ominus(t,(\phi_\oplus^t)(p)) - H_{\oplus}(t, (\phi_\oplus^t)(p))$ that generates $\phi_{\oplus}^{-1} \circ \phi_\ominus$ satisfies 
 $\int_{0}^1 (\max F_t - \min F_t) \, dt < \epsilon$, and that all the $1$-periodic orbits for $H_{\ominus}$ are also $1$-periodic orbits of $H'$.   In particular we have $d_{\hofer}(\phi_{\ominus}, \phi_{\oplus})< \epsilon$.
 
 Now apply Theorem \ref{thm:diska} with $\epsilon$ from above, and now $H_{\oplus} = H'$, $J_{\oplus}= J'$, $\mathcal{Y}_{\oplus} = \mathcal{Y}$, and note that one can choose in the conclusions of the theorem  $H_{\ominus}$ as above. Hence we obtain a collection of non-degenerate $1$-periodic orbits $\mathcal{Y}_{\ominus}$ such that  $$\mathcal{B}(\mathcal{Y}_\ominus) \mbox{ is freely isotopic as a braid to } \mathcal{B}(\mathcal{Y}_\oplus).$$  Since all $1$-periodic orbits for $H_{\ominus}$ are already $1$-periodic orbits for $H'$, the conclusion follows. 
 

\end{proof}

\appendix


\section{Combinatorial lemma} \label{appendix:comblemma}
In this appendix we prove Lemma \ref{lem:combinatorial}. For the convenience of the reader we restate it here. 
\begin{lem*} 
Let $V$ and $R$ be finite dimensional $\Z_2$-vector spaces whose dimensions we denote by $n$ and $m$, respectively. Let $\{v_1,...,v_{n}\}$ and $\{r_1,...,r_{m}\}$ be basis of respectively $V$ and $W$, and $\mathfrak{F}: V \to W$ and $\mathfrak{G}: W \to V $ be linear maps such that $\mathfrak{G} \circ \mathfrak{F}$ is an isomorphism. Then, it is possible to find an injective map $\mathfrak{f}: \{1,...,n\} \to \{1,...,m\}$ and a bijective map $\mathfrak{g}: \{1,...,n\} \to \{1,...,n\}$ such that for each $i\in \{1,...,n\}$:
\begin{itemize}
    \item  the element $r_{\mathfrak{f}(i)}$ appears in $\mathfrak{F}(v_i) $, and the element $v_{\mathfrak{g}(i)}$ appears in $\mathfrak{G}(r_{\mathfrak{f}(i)}) $.
\end{itemize}
\end{lem*}

We start introducing some terminology. Given a $\Z_2$vector space $B$ of dimension $n$ and a basis $\{b_1,...,b_n\}$ of $B$, we say  we say that an element $b_i$ of the base $\{b_1,...,b_{n}\}$ appears in an element $b \in B$, if this element appears when we write $b$ in an unique way as a sum of elements of $\{b_1,...,b_{n}\}$. For each $i \in \{1, \ldots, n\}$, we let $H^B_i = \langle b_1, \ldots, b_{i-1}, b_{i+1}, \ldots, b_{n} \rangle$ the subspace generated by all but the $i$-th basis vector and $pr_i : B \to  H^B_i$ be the projection onto $H^B_i$ which satisfies $pr_i(b_i)=0$. The element $b_i$ appears in $b$, if and only if, $b \notin H^B_i$. 

To prove Lemma \ref{lem:combinatorial} we need the following preliminary lemma.

\begin{lem} \label{lem:prelim}
Let $n\geq 1$, $k\geq 0$. Let $A$ be a $n+k$-dimensional vector space with basis $e_1, \ldots e_{n+k}$. 

Let $w_1, \ldots w_n \subset V$ be linearly independent vectors of $A$ and denote $W =\langle w_1, \ldots, w_n\rangle$. Let $Z \subset V$ be a $k$-dimensional linear subspace that is transverse to $W$, i.e.  $W\cap Z = \{0\}$ and $Z \oplus W = A$.
 
Then there is an injective function $\iota:\{1, \ldots, n\} \to \{1, \ldots, n+k\}$  such that 
for $L= \langle w_{\iota(1)}, \ldots, w_{\iota(n)}\rangle$
\begin{enumerate}
\item $e_{\iota(j)}$ appears in $w_j$, 
\item $Z$ is transverse to the $n$-dimensional vector space $L$. 
\end{enumerate}

\end{lem}

\begin{proof}

For each $i \in \{1, \ldots, k+n\}$, let $H_i = \langle e_1, \ldots, e_{i-1}, e_{i+1}, \ldots, e_{n+k} \rangle$ be the subspace generated by all but the $i$-th basis vector and $pr_i : A \to  H_i$ the canonical projection. 

We give a proof by induction on $n$. 
For that start with the case $n=1$, $k\geq 0$. 
Take a maximal set of basis vectors $E= \{e_{j_1}, \cdots, e_{j_l}\}$ such that $w=w_1\in \cap_{i=1}^l H_{j_l}=:\widehat{H}$. 
Since $w \notin Z$, we have $\widehat{H} \notin Z$, hence there is a basis vector $e_{i_1} \notin E$ with $e_{i_1} \notin Z$. 
By the construction of $E$, $w \notin H_{i_1}$, hence we can define $\iota$ by  $\iota(w) = i_1$. 

Assume now the statement holds for some $n-1\geq 1$, $k\geq 0$. 
Let $w_1, \ldots w_{n}$ and $Z$ satisfy the assumptions of the lemma. 
Consider the $n+k-1$ dimensional subspace  $Q = \langle w_1, \ldots,w_{n-1} \rangle  \oplus Z\subset V$. 
Then $\langle w_n\rangle \cap Q = \{0\}$, and as above there is a basis vector $e_{i_n}\notin Q$ with 
$w_n \notin H_{i_n}$. 
Consider $Z'= pr_{i_n}(Z)$ and $w'_i = pr_{i_n}(w_i)$ for $i=1, \ldots, n-1$. Let $W':= \langle w'_1, \ldots, w'_{n-1}\rangle$. 
Since $e_{i_n}\notin Q$ we have that $Z'$ is $k$-dimensional,  $W'$ is $n-1$ dimensional, 
and $Z' \cap  W' = \{0\}$.  By the induction hypothesis there is an injective function 
$\iota':\{1, \ldots {n-1}\} \to \{1, \ldots, n+k\}\setminus\{i_n\}$  such that 
for $L':= \langle w'_{\iota'(1)}, \ldots, w_{\iota'(n)}\rangle$, 
\begin{itemize}
\item $w'_j \notin \langle \{e_1, \ldots e_{k+n}\} \setminus \{e_{i_n}, e_{\iota'(j)}\} \rangle$, $j=1, \ldots n-1$,\footnote{This is equivalent to the fact that $e_{\iota'(j)}$ appears in $w'_j$, for the space $H_{i_n}$ with the basis $ \{e_1, \ldots e_{k+n}\} \setminus \{e_{i_n} \}$.} 
\item $L'\cap Z' = \{0\}$. 
\end{itemize}
Note, that by definition also $w'_j \notin \langle \{e_1, \ldots e_{k+n}\} \setminus \{ e_{\iota'(j)}\} \rangle = H_{\iota'(j)}$. 
Define now $\iota:\{1, \ldots, n\} \to \{1, \ldots, n+k\}$ by 
\[
\iota(i) =\begin{cases}  \iota'(i), &\text{ if } i = 1, \ldots n-1 \\ i_n, &\text{ if } i=n  \end{cases}
\]
Since for $i=1, \cdots, n-1$, , $w'_i = pr_{i_n}(w_i)$, we also have $w_i \notin H_{\iota(i)}$. 
Also $L\cap Z = \{0\}$, with $L = L'\oplus \langle e_{i_n}\rangle$. 
Indeed, let  $v\in L \cap Z$, then  $pr_{i_n}(v)=0 $ since it belongs to the intersection $ L' \cap Z' = \{0\}$. It follows that $v \in \langle e_{i_n}\rangle$, and since $e_{i_n} \notin Z$, one has $v = 0$. 

\end{proof}
We now apply Lemma \ref{lem:prelim} to prove Lemma \ref{lem:combinatorial}. 

\ 

\textit{Proof of Lemma \ref{lem:combinatorial}:} \\
By the hypothesis of the lemma, it is clear that $\mathfrak{F}$ is injective and thus $m \geq n$.
Apply first Lemma \ref{lem:prelim} with $ A = R$, $w_1 = \mathfrak{F}(v_1), \ldots, w_n = f(v_n)$ and $Z = \ker(\mathfrak{G})$. 
This gives an injective map $\mathfrak{f}: \{1, \cdots, n \} \to \{1, \cdots, m\}$  with $r_{\mathfrak{f}(i)}$ appearing in $\mathfrak{F}(v_i)$ and  such that in particular $r_{\mathfrak{f}(1)}, \cdots, r_{\mathfrak{f}(n)}$ generate a subspace $C$ that is transverse to the kernel of $\mathfrak{G}$, and hence 
$\mathfrak{G}|_C: C \to V$ is an isomorphism.

Apply Lemma \ref{lem:prelim} again, now with $A=V$ and $w_1 = \mathfrak{G}(r_{\mathfrak{f}(1)}), \ldots, w_n = \mathfrak{G}(r_{\mathfrak{f}(n)})$ and $Z = \{0\}$. 
This gives $\mathfrak{g}:\{1, \cdots, n\} \to \{1\cdots, n\}$ bijective, such that $v_{\mathfrak{g}(i)}$ appears in $\mathfrak{G}(r_{\mathfrak{f}(i)})$. 
\qed

\section{Approximation of entropy by the entropy of braid types} \label{appendix:approx}

Recall that if $\varphi: \Sigma \to \Sigma$ is a diffeomorphism on a compact surface $\Sigma$ and $\mathcal{P}$ a periodic orbit and $\overline{P}\subset \Sigma$ the associated set of periodic points. Then  $\Gamma_{\pi_1}([\varphi, \mathcal{P}])$ denotes the growth rate of the induced action of $\varphi$ on the fundamental group of $\Sigma \setminus \overline{\mathcal{P}}$, see section \ref{sec:braids}. 
Also recall that $\Gamma_{\pi_1}([\varphi, \mathcal{P}]) \leq h_{\topo}(\varphi)$. 

The aim of this section is to give a proof of the following result.  

\begin{thm}\label{thm:approximation}
Let $\Sigma$ be a compact surface, and let $\varphi: \Sigma\to \Sigma$ be a diffeomorphism such that $h:=h_{\topo}(\varphi)>0$. 
Then, for any $\epsilon>0$ there is a hyperbolic periodic orbit $\mathcal{P}$ of $\varphi$ such that
$\Gamma_{\pi_1}([\varphi,\mathcal{P}])> h - \epsilon$. 
\end{thm}

This is variant of a celebrated result of Katok first announced in \cite{Katok1984}, see also  \cite[Supplement]{Hasselblatt-Katok},  about the approximation of the topological entropy by the entropies on locally maximal hyperbolic invariant sets on which the dynamics is conjugated to a subshift of finite type, and in fact the orbits $\mathcal{P}$ of Theorem $\ref{thm:approximation}$ can be found inside those sets.
To our knowledge there is no proof of Theorem \ref{thm:approximation} in the literature, so we include it here. In \cite{FranksHandel1988} it was proved that there is a collection of orbits for which the exponential growth of $\varphi$ relative to it is positive. We apply the strategy in \cite{FranksHandel1988} to the hyperbolic horseshoes that arise from Katok and Mendoza's results in \cite[Supplement]{Hasselblatt-Katok}.  

We recall relevant results from \cite[Suppl.]{Hasselblatt-Katok} which are applications of the theory of non-uniformly hyperbolic dynamics, or  Pesin theory, and we refer also to  \cite{BarreiraPesin} and references therein. Consider a $\varphi$-invariant ergodic Borel probability measure $\mu$ on $\Sigma$ and assume it is hyperbolic,  i.e. its Lyapunov exponents satisfy $\chi_1< 0< \chi_2$. See section 2 in \cite[Suppl.]{Hasselblatt-Katok} for the definition and existence of Lyapunov exponents of the measure $\mu$, and important properties. It is shown that for almost every $x \in \Sigma$ there are so-called Lyapunov charts $\psi_x: B(0,r(x)) \subset \R^2 \to M$ for $x$, $r(x)>0$, with respect to which the dynamics has uniformly hyperbolic behaviour. This can be expressed with the notion of admissible stable and unstable manifolds, defined locally in these charts, and which are locally preserved by $\varphi$. In order to do that one restricts $\psi_x$ to  $[-\kappa,\kappa]^2$ for sufficiently small $\kappa>0$ and  considers the \textit{regular rectangles} $R(x,\kappa) = \psi_x([-\kappa,\kappa]^2)$ in $M$. We denote the left boundary of $R$ by $\partial_l R$ given by $\partial_l R = \{ \psi_x(-\kappa, s) \, | \,  s \in [-\kappa, \kappa]\}$ and denote the  right, bottom and top boundaries of $R= R(x,\kappa)$  by $\partial_r R, \partial_b R$, and $\partial_t R$, respectively, which are defined analogously.  
Let $\gamma \in (0,1)$. A submanifold $W \subset R(x,\kappa)$ is called an \textit{admissible stable $(\gamma,\kappa)$-manifold near x} if  $W= \psi_x(\{(\phi(v),v)\, |\, v\in [-\kappa,\kappa]\})$, where $\phi:[-\kappa,\kappa] \to [-\kappa,\kappa]$ is a $C^1$-map with $\phi(0) \leq \kappa/4$ and $|D\phi| \leq \gamma$. Analogously   $W \subset R(x,\kappa)$ is called an \textit{admissible unstable  $(\gamma,\kappa)$-manifold near x} if 
$W= \psi_x(\{(u,\phi(u))\, |\, v\in [-\kappa,\kappa]\}$, where $\phi:[-\kappa,\kappa] \to [-\kappa,\kappa]$ is a $C^1$-map with $\phi(0) \leq \kappa/4$ and $|D\phi| \leq \gamma$.
Admissible stable and unstable $(\gamma,\kappa)$-manifolds near $x$  intersects in exactly one point in $R(x,\kappa)$ with angle bounded away from zero and so they endow $R(x,\kappa)$ with a product structure. Moreover one defines admissible unstable (stable)  rectangles  as sets bounded by two admissible unstable (stable) manifolds. That is, an admissible stable $(\gamma,\kappa)$-rectangle in $R(x,\kappa)$ is a set of the form $V=\psi_x(V')$, $V'= \{(u,v) \in [-\kappa,\kappa]^2 \, | \, u = (1-\tau) \phi_1(v) + \tau \phi_2(v), \,  \tau \in [0,1]\}$, where the left boundary $\partial_l V  := \psi_x(\{(\phi_1(v),v)\, |\, v\in [-\kappa,\kappa]\}$ and the right boundary $\partial_r V:=$ $\psi_x(\{(\phi_2(v),v)\, |\, v\in [-\kappa,\kappa]\}$ are two admissible unstable $(\gamma,\kappa)$ manifolds near $x$ for which  $\phi_1(v) < \phi_2(v)$ for all $v\in [-\kappa, \kappa]$.
Analogously define admissible stable $(\gamma,\kappa)$-rectangles $H$ in $R(x,\kappa)$ with bottom (top) boundaries $\partial_b H$  ($\partial_t H)$.  We define $\partial_b V = \partial_b R(x,\kappa) \cap V$, $\partial_t V = \partial_t R(x,\kappa) \cap V$, and similarly $\partial_l H$ and $\partial_r H$. 
The following statement asserts the existence of rectangular covers and  hyperbolic properties  of $\phi$ on those rectangles.

\begin{prop}\cite[Theorem S.14]{Hasselblatt-Katok}\label{prop:KatokRectangles}
For every $\delta>0$ and  $\rho>0$ there is a compact set $\Lambda_{\delta}$ with $\mu(\Lambda_{\delta})>1-\delta$, and constants $\beta>0$, $\kappa>0$, $\gamma\in (0,1)$  and regular rectangles $R(x_1) = R(x_1,\kappa), R(x_2) = R(x_2,\delta), \ldots, R(x_t)=  R(x_t,\kappa)$, $t\in \N$, for some $x_1, \ldots, x_t \in \Lambda_{\delta}$, such that
\begin{enumerate}
    \item $\Lambda_{\delta} \subset \bigcup_{i=1}^t B(x_i, \beta)$ with $B(x_i, \beta) \subset {\rm int}R(x_i)$ 
    \item ${\rm diam}R(x_i) \leq \rho/3$ for $i=1, \ldots, t$
    \item If $y\in \Lambda_{\delta}, \varphi^n(y) \in \Lambda_{\delta}$ for some $n>0, y \in B(x_i, \beta)$, and $f^n(x) \in B(x_j, \beta)$, then the connected component $V$ of $R(x_i)\cap f^{-n}(R(x_j))$ containing $y$ is an admissible stable $(\gamma, \kappa)$ rectangle near $x_i$  and $H= f^n(V)$ is an admissible unstable $(\gamma, \kappa)$ rectangle near $x_j$. 
    \item {\rm diam} $f^k(V) < \rho$ for all $0\leq k \leq n$. 
\end{enumerate}
\end{prop}
In $(3)$ above we have that the union $(\partial_l V \cup \partial_r V)$ map to the union of  $(\partial_l H \cup \partial_r H)$, etc. 

Assume that $h_{\mu}(\varphi)>0$. 
One can show now the following, see Theorem S.5.9 in \cite{Hasselblatt-Katok} and its proof.  

\begin{prop}\label{prop:KatokHorseshoe}
Let $\delta>0, \rho>0$ and choose the constants $\beta, \kappa, \gamma$, the set $\Lambda_{\delta}$ and rectangles $R(x_1), \ldots, R(x_t)$ as in Proposition \ref{prop:KatokRectangles}.  
Then for any $\epsilon_1>0$ sufficiently small there is an unbounded set $\mathcal{N} \subset \N$ such that for all $n\in \mathcal{N}$ there is  $x \in \{x_1, \ldots, x_t\}$, and $M > e^{n(h_{\mu}(\varphi) - \epsilon_1)}$ disjoint  admissible stable $(\gamma,\kappa)$-rectangles $V_1, \ldots, V_M$ near $x$ that are mapped as in $(3)$ above by $\varphi^n$ to $M$ admissible unstable $(\gamma,\kappa)$-rectangles  $H_1, \ldots, H_M$ near $x$. 

Moreover,  $$\Lambda_M =\bigcap_{k\in \Z} \left(\varphi^{kn}(V_1) \cup \cdots \cup  \varphi^{kn}(V_{M})\right)$$
is a locally maximal hyperbolic invariant set with respect to $\varphi^n$, and $\varphi^n|_{\Lambda_M}$ is topologically conjugate to a full two-sided shift in the symbolic space of $M$ symbols. In particular, $h_{\topo}(\varphi^n|_{\Lambda_M}) > h_{\mu}(\varphi) - \epsilon_1$.  
\end{prop}

Note that it follows from the variational principle, the ergodic decomposition theorem and Ruelle's inequality that one can approximate $h_{\topo}(\varphi)$  by $h_{\mu}(\varphi)$ of such measures $\mu$ as considered above, and conclusions of Proposition \ref{prop:KatokHorseshoe} imply approximation of the topological entropy $h_{\topo}(\varphi)$ by the topological entropy of hyperbolic horseshoes of  $\varphi$. 

The proof of Theorem \ref{thm:approximation} will use the above results and a variant of arguments of Franks and Handel in \cite{FranksHandel1988}. 
\begin{proof}[Proof of Theorem \ref{thm:approximation}]
Let $\varepsilon>0$ and let $\mu$ be an invariant ergodic hyperbolic measure with $h_{\mu}(\varphi) > h - \epsilon/3$. Choose $\varepsilon_1 = \varepsilon/3$. Let now  $\mathcal{N}\subset \N$  as in Proposition \ref{prop:KatokHorseshoe}. 
Choose some $n\in \mathcal{N}$ with  $\frac{ \log(20)}{n} < \varepsilon_1$. Let  $x \in \{x_1, \ldots, x_t\}$ and $M$, $V_1, \ldots, V_M$, $H_1, \ldots H_M \subset R(x)$ as in Proposition \ref{prop:KatokHorseshoe}.
We say that $H_i$ is positively oriented if $\varphi^n(\partial_t V_i) = \partial_t H_i$ and negatively oriented if $\varphi^n(\partial_t V_i) = \partial_b H_i$. We say that $H_i$ lies below $H_j$ in $R(x)$, $i\neq j$, if every path in $R(x)$  from $\partial_b R(x)$ to $\partial_t R(x)$ intersects first $H_i$. 


For convenience we restrict to a subset $V'_1, \ldots, V'_m$ of the admissible stable rectangles $V_1, \ldots, V_M$ as well as a subset $H'_1, \ldots, H'_m$ of the admissible unstable rectangles  $H_1, \ldots H_M$, both of size $m\geq M/10$ such that

\begin{itemize}
\item $\varphi^n(V'_i) = H'_i$ for $i=1, \ldots, m$.  
\item All rectangles $H'_1, \ldots, H'_m$ are either all positively oriented or all negatively oriented. 
\item Either $H'_1$ is below $H'_2, \ldots, H'_{m}$ and $H'_m$ is above $H'_1, \ldots, H'_{m-1}$, or 
 $H'_1$ is above $H'_2, \ldots, H'_{m}$ and $H'_m$ is below $H'_1, \ldots, H'_{m-1}$. 
\end{itemize}

That we can ensure the last condition, follows from the following observation.  
\begin{lem}
Let $x_1, \cdots, x_n$ be a sequence of $n$ pairwise different natural numbers. Then it has a subsequence of length $\geq \frac{n}{5}$ such that all elements of the subsequence lie in the closed interval $I$ whose boundary is given by the first and the last element of that subsequence.
\end{lem}
\begin{proof}
Let $i_0, i_1$ such that $x_{i_0} = \min \{ x_i \, | \, i = 1, \ldots, n \}$ and 
$x_{i_1} = \max \{ x_i \, | \, i = 1, \ldots, n\}$. We assume that  $i_0 < i_1$, the other case is analogous.  
If $i_1 -i_0 + 1 \geq \frac{n}{5}$ the subsequence $x_{i_0}, x_{i_0 + 1}, \ldots, x_{i_1}$ is a subsequence that satisfies the properties claimed in the  lemma. Otherwise,  $i_0>  \frac{2n}{5}$ or $n-i_1 > \frac{2n}{5}$. 
Note that any element of the sequence lies in the interval $[x_1,x_{i_1}]$ or in the interval $[x_{i_0}, x_1]$. So if $i_0>\frac{2n}{5}$, $[x_1,x_{i_1}]$ or $[x_{i_0}, x_1]$ contains $> \frac{n}{5}$ elements from $\{x_1, \ldots, x_{i_0}\}$. Similarly,  if $n-i_1 > \frac{2n}{5}$, the interval $[x_n, x_{i_1}]$ or the interval $[x_{i_0}, x_n]$ contains $> \frac{n}{5}$ elements from $\{x_{i_1}, \ldots, x_{n}\}$. 
In all situations these elements form a subsequence as claimed.
\end{proof}
 
We now proceed with the argument for the case that all $H'_i$ are positively oriented, $H'_1$ is below $H'_2, \ldots, H'_{m}$ and $H'_m$ is above $H'_1, \ldots, H'_{m-1}$. The other cases are treated analogously. By abuse of notation we  drop the symbol  $'$, and denote these rectangles by $V_1, \ldots, V_m$ and $H_1, \ldots H_m$. 
Let $\Lambda_n = \bigcap_{k\in \Z} (\varphi^{kn}(V_1) \cup \ldots \cup \varphi^{kn}(V_{m}))$, and $\Sigma_m$ the set of  bi-infinite sequences of $m$ symbols. 
The map $\theta: \Lambda_n \to \Sigma_m$ defined as $x \mapsto (\theta(x)_k)_{k\in\Z}$, with $\theta_k(x) = j$ if $\varphi^{kn}(x) \in V_j$, provides a conjugation of $\varphi^n$ with the two-sided shift $\sigma_m$ on $\Sigma_m$.  

For any $a_0, \ldots a_k \in \{1, \ldots, m\}$ we denote by $(a_0\cdots a_k)^{\infty}$ the $k$-periodic orbit of $\sigma_n$ in  $\Sigma_m$ that is given by an infinite repetition of $a_0\cdots a_k$ in positive and negative direction. 
If we want to refer to a periodic point of that orbit, we say that $(a_i a_{i+1} \cdots a_k a_1 \cdots a_{i-1})^{\infty}$ is the point that has $a_i$ at the $0$th position. By abuse of notation we denote the  periodic orbits or points of $\varphi^n$ in $\Lambda_n$ that correspond to those in $\Sigma_m$ via  $\theta$ also by the symbols $(a_0\cdots a_k)^{\infty}$ etc. 

Consider the periodic orbit  $\mathcal{Q}$ of $\varphi^n$ given by 
\begin{align*}
\mathcal{Q} = \left(\prod_{j=2}^{m-1}(mj1j1jmj)\right)^{\infty} 
\end{align*}

We write also $\mathcal{Q}$ as the orbit 
\begin{equation}\label{qqqq}
\left(q_1^2, q_2^2, \ldots, q_8^2, q_1^3, q_2^3 \ldots, q_8^3, \ldots, q_8^{m-2}, q_1^{m-1}, q_2^{m-1}, \ldots, q_8^{m-1}\right),
\end{equation}
where 
the periodic point $q_1^2$ is given in the symbolic expression by 
$$q_1^2 = \left(\left(\prod_{j=2}^{m-2}(1j1jmjm(j+1))\right) 1(m-1)1(m-1)m(m-1)m2\right)^{\infty},$$and the other periodic points $q_l^j$, $j\in \{2, \ldots, m-1\}$, $l\in \{1, \ldots, 8\}$, accordingly via cyclic permutations of the symbols.  
We note the following, which will be relevant below: For any $j \in \{2, \ldots, m-1\}$
\begin{align}\label{1jm}
\begin{split} &q_1^j \in V_1 \cap H_j \cap \varphi^n(H_m), \quad 
q_3^j\in V_1 \cap H_j \cap \varphi^n(H_1) \\ &
q_5^j\in V_m \cap H_j \cap \varphi^n(H_1), \quad q_7^j\in V_m \cap H_j \cap \varphi^n(H_m),
\end{split}
\end{align}
and 
\begin{equation}\label{11mm}
q_2^j,q_4^j, q_6^j, q_8^j \in H_1 \cup H_m.
\end{equation}


We show below that
\begin{prop}\label{prop:entropybound}
$\Gamma_{\pi_1}([\varphi^n, \mathcal{Q}]) \geq \log(m-2)$. 
\end{prop}
From the Proposition we obtain
\begin{align}
\begin{split}
\Gamma_{\pi_1}([\varphi^n, \mathcal{Q}]) &\geq  \log(m-2) \geq \log(M/20) \\ &\geq \log\left(\frac{e^{n(h-2\epsilon_1)}}{20}\right)  \geq n(h-2\epsilon_1) - \log(20) \\ &> n(h-\varepsilon).
\end{split}
\end{align}
The hyperbolic periodic orbit $\mathcal{Q}$ of $\varphi^n$ of period $8(m-2)$ defines a hyperbolic periodic orbit $\mathcal{P}$ of $\varphi$ of period $8n(m-2)$, with $\overline{\mathcal{Q}} \subset \overline{\mathcal{P}}$ for the associated sets of $\mathcal{Q}$ and $\mathcal{P}$. 
Then
$\Gamma_{\pi_1}([\varphi,\overline{\mathcal{P}}]) =  \frac{1}{n}\Gamma_{\pi_1}([\varphi^n,\overline{\mathcal{P}}]) \geq \frac{1}{n} \Gamma_{\pi_1}([\varphi^n,  \overline{\mathcal{Q}}]) > h-\varepsilon$, which will finish the proof of Theorem \ref{thm:approximation}.
\end{proof}

\begin{proof}[Proof of Proposition \ref{prop:entropybound}]
The proof is a variation of an argument from  \cite{FranksHandel1988}. 

Identify the universal covering of $\Sigma \setminus \overline{\mathcal{Q}}$ with the Poincaré disk $\Hy$. This gives us a choice of hyperbolic metric on $\Sigma \setminus \overline{\mathcal{Q}}$. Every proper arc in $\Sigma \setminus \overline{\mathcal{Q}}$, i.e. an arc $\alpha$ that lies in $\Sigma\setminus \overline{\mathcal{Q}}$ up to its endpoints and whose endpoints lie in $\overline{\mathcal{Q}}$, defines uniquely a geodesic that traces a proper arc homotopic to $\alpha$. Also, any closed curve $\gamma$ in $\Sigma\setminus \overline{\mathcal{Q}}$ defines a closed geodesic in $\Sigma \setminus \overline{\mathcal{Q}}$ freely homotopic to $\gamma$. 

For any $\gamma_1$ and $\gamma_2$ that are either proper arcs or closed curves,  denote by  $I(\gamma_1, \gamma_2)$ their geometric intersection number, i.e. the minimal number of intersections of curves $\gamma'_1$ and $\gamma'_2$, where $\gamma'_1$ and $\gamma'_2$ are  homotopic to $\gamma_1$ and $\gamma_2$, respectively. Similarly, we can define $I(\cdot, \cdot)$ for families of proper arcs resp. closed curves.  We will below consider a proper arc $\tau$ with endpoints in $\overline{\mathcal{Q}}$ and a collection of proper geodesic arcs $\mathcal{E}$ such that for all $k\in \N$,
\begin{align}\label{intersection}
I(\varphi^{kn}(\tau), \mathcal{E}) \geq (m-2)^k.
\end{align}
From this the  Proposition will follow as in \cite{FranksHandel1988}. 



Let us introduce some terminology. 
Let $S_{\infty}$ the boundary at infinity of $\Hy$. Every geodesic in $\Hy$ has two endpoints in $S_{\infty}$.
We say that a rectangle $Q=(\alpha_1, \alpha_2, \alpha_3, \alpha_4)$ in $\Hy \cup S_{\infty}$ with vertices in $S_{\infty}$ and adjacent edges $\alpha_1, \ldots, \alpha_4$  is a geodesic rectangle in $\Hy$ if $\alpha_1, \ldots, \alpha_4$ are geodesics. We keep the information of the ordering of the edges and  say that the left vertical edge of $Q$ is $\alpha_1$, the right vertical edge of $Q$ is $\alpha_3$ and the horizontal edges are $\alpha_2$ and $\alpha_4$. 
We say that a geodesic $\beta$ in $\Hy$ intersects $Q$ horizontally from left to right if $\beta$ intersects each of its vertical edges, first the left and then the right vertical edge. We say that  geodesic  rectangle $Q$  intersects $Q'$ horizontally from left to right if both horizontal edges of $Q$, parametrized from the left vertical edges to the right vertical edges of $Q$, intersect $Q'$ horizontally from left to right.
We say that $Q\prec Q'$, if there is a geodesic in $\Hy$ that first intersects $Q$ and then $Q'$ horizontally from left to right.  

Let $j \in \{2, \ldots, m-1\}$. 
Writing $\mathcal{Q}$ as in \eqref{qqqq}, we let $e_j$ be a simple proper arc from $q_1^j$ to $q_3^j$  in $H_j \cap V_1$, $f_j$ a simple proper arc 
from $q_3^j$ to $q_5^j$ in $H_j \cap \varphi^n(H_1)$, $g_j$ a simple proper arc from  $q_5^j$ to $q^7_j$ in $ H_j\cap V_m$, and $h_j$ a simple proper arc from  $q_7^j$ to $q_1^j$ in $H_j \cap \varphi^n(H_m)$. 
 $e_j, f_j, g_j$, resp.  $h_j$, uniquely define homotopic simple geodesic arcs $\hat{e}_j, \hat{f}_j$, $\hat{g}_j$, resp. $\hat{h}_j$. 
 Let $\hat{L}_j$ be the rectangle formed by $\hat{e}_j, \hat{f}_j$, $\hat{g}_j$, and $\hat{h}_j$. By the choice of $\mathcal{Q}$, see \eqref{1jm} and \eqref{11mm}, the full rectangles $\hat{L}_J$ are up to their vertices completely contained in $\Sigma \setminus \overline{\mathcal{Q}}$, and hence can be lifted to $\Hy$. Any lift $\widetilde{L}_j$ of $\hat{L}_j$ to $\Hy$ is a geodesic rectangle $Q=(\alpha_1, \ldots, \alpha_4)$ in $\Hy$, where $\alpha_1, \alpha_2, \alpha_3, \alpha_4$ are suitable lifts of $\hat{e}_j, \hat{f}_j$, $\hat{g}_j$, $\hat{h}_j$, respectively. 
Note that any two lifts of $\hat{L}_2,\ldots, \hat{L}_{m-1}$ to $\Hy$ are pairwise disjoint.

It is a theorem of Nielsen that any lift of $\varphi^n|_{\Sigma \setminus \mathcal{P}}$ to $\Hy$ extends to a homeomorphism $F$ on $\Hy \cup S_{\infty}$, and let $F$ be such an extension. Let $Q$ be geodesic rectangle in $\Hy$, then the image $F(Q)$ defines uniquely a geodesic rectangle $[F(Q)]$ in $\Hy$. 
Note that the endpoints of  $\varphi^{n}(f_i)$ and $\varphi^{n}(h_i)$ for $i\in \{2, \ldots, m-1\}$  lie in $\overline{\mathcal{Q}} \cap (H_1 \cup H_{m})$. It is now easy to verify that for all $j\in\{2, \ldots, m-1\}$ there is an arc $e'_j$ homotopic to $e_j$ and an arc $g'_j$ homotopic to $g_j$  such that for all $i\in \{2, \ldots, m-1\}$, $\varphi^n(f_i)$ intersects both $e'_j$ and $g'_j$ in exactly one point, and the order of those intersections coincides with the orientation of $f_i$ and $h_i$ if parametrized from $e_i$ to $g_i$.
 It follows that for any lift $Q$  of $\hat{L}_j$ there are lifts $Q_i$ of $\hat{L}_i$, $i=2, \ldots, m-1$, such that the  geodesic rectangle $[F(Q)]$ intersects $Q_2, \ldots, Q_{m-1}$ all horizontally from left to right. 

Let $q, q'$ be distinct points in $\overline{\mathcal{Q}}$ and $\sigma$  a geodesic arc in $\Sigma \setminus \overline{\mathcal{Q}}$ tracing a proper arc from $q$ to $q'$. Take a lift $\widetilde{\sigma}$ of $\sigma$ to $\Hy$. Let $D_j(\sigma)$ be the number of geodesic rectangles  $Q=(\alpha_1, \ldots, \alpha_4)$ that arise as lifts of rectangles $\hat{L}_j$ and which $\sigma$ intersects horizontally from left to right. Let $D(\sigma) = \sum_{j=2}^{m-1} D_j(\sigma)$. These numbers obviously do not depend on the choice of lift of  $\sigma$. Let $\sigma'$ be the geodesic arc in $\Sigma \setminus \mathcal{P}$ that traces a proper arc homotopic to $\varphi^n(\sigma)$. 
We claim that
\begin{align}\label{Destimate}
D(\sigma') \geq (m-2)D(\sigma).
\end{align}
This can be seen as follows. 
Take a lift $\widetilde{\sigma}$ of $\sigma$ to $\Hy$.
We can order the lifts of the $\hat{L}_j$, $j\in \{2, \ldots, m-1\}$ that $\widetilde{\sigma}$ intersects horizontally from left to right by $\prec$, i.e. these are geodesic rectangles $J_1, \ldots J_k$, with $k=D(\sigma)$, and $J_i \prec J_j$ if $i<j$. Since $F|_{S_{\infty}}$ keeps the cyclic order of points at infinity of geodesics, it follows that the geodesic arc in $\Hy$ with the same endpoints as $F(\widetilde\sigma)$ has to intersect the geodesic rectangles $[F(J_1)]$, $[F(J_2)], \ldots [F(J_k)]$ horizontally from left to right, and $[F(J_i)] \prec [F(J_j)]$ if $i< j$. 
Since each of the rectangles $[F(J_i)]$ intersects horizontally from left to right suitable lifts $Q_2, \ldots, Q_{m-1}$ of $\hat{L}_2, \ldots, \hat{L}_{m-1}$, \eqref{Destimate} follows.  

Now choose $q$ and $q'$ two distinct points in $\overline{\mathcal{Q}} \cap (H_1 \cup H_m)$ and $\tau$ an arc from $q$ to $q'$ with $D(\tau) = 1$. Furthermore, let $\mathcal{E} = \bigcup_{j=2}^{m-1} \hat{e}_j$. 
From \eqref{Destimate} the estimate \eqref{intersection} follows, that is 
\begin{align*}
I(\varphi^{kn}(\tau), \mathcal{E}) \geq (m-2)^k.
\end{align*}
Choose now as in \cite{FranksHandel1988} for each $k\in \N$ a closed curve  $\gamma_k$ in $\Sigma\setminus \overline{\mathcal{Q}}$ as the frontier of a small contour of the union of the geodesic representative of $\varphi^{kn}(\tau)$ and its two endpoints. 
For those $k \in \N$ for which $\varphi^{kn}(q), \varphi^{kn}(q') \in H_1 \cup H_m$, the curve $\gamma_k$ and $\mathcal{E}$ have $2I(\varphi^{kn}(\tau), \mathcal{E})$ many intersections. Moreover there is no bigon formed by the union of $\mathcal{E}$ with $\gamma_k$, and hence the number of those intersections is minimal among all curves homotopic to $\gamma_k$. Since $\varphi^{kn}(\gamma_0)$ is freely homotopic to $\gamma_k$ it follows that $\Gamma_{\pi_1}([\varphi^n, \mathcal{Q}]) \geq \log(m-2)$.

\end{proof}

\bibliographystyle{plain}
\bibliography{ReferencesS5}

\end{document}